\documentclass[12pt]{amsart} 

\usepackage{amsmath, amssymb, amsfonts, amsbsy, amsthm, latexsym, stmaryrd, mathtools, fullpage, enumerate, ytableau, shuffle, multirow, mathdots, enumerate, tikz, hyperref, float}

\usepackage{comment}
\usepackage{mathtools}

\allowdisplaybreaks

\newtheorem{theorem}{Theorem}[section]
\newtheorem{lemma}[theorem]{Lemma}
\newtheorem{proposition}[theorem]{Proposition}
\newtheorem{corollary}[theorem]{Corollary}
\theoremstyle{definition}

\newtheorem{remark}[theorem]{Remark}
\newtheorem{definition}[theorem]{Definition}
\newtheorem*{definition*}{Definition}
\newtheorem*{theorem*}{Theorem}
\newtheorem*{proposition*}{Proposition}
\newtheorem*{corollary*}{Corollary}

\newcommand{\Z}{\mathbb{Z}}
\newcommand{\cal}[1]{\mathcal{#1}}
\newcommand{\hatcal}[1]{\widehat{\mathcal{#1}}}
\renewcommand{\hat}[1]{\widehat{#1}}

\usepackage[colorinlistoftodos]{todonotes}

\author{Yibo Gao}
\address{Department of Mathematics, Massachusetts Institute of Technology, \mbox{Cambridge, MA 02139}}
\email{\href{mailto:gaoyibo@mit.edu}{{\tt gaoyibo@mit.edu}}}
\author{Zhaoqi Li}
\address{Department of Mathematics, Statistics, and Computer Science, Macalester College, \newline St Paul, MN 55105}
\email{\href{mailto:zli@macalester.edu}{{\tt zli@macalester.edu}}}
\author{Thuy-Duong Vuong}
\address{Department of Electrical Engineering and Computer Science, Massachusetts Institute of Technology, \mbox{Cambridge, MA 02139}}
\email{\href{mailto:dvuong@mit.edu}{{\tt dvuong@mit.edu}}}
\author{Lisa Yang}
\address{Department of Mathematics, Massachusetts Institute of Technology, \mbox{Cambridge, MA 02139}}
\email{\href{mailto:lisayang@mit.edu}{{\tt lisayang@mit.edu}}}

\begin{document}
\title{Toric Mutations in the dP$_2$ Quiver and Subgraphs of the dP$_2$ Brane Tiling}
\date{\today}

\begin{abstract}
Brane tilings are infinite, bipartite, periodic, planar graphs that are dual to quivers. In this paper, we examine the del Pezzo 2 (dP$_2$) quiver and its brane tiling, which arise from the physics literature, in terms of toric mutations on its corresponding cluster. Specifically, we give explicit formulas for all cluster variables generated by toric mutation sequences. Moreover, for each such variable, we associate a subgraph of the dP$_2$ brane tiling to it such that its weight matches the variable.
\end{abstract}
\maketitle

\section{Introduction}
Cluster algebras are a class of commutative rings generated by cluster variables, which are partitioned into sets called clusters. Given an initial seed, an operation known as seed mutation can be applied iteratively to generate all cluster variables. The concept of cluster algebras was first introduced by Fomin and Zelevinsky \cite{fomin2002cluster} as a tool to study total positivity and dual canonical bases in Lie theory. They have rich applications in different branches of mathematics, including algebraic combinatorics, tropical geometry, Teichmuller theory, and representation theory. 

It is common to picture a cluster as a quiver with a cluster variable on each vertex. Some special quivers have planar duals, known as brane tilings, which are doubly-periodic, bipartite, planar graphs. The notion of brane tilings was first introduced in theoretical physics \cite{franco2006brane}. For such quivers, combinatorial interpretations of the cluster variables have been obtained by associating a subgraph of the brane tiling to each cluster variable such that the Laurent polynomial of the cluster variable is recoverable from a weighting scheme applied to the subgraph.  See  \cite{musiker2010cluster}, \cite{musiker2011graph}, and \cite{lee2013combinatorial}. In particular, the quiver and brane tiling of the third del Pezzo (dP$_3$) surface \cite{hanany2012brane} has been studied widely by \cite{zhangcluster}, \cite{leoni2014aztec}, and \cite{lai2015beyond}. 
In this paper, we generalize the techniques utilized in these papers and focus on the second del Pezzo (dP$_2$) surface. Specifically, we classify all cluster variables generated by toric mutations and give combinatorial interpretations for them. 

In Section~\ref{sec:prelim}, we start with background material on quivers and cluster mutations, and then introduce our main objects of the paper, the dP$_2$ quiver and its corresponding brane tiling. Our contribution begins in Section~\ref{sec:rho}, where we introduce $\rho$-mutations (Definition~\ref{def:rho}) so that all toric mutations can be written in a specific form of $\rho$-mutation sequences (Theorem~\ref{thm:rho}). From there, we are able to explicitly write down the cluster variables arise from these $\rho$-mutation sequences (Theorem~\ref{thm:rhocv}) and classify all of them in a simple format (Corollary~\ref{cor:rhocv}). In the second half of the paper, we give combinatorial interpretations for these cluster variables by associating subgraphs of the dP$_2$ brane tiling to them, such that the ``weights" of the subgraphs equal the cluster variables. Our main theorem (Theorem~\ref{thm:contours}) provides each cluster variable with a contour, where the weighting scheme is described in Section~\ref{sec:weight} and the rules to obtain a subgraph from a contour is given in Section~\ref{sec:contour}. The proof of our main theorem is shown in Section~\ref{sec:proof}, which is done by induction and case work. We provide detailed description of our proof techniques and sample proofs for a few cases. The rest of the cases are proved in the same way with necessary data given in Appendix~\ref{sec:appendix}. We finish with a discussion regarding a similar quiver and brane tiling, which also generate the Laurent polynomials of the Somos-5 sequence, studied by David Speyer \cite{speyer2007perfect} in Section~\ref{sec:speyer}.

\section{Preliminaries}\label{sec:prelim}
\subsection{Quiver and Cluster Mutations}
\begin{definition}[Quiver and Cluster]
A \textit{quiver} is a finite directed graph $Q$ with a set of vertices $V$ and a set of edges $E$. We can associate a cluster variable $x_i$ to the vertex labeled $i$. The \textit{cluster} is the ordered set of cluster variables $\{x_1,\ldots,x_n\}$ at each vertex, assuming $|V|=n$.  For a cluster $S = \{x_1, \ldots, x_n\}$, let $S[i]$ refer to the $i$th cluster variable.
\end{definition}
In this paper, we allow quivers to have multiple edges connecting two vertices but there can be no 2-cycles or 1-cycles (loops).

\begin{definition}[Quiver Mutation]
Mutating at a vertex $i$ in $Q$ is denoted by $\mu_i$ and corresponds to the following actions on the quiver:
\begin{itemize}
\item For every 2-path through $i$ (e.g. $j\rightarrow i\rightarrow k$), add an edge from $j$ to $k$.
\item Reverse the directions of the arrows incident to $i$.
\item Delete any 2-cycles created from the previous two steps.
\end{itemize}
Correspondingly, the cluster variable at vertex $i$ is updated and all other cluster variables stay the same. The update follows this binomial exchange relation:
$$x_i'x_i=\prod_{i\rightarrow j\text{ in }Q}x_j^{a_{i\rightarrow j}}+\prod_{j\rightarrow i\text{ in }Q}x_j^{a_{j\rightarrow i}},$$
where $x_i'$ is the new cluster variable at vertex $i$ and $a_{i\rightarrow j}$ is the number of edges from $i$ to $j$.
\end{definition}
The binomial exchange relation replaces $S[i]$ by the new cluster variable $x_i'$.  We denote this replacement by 
$$S[i] \leftarrow \frac{\prod\limits_{i\rightarrow j\text{ in }Q}x_j^{a_{i\rightarrow j}}+\prod\limits_{j\rightarrow i\text{ in }Q}x_j^{a_{j\rightarrow i}}}{x_i}.$$

\subsection{The Del Pezzo 2 Quiver and its Brane Tiling}
In this paper, we study a special quiver associated to the second del Pezzo surface (dP$_2$) \cite{beasley2001toric} and its brane tiling, as seen in Figure \ref{fig:dp2}.
\begin{figure}[h!]
\includegraphics[scale=0.3]{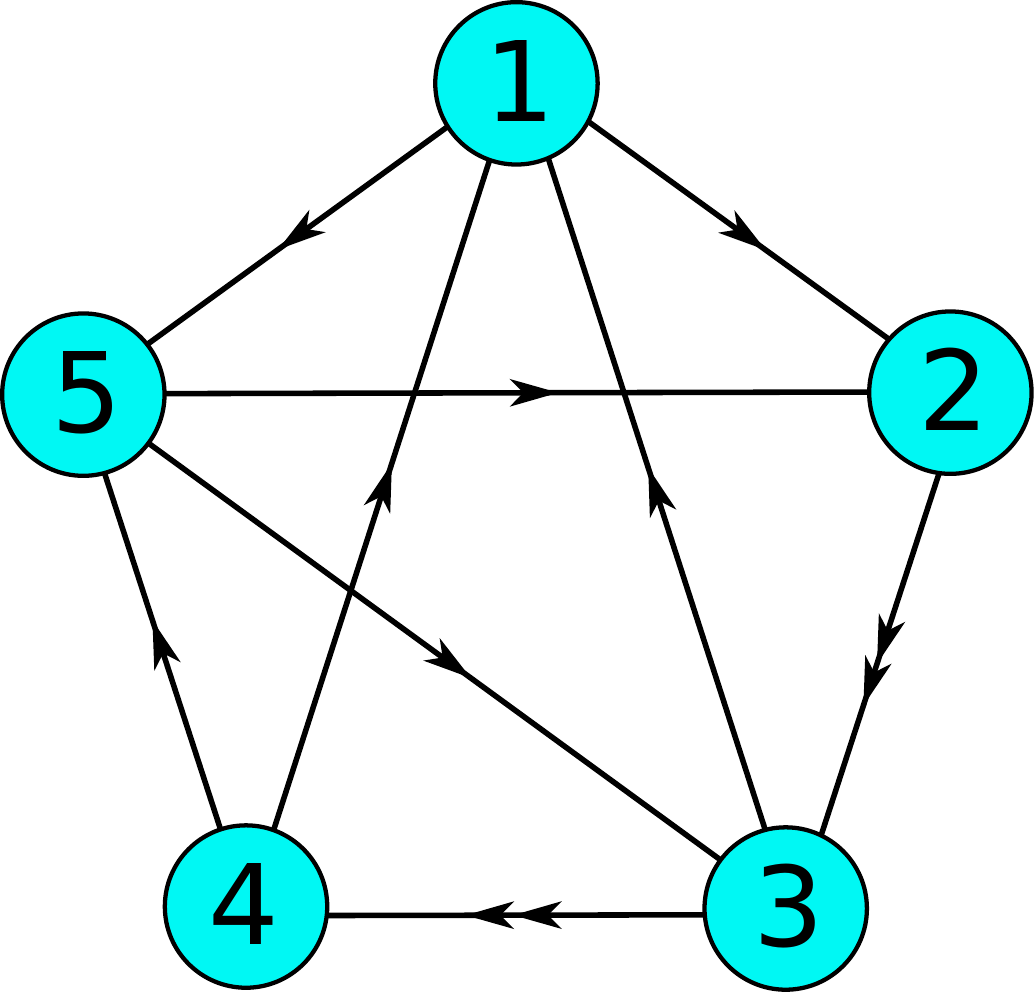}
\qquad
\includegraphics[scale=0.3]{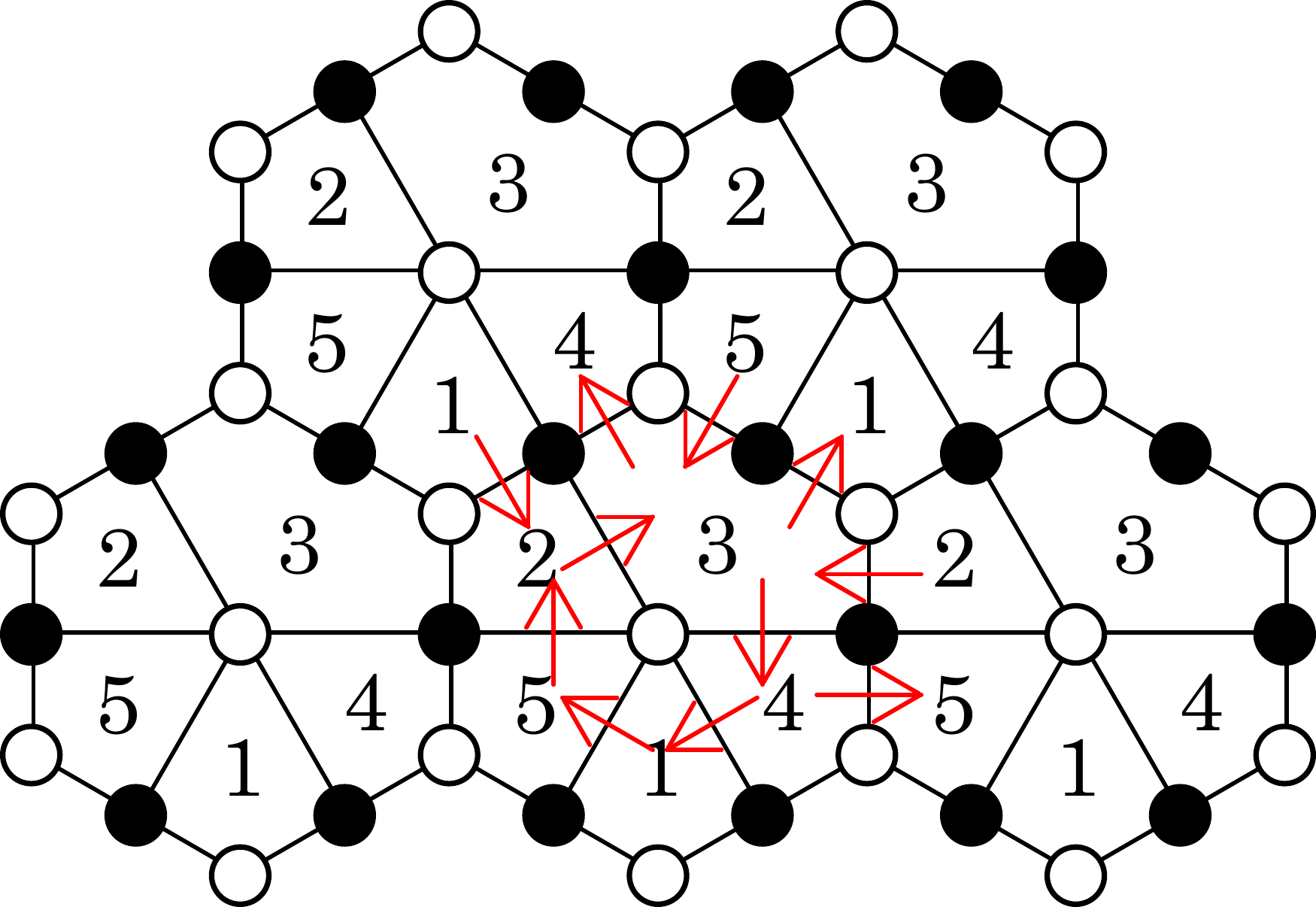}
\caption{dP$_2$ quiver $Q$ and its associated brane tiling $\cal{T}$ (Figure 30 of \cite{hanany2012brane})}
\label{fig:dp2}
\end{figure}

To get from a brane tiling to the corresponding quiver, we look at each edge $e$ up to translation, noticing that any brane tiling is periodic, bipartite and planar. Assume that $e$ borders block $i$ and $j$ such that as we go across from block $i$ to block $j$, the black end point of $e$ is on the left and the white end point of $e$ is on the right. For this edge $e$, we add an edge in the quiver that goes from $i$ to $j$. The red arrows in Figure \ref{fig:dp2} show this process.

We use $Q$ to denote the dP$_2$ quiver and $\cal{T}$ to denote its associated brane tiling
.
\subsection{Toric Mutation and Two Models of Quivers}
\begin{definition}[Toric Vertex and Toric Mutation]
We say that a vertex in a quiver is \textit{toric} if it has in-degree 2 and out-degree 2. A \textit{toric mutation} is a cluster mutation at a toric vertex.
\end{definition}
\begin{definition}[Model]
We say that two quivers $Q_1$ and $Q_2$ are of the same \textit{model} if they are isomorphic as directed graphs (there exists a bijection between their vertices that preserves edges), or if $Q_1$ is isomorphic as graph to $Q_2$ with all edges in $Q_2$ reversed.
\end{definition}
It is easy to check that the dP$_2$ quiver $Q$ has two models that can be reached from the original quiver by toric mutations. Use \textbf{Model 1} to denote the original quiver $Q$ and \textbf{Model 2} to denote the quiver obtained from $Q$ by mutating at vertex 2. Figure \ref{fig:model} shows these two models. As a side note, the word ``model'' is also seen as ``phase'' in the literature \cite{hanany2012brane}.
\begin{figure}[h!]
\includegraphics[scale=0.3]{model1.pdf}
\qquad
\includegraphics[scale=0.3]{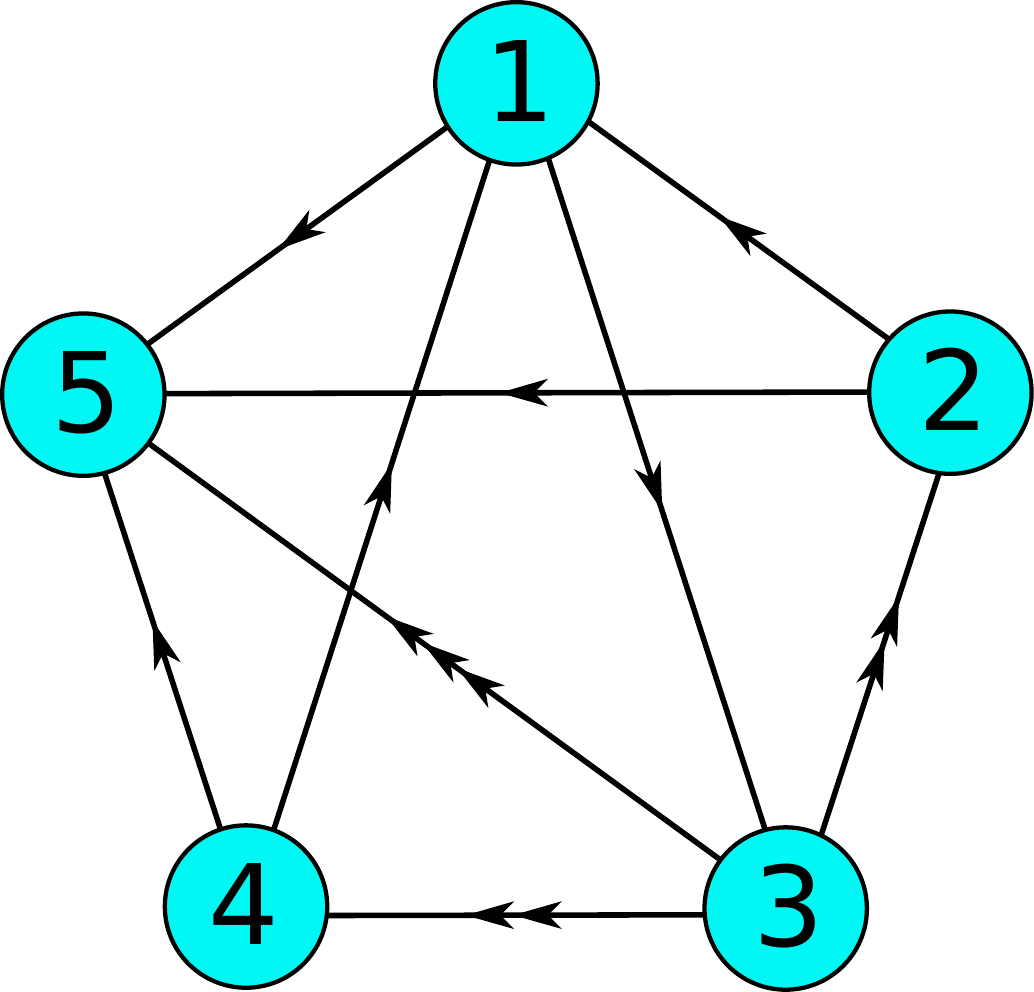}
\caption{model 1 and model 2 of the dP$_2$ quiver (Figure 30 and 31 of \cite{hanany2012brane})}
\label{fig:model}
\end{figure}

Transitions between these two models are shown in Figure \ref{fig:modeladj}.
\begin{figure}[h!]
\includegraphics[scale=0.4]{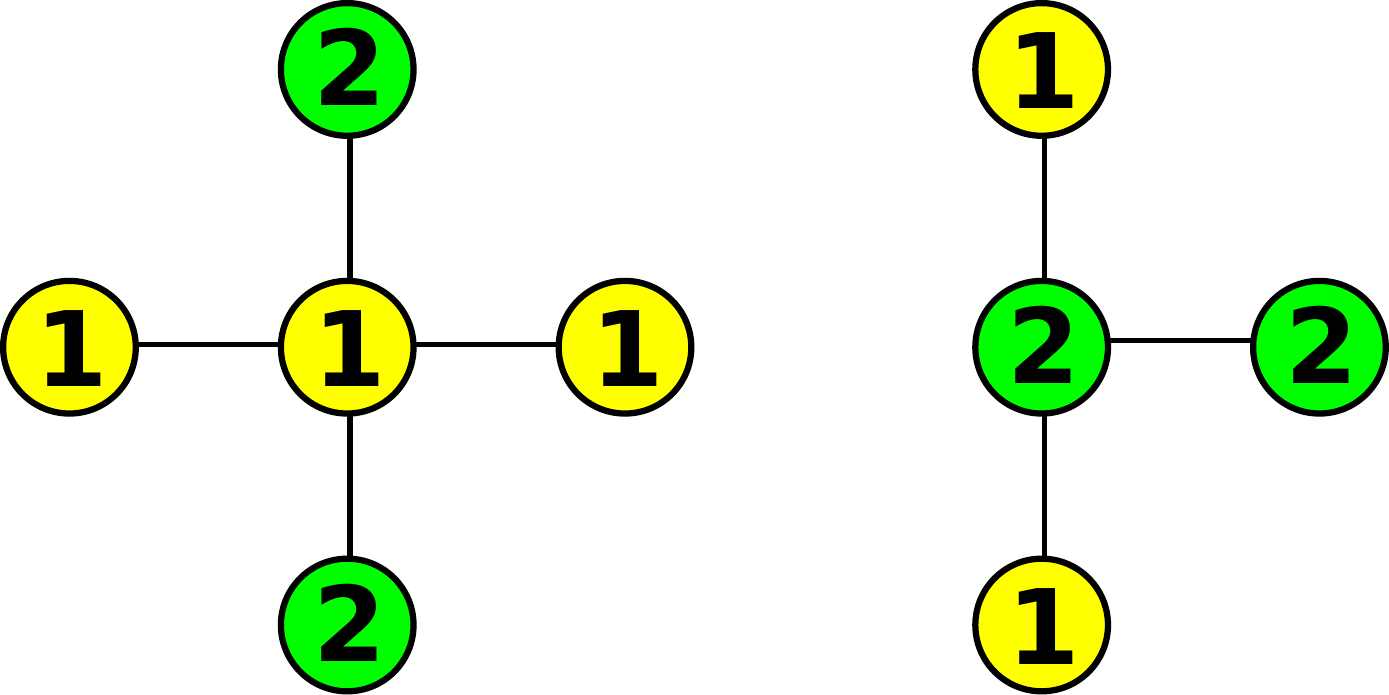}
\caption{Adjacency between different models (Figure 18 of \cite{eager2012colored})}
\label{fig:modeladj}
\end{figure}

\section{Classification of Toric Mutation Sequences}\label{sec:rho}

\begin{definition}\label{def:rho} [$\rho$-mutation sequences]
We define the following operation sequences consisting of mutations and permutations, where concatenation of operations is done from left to right. A permutation permutes the vertices and their associated cluster variables accordingly. 
$$\rho_1 = \mu_1\circ(54321),\quad
\rho_2 = \mu_5\circ(12345),\quad\rho_3 = \mu_2\circ\mu_4\circ(24),$$
$$\rho_4 = \mu_2\circ\mu_1\circ\mu_4\circ(531),\quad\rho_5 = \mu_4\circ\mu_5\circ\mu_2\circ(351),$$
$$\rho_6=\mu_2\circ\mu_1\circ\mu_2\circ(531)(24),\quad\rho_7=\mu_4\circ\mu_5\circ\mu_4\circ(135)(24).$$
We call each $\rho_i$ a $\rho$-mutation and any concatenation of $\rho_i$'s a $\rho$-mutation sequence.
\end{definition}
As a side note, it is technically more correct to name ``$\rho$-mutation'' as ``$\rho$-operation''. However, we follow conventions set in \cite{leoni2014aztec} and \cite{lai2015beyond} and thus choose the name ``$\rho$-mutation''.

These $\rho$-mutations all fix the quiver (but not the cluster variables), that is, $\rho_i(Q)=Q$, for $i=1,\ldots,7$. Notice that in the original quiver $Q$, there are no edges connecting vertex 2 and 4. This means mutation at 2 and mutation at 4 commute, so $\rho_3$ can also be written as $\rho_3=\mu_4\circ\mu_2\circ(24).$

It is easy to construct Figure~\ref{fig:model1-1}, which shows all possible toric mutation sequences that start from the original dP$_2$ quiver and return to model 1 the first time, from Figure~\ref{fig:modeladj}. In this way, it is clear that combinations of these seven $\rho$-mutations give us all possible toric mutation sequences that start in model 1 and end in model 1 up to a permutation of vertices.
\begin{figure}[h!]
\includegraphics[scale=0.4]{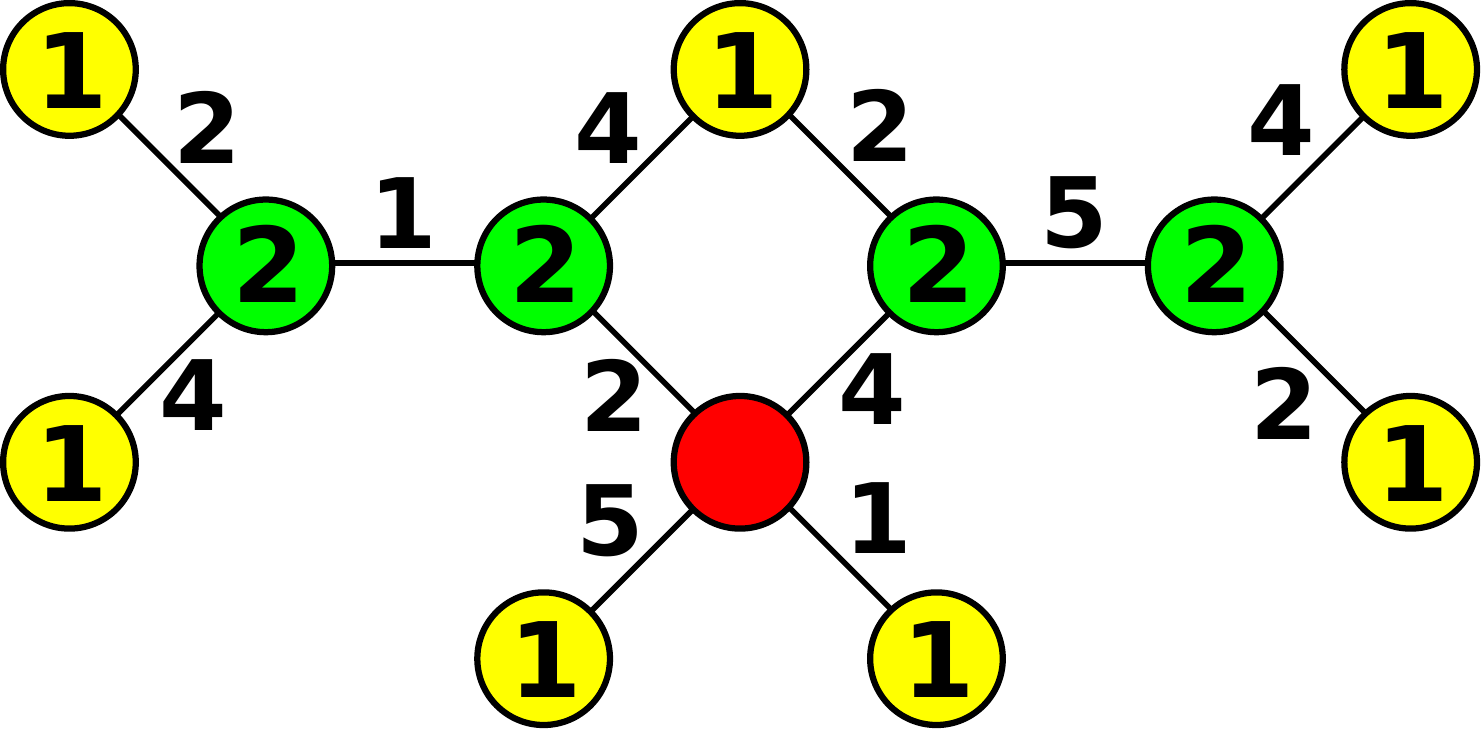}
\caption{All possible toric mutation sequences that start from model 1 and return to model 1 the first time. The red circle represents the initial quiver $Q$. Numbers on edges represent the vertices mutated.}
\label{fig:model1-1}
\end{figure}

\begin{proposition} [Relations of $\rho$-mutations]
\label{prop:rho}
\begin{align*}
&\rho_4\{x_1,x_2,x_3,x_4,x_5\}=\rho_1^2\rho_3\{x_1,x_2,x_3,x_4,x_5\},\\
&\rho_5\{x_1,x_2,x_3,x_4,x_5\}=\rho_2^2\rho_3\{x_1,x_2,x_3,x_4,x_5\},\\
&\rho_6\{x_1,x_2,x_3,x_4,x_5\}=\rho_1^2\{x_1,x_2,x_3,x_4,x_5\},\\
&\rho_7\{x_1,x_2,x_3,x_4,x_5\}=\rho_2^2\{x_1,x_2,x_3,x_4,x_5\}.
\end{align*}
$$\rho_1\rho_2\{x_1,\ldots,x_5\} = \rho_2\rho_1\{x_1,\ldots,x_5\} = \rho_3^2\{x_1,\ldots,x_5\} = \{x_1,x_2,x_3,x_4,x_5\},$$
$$\rho_1^2\rho_3\{x_1,\ldots,x_5\} = \rho_3\rho_1^2\{x_1,\ldots,x_5\},\quad\rho_2^2\rho_3\{x_1,\ldots,x_5\} = \rho_3\rho_2^2\{x_1,\ldots,x_5\},$$
$$\rho_1\rho_3\rho_2\{x_1,\ldots,x_5\} = \rho_2\rho_3\rho_1\{x_1,\ldots,x_5\}.$$
\end{proposition}

Note that it suffices to define $\rho_1, \rho_2, \rho_3$ because $\rho_4$, $\rho_5$, $\rho_6$, $\rho_7$ can be written in terms of the previous three. 

\begin{theorem}\label{thm:rho}
Any toric mutation sequence in dP$_2$ quiver that starts and ends at model 1 can be written, up to a permutation of cluster variables, as
$\rho_t^k(\rho_3\rho_1)^m\rho_3^{w}$, where $k,m\in\Z_{\geq0}$, $t\in\{1,2\}$ and $w\in\{0,1\}.$
\end{theorem}
\begin{proof}
This theorem is essentially saying that all $\rho$-mutation sequences can be written in a certain form. Fix a generic $\rho$-mutation sequence. 

Since $\rho_1\rho_2=\rho_2\rho_1=\rho_3^2=1$, we can assume that this sequence does not contain consecutive $\rho_3$'s and does not contain adjacent $\rho_1$ and $\rho_2$. Therefore, we can write it as $\rho_{j_1}^{\alpha_1}\rho_3\rho_{j_2}^{\alpha_2}\rho_3\cdots\rho_{j_N}^{\alpha_N}$ with possibly a $\rho_3$ at the beginning and a $\rho_3$ at the end, where $j_i\in\{1,2\}$ and $\alpha_i\in\Z_{>0}.$ 

Notice that by Proposition \ref{prop:rho}, $\rho_1^2$ and $\rho_2^2$ commute with everything. So whenever we see two consecutive $\rho_1$'s or consecutive $\rho_2$'s, we can pull them to the front. As a result, we can further simplify this sequence as $\rho_t^n\rho_3\rho_{\ell_1}\rho_3\rho_{\ell_2}\cdots\rho_{\ell_s}$ with possibly a $\rho_3$ at the end, where $t,\ell_1,\ldots,\ell_s\in\{1,2\}$ and $n\in\Z_{\geq0}$. 

Proposition \ref{prop:rho} gives $\rho_1\rho_3\rho_2=\rho_2\rho_3\rho_1$, which means $\rho_1$ and $\rho_2$ ``commute'' with a $\rho_3$ in between. Therefore, in $\rho_3\rho_{\ell_1}\rho_3\rho_{\ell_2}\cdots\rho_{\ell_s}$ (with possibly a $\rho_3$ in the end), we are able to put all $\rho_1$'s in front of $\rho_2$'s. The sequence now has the form $\rho_t^n(\rho_3\rho_1)^r(\rho_3\rho_2)^s$, with possibly a $\rho_3$ in the end. 

Take a sufficiently large $M$ and write the sequence as $\rho_t^n\rho_2^M\rho_1^M(\rho_3\rho_1)^r(\rho_3\rho_2)^s$. Since $\rho_1^2$ commute with everything, we take $\rho_1^2$ in the term $\rho_1^M$ to cancel all the $\rho_2$'s in $(\rho_3\rho_2)^s$, since $M$ is sufficiently large. Finally, we naturally merge the remaining $\rho_1$'s in the previous $\rho_1^M$ with $\rho_t^n$, $\rho_2^M$ and get $\rho_t^k(\rho_3\rho_1)^m$ with possibly a $\rho_3$ in the end, with $t\in\{1,2\}$ and $m,k\in\Z_{\geq0}$, as desired.
\end{proof}

\begin{remark}
Figure \ref{fig:walk} gives a way to visualize the $\rho$-mutation sequences as an analog of alcove walk discussed in the dP3 case \cite{lai2015beyond}. In Figure \ref{fig:walk}, each vertex corresponds to a cluster with a model 1 quiver. We can arbitrarily select one as the initial cluster. A horizontal step to the right is $\rho_1$; a horizontal step to the left is $\rho_2$; and a vertical step is $\rho_3$. 

\begin{figure}[h!]
\includegraphics[scale=1.0]{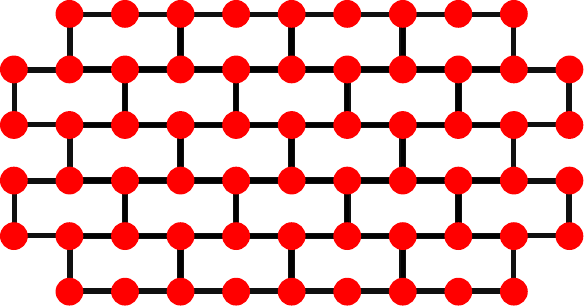}
\caption{Visualization of $\rho$-mutation sequences}
\label{fig:walk}
\end{figure}
\end{remark}

\section{Explicit Formulas for Cluster Variables}
In this section, we give explicit formulas for all cluster variables that can be generated by toric mutations for the dP$_2$ quiver. 

Suppose that the cluster variables are initialized as $\{x_1,x_2,x_3,x_4,x_5\}$.

\begin{definition}[Laurent Polynomial for Somos-5 Sequence]\label{def:x}

\

For $n\geq6$, define recursively
$$x_n:=\frac{x_{n-1}x_{n-4}+x_{n-2}x_{n-3}}{x_{n-5}}.$$

For $n\leq0$, define recursively
$$x_n:=\frac{x_{n+1}x_{n+4}+x_{n+2}x_{n+3}}{x_{n+5}}.$$
\end{definition}

\begin{remark}
For each $n\in\Z$, Definition \ref{def:x} gives us a way to define $x_n$ as a rational function in $x_1,x_2,x_3,x_4,x_5$. Moreover, the equation
\begin{equation}\label{eqn:somos5}
x_{n}x_{n+5}=x_{n+1}x_{n+4}+x_{n+2}x_{n+3}
\end{equation}
is satisfied for each $n\in\Z$. Therefore, it is clear that if we assign 1 to $x_1,\ldots,x_5$, then both $\{x_n\}_{n=1,2,\ldots}$ and $\{x_{6-n}\}_{n=1,2,\ldots}$ give us the Somos-5 sequence.
\end{remark}

\begin{definition}
Define the following constants
$$A:=\frac{x_1x_5+x_3^2}{x_2x_4},\qquad B:=\frac{x_2x_6+x_4^2}{x_3x_5}\Big(=\frac{x_1x_4^2+x_2x_3x_4+x_2^2x_5}{x_1x_3x_5}\ \Big).$$
\end{definition}
\begin{lemma}\label{lem:AB}
For each $n\in\Z$,
$$A=\frac{x_{2n-1}x_{2n+3}+x_{2n+1}^2}{x_{2n}x_{2n+2}},\qquad B=\frac{x_{2n}x_{2n+4}+x_{2n+2}^2}{x_{2n+1}x_{2n+3}}.$$
\end{lemma}
\begin{proof}
The lemma is correct when $n=1$ by definition. By an inductive argument, it suffices to show that, for each $m\in\Z$,
$$\frac{x_mx_{m+4}+x_{m+2}^2}{x_{m+1}x_{m+3}}=\frac{x_{m+2}x_{m+6}+x_{m+4}^2}{x_{m+3}x_{m+5}}.$$
According to Equation \eqref{eqn:somos5}, we have
\begin{align*}
\frac{x_{m+2}x_{m+6}+x_{m+4}^2}{x_{m+3}x_{m+5}}=&\frac{x_{m+2}\displaystyle{\frac{x_{m+2}x_{m+5}+x_{m+3}x_{m+4}}{x_{m+1}}}+x_{m+4}^2}{x_{m+3}x_{m+5}}\\
=&\frac{x_{m+2}^2}{x_{m+1}x_{m+3}}+\frac{x_{m+4}(x_{m+2}x_{m+3}+x_{m+1}x_{m+4})}{x_{m+1}x_{m+3}x_{m+5}}\\
=&\frac{x_{m+2}^2}{x_{m+1}x_{m+3}}+\frac{x_{m+4}x_mx_{m+5}}{x_{m+1}x_{m+3}x_{m+5}}\\
=&\frac{x_mx_{m+4}+x_{m+2}^2}{x_{m+1}x_{m+3}}
\end{align*}
as desired.
\end{proof}
\begin{theorem}\label{thm:rhocv}
Define $\rho_1^{k}:=\rho_2^{-k}$ for $k<0$. Define $g(s,k):=\left\lfloor\frac{s}{2}\right\rfloor\left\lfloor\frac{s+1}{2}\right\rfloor$ if $k$ is even and $g(s,k):=\left\lfloor\frac{s-1}{2}\right\rfloor\left\lfloor\frac{s}{2}\right\rfloor$ if $k$ is odd.
Then we have, for $k\in\Z$ and $s\in\Z_{\geq0}$,
\begin{align*}
\rho_1^{k}(\rho_3\rho_1)^s\{x_1,x_2,x_3,x_4,x_5\}=\{A^{g(s+1,k)}B^{g(s+1,k+1)}x_{k+s+1}&,\\
A^{g(s,k)}B^{g(s,k+1)}x_{k+s+2}&,\\
A^{g(s+1,k)}B^{g(s+1,k+1)}x_{k+s+3}&,\\
A^{g(s,k)}B^{g(s,k+1)}x_{k+s+4}&,\\
A^{g(s+1,k)}B^{g(s+1,k+1)}x_{k+s+5}&\}.
\end{align*}
\end{theorem}
\begin{proof}
We divide our toric mutation sequence into two steps: $\rho_1^{k}$ and $(\rho_3\rho_1)^{s}$. Then we use straightforward induction.

\

\noindent\textbf{Step 1:} $\rho_1^{k}\{x_1,x_2,x_3,x_4,x_5\}=\{x_{k+1},x_{k+2},x_{k+3},x_{k+4},x_{k+5}\}$ for $k\in\Z$.

This is true for $k=0$. Let us suppose that this holds for some $k\geq0$. Then 
\begin{align*}
\rho_1^{k+1}\{x_1,x_2,x_3,x_4,x_5\}=&\rho_1\{x_{k+1},x_{k+2},x_{k+3},x_{k+4},x_{k+5}\}\\
=&(54321)\big(\mu_1\{x_{k+1},x_{k+2},x_{k+3},x_{k+4},x_{k+5}\}\big)\\
=&(54321)\{\frac{x_{k+2}x_{k+5}+x_{k+3}x_{k+4}}{x_{k+1}},x_{k+2},x_{k+3},x_{k+4},x_{k+5}\}\\
=&(54321)\{x_{k+6},x_{k+2},x_{k+3},x_{k+4},x_{k+5}\}\\
=&\{x_{k+2},x_{k+3},x_{k+4},x_{k+5},x_{k+6}\}.
\end{align*}
By induction, this proves the claim for $k\geq0$. The proof for $k\leq0$ can be done the same way.

\

Before doing the next step, we first show that if $k+s$ is odd, then
\begin{equation}
\label{eqn:ksodd}
\begin{split}
g(s+1,k)=&2g(s,k)-g(s-1,k)+1\\
g(s+1,k+1)=&2g(s,k+1)-g(s-1,k+1).
\end{split}
\end{equation}

If $k$ is even, then $s$ is odd and these two equations become
\begin{align*}
&\left\lfloor\frac{s+1}{2}\right\rfloor\left\lfloor\frac{s+2}{2}\right\rfloor=2\left\lfloor\frac{s}{2}\right\rfloor\left\lfloor\frac{s+1}{2}\right\rfloor-\left\lfloor\frac{s-1}{2}\right\rfloor\left\lfloor\frac{s}{2}\right\rfloor+1\\
&\Leftrightarrow\Big(\frac{s+1}{2}\Big)^2=2\Big(\frac{s-1}{2}\Big)\Big(\frac{s+1}{2}\Big)-\Big(\frac{s-1}{2}\Big)^2+1
\end{align*}
and
\begin{align*}
&\left\lfloor\frac{s}{2}\right\rfloor\left\lfloor\frac{s+1}{2}\right\rfloor=2\left\lfloor\frac{s-1}{2}\right\rfloor\left\lfloor\frac{s}{2}\right\rfloor-\left\lfloor\frac{s-2}{2}\right\rfloor\left\lfloor\frac{s-1}{2}\right\rfloor\\
&\Leftrightarrow\Big(\frac{s-1}{2}\Big)\Big(\frac{s+1}{2}\Big)=2\Big(\frac{s-1}{2}\Big)^2-\Big(\frac{s-3}{2}\Big)\Big(\frac{s-1}{2}\Big)
\end{align*}
which are clearly correct.

If $k$ is odd, then $s$ is even and these two equations become
\begin{align*}
&\left\lfloor\frac{s}{2}\right\rfloor\left\lfloor\frac{s+1}{2}\right\rfloor=2\left\lfloor\frac{s-1}{2}\right\rfloor\left\lfloor\frac{s}{2}\right\rfloor-\left\lfloor\frac{s-1}{2}\right\rfloor\left\lfloor\frac{s-2}{2}\right\rfloor+1\\
&\Leftrightarrow\Big(\frac{s}{2}\Big)^2=2\Big(\frac{s}{2}\Big)\Big(\frac{s-2}{2}\Big)-\Big(\frac{s-2}{2}\Big)^2+1
\end{align*}
and
\begin{align*}
&\left\lfloor\frac{s+1}{2}\right\rfloor\left\lfloor\frac{s+2}{2}\right\rfloor=2\left\lfloor\frac{s}{2}\right\rfloor\left\lfloor\frac{s+1}{2}\right\rfloor-\left\lfloor\frac{s-1}{2}\right\rfloor\left\lfloor\frac{s}{2}\right\rfloor\\
&\Leftrightarrow\Big(\frac{s}{2}\Big)\Big(\frac{s+2}{2}\Big)=2\Big(\frac{s}{2}\Big)^2-\Big(\frac{s-2}{2}\Big)\Big(\frac{s}{2}\Big)
\end{align*}
which are clearly correct.

With the same argument, we can show that if $k+s$ is even, then
\begin{equation}
\label{eqn:kseven}
\begin{split}
g(s+1,k)=&2g(s,k)-g(s-1,k)\\
g(s+1,k+1)=&2g(s,k+1)-g(s-1,k+1)+1.
\end{split}
\end{equation}

\noindent\textbf{Step 2:} Calculate $\rho_1^k(\rho_3\rho_1)^s\{x_1,x_2,x_3,x_4,x_5\}.$

From step 1, $\rho_1^k(\rho_3\rho_1)^s\{x_1,x_2,x_3,x_4,x_5\}=(\rho_3\rho_1)^s\{x_{k+1},x_{k+2},x_{k+3},x_{k+4},x_{k+5}\}.$

Since $g(0,k)=g(1,k)=0$, no matter the parity of $k$, when $s=0$, the theorem holds. Now assume that the theorem holds for some $s-1\geq0$. It suffices to show
\begin{align*}
(\rho_3\rho_1)\{&A^{g(s,k)}B^{g(s,k+1)}x_{k+s},A^{g(s-1,k)}B^{g(s-1,k+1)}x_{k+s+1},A^{g(s,k)}B^{g(s,k+1)}x_{k+s+2},\\
&A^{g(s-1,k)}B^{g(s-1,k+1)}x_{k+s+3},A^{g(s,k)}B^{g(s,k+1)}x_{k+s+4}\}\\
=\{&A^{g(s+1,k)}B^{g(s+1,k+1)}x_{k+s+1},A^{g(s,k)}B^{g(s,k+1)}x_{k+s+2},A^{g(s+1,k)}B^{g(s+1,k+1)}x_{k+s+3},\\
&A^{g(s,k)}B^{g(s,k+1)}x_{k+s+4},A^{g(s+1,k)}B^{g(s+1,k+1)}x_{k+s+5}\}.
\end{align*}
Denote $\rho_1^k(\rho_3\rho_1)^{s-1}\{x_1,\ldots,x_5\}$ as $S$, and let $S[i]$ be the $i^{\text{th}}$ element of $S$. Recall that to apply $\rho_3\rho_1$ to $S$, we first do $\rho_3=\mu_2\circ\mu_4\circ(24)$. As we mutate vertex 2, the new cluster variable at vertex 2 is updated as
\begin{align*}
S[2]\leftarrow&\frac{S[1]S[5]+S[3]^2}{S[2]}\\
=&\frac{A^{2g(s,k)}B^{2g(s,k+1)}(x_{k+s}x_{k+s+4}+x_{k+s+2}^2)}{A^{g(s-1,k)}B^{g(s-1,k+1)}x_{k+s+1}}
\end{align*}
According to Lemma \ref{lem:AB} and Equation \eqref{eqn:ksodd}, if $k+s$ is odd, the above expression becomes
$$\frac{A^{2g(s,k)}B^{2g(s,k+1)}Ax_{k+s+3}}{A^{g(s-1,k)}B^{g(s-1,k+1)}}=A^{g(s+1,k)}B^{g(s+1,k+1)}x_{k+s+3}.$$
Similarly by Lemma \ref{lem:AB} and Equation \eqref{eqn:kseven}, if $k+s$ is even, we have
$$\frac{A^{2g(s,k)}B^{2g(s,k+1)}Bx_{k+s+3}}{A^{g(s-1,k)}B^{g(s-1,k+1)}}=A^{g(s+1,k)}B^{g(s+1,k+1)}x_{k+s+3}.$$

Then as we mutate vertex 4, with the same argument, we can show that
$$S[4]\leftarrow A^{g(s+1,k)}B^{g(s+1,k+1)}x_{k+s+1}.$$

So if we let $S'=\rho_3 S$, then $S'$ and $S$ differ only in the $2^{\text{nd}}$ and the $4^{\text{th}}$ coordinate. Specifically, $$S'[2]=A^{g(s+1,k)}B^{g(s+1,k+1)}x_{k+s+1},\quad S'[4]=A^{g(s+1,k)}B^{g(s+1,k+1)}x_{k+s+3}.$$

Finally, we mutate at vertex 1 in $S'$ and get
\begin{align*}
S'[1]\leftarrow&\frac{S'[2]S'[5]+S'[3]S'[4]}{S'[1]}\\
=&\frac{A^{g(s+1,k)+g(s,k)}B^{g(s+1,k+1)+g(s,k+1)}(x_{k+s+1}x_{k+s+4}+x_{k+s+2}x_{k+s+3})}{A^{g(s,k)}B^{g(s,k+1)}x_{k+s}}\\
=&A^{g(s+1,k)}B^{g(s+1,k+1)}x_{k+s+5}.
\end{align*}

After applying a permutation $(54321)$, we obtain the desired identity, which completes the induction step.
\end{proof}

\begin{corollary}\label{cor:rhocv}
All cluster variables that may appear through toric mutation sequences can be written in the forms
$$A^{n^2}B^{n(n-1)}x_{2m},\quad A^{n(n-1)}B^{n^2}x_{2m-1}\qquad\text{where }m,n\in\Z.$$
\end{corollary}
\begin{proof}
We first explain that all cluster variables that appear from toric mutations can be achieved by $\rho$-mutation sequences in the form of $\rho_1^k(\rho_3\rho_1)^s$, for some $k\in\Z$ and $s\in\Z{\geq0}$. According to Theorem \ref{thm:rho}, every toric mutation sequence from model 1 to model 1 can be written as $\rho_1^k(\rho_3\rho_1)^s$ or $\rho_1^{k}(\rho_3\rho_1)^s\rho_3$ for some $k\in\Z$ and $s\in\Z_{\geq0}$. The proof for Theorem \ref{thm:rhocv} shows that cluster variables of $\rho_1^{k}(\rho_3\rho_1)^s\rho_3\{x_1,\ldots,x_5\}$ are included in $\rho_1^{k}(\rho_3\rho_1)^s\{x_1,\ldots,x_5\}$ and $\rho_1^{k}(\rho_3\rho_1)^{s+1}\{x_1,\ldots,x_5\}$. Now we consider any toric mutation sequence that takes the original model 1 quiver to some model 2 quiver. According to Figure \ref{fig:model1-1}, this model 2 quiver can reach two different model 1 quivers in one step of toric mutation. So the cluster variables corresponding to this specific toric mutation sequence that ends on a model 2 quiver are included in the cluster variables that are generated by these two model 1 quivers.

Then we can take a closer look at the cluster variables shown in Theorem \ref{thm:rhocv}. Since $g(s,k)$ depends on the value of $s$ and the parity, but not the actual value, of $k$, it is easy to see that all cluster variables that appear can be written as $A^{g(s,k)}B^{g(s,k+1)}x_{k+s}$ for some $k\in\Z$ and $s\in\Z_{\geq0}$. Conversely, for any $k\in\Z$ and $s\in\Z_{\geq0}$, $A^{g(s,k)}B^{g(s,k+1)}x_{k+s}$ can be generated by a toric mutation sequence according to Theorem \ref{thm:rhocv}. To look at this term closely, we consider the following four cases according to the parity of $s$ and $k$.

\textbf{Case 1:} $s$ is even and $k$ is even. Let $s=2n$ and $k+s=2m$. We have $n\geq0$. Then
$$A^{g(s,k)}B^{g(s,k+1)}x_{k+s}=A^{\left\lfloor\frac{s}{2}\right\rfloor\left\lfloor\frac{s+1}{2}\right\rfloor}B^{\left\lfloor\frac{s-1}{2}\right\rfloor\left\lfloor\frac{s}{2}\right\rfloor}x_{2m}=A^{n^2}B^{n(n-1)}x_{2m}.$$

\textbf{Case 2:} $s$ is odd and $k$ is odd. Let $s=2n+1$ and $k+s=2m$. We have $n\geq0$. Then
$$A^{g(s,k)}B^{g(s,k+1)}x_{k+s}=A^{\left\lfloor\frac{s-1}{2}\right\rfloor\left\lfloor\frac{s}{2}\right\rfloor}B^{\left\lfloor\frac{s}{2}\right\rfloor\left\lfloor\frac{s+1}{2}\right\rfloor}x_{2m}=A^{n^2}B^{n(n+1)}x_{2m}.$$

\textbf{Case 3:} $s$ is even and $k$ is odd. Let $s=2n$ and $k+s=2m-1$. We have $n\geq0$. Then
$$A^{g(s,k)}B^{g(s,k+1)}x_{k+s}=A^{\left\lfloor\frac{s-1}{2}\right\rfloor\left\lfloor\frac{s}{2}\right\rfloor}B^{\left\lfloor\frac{s}{2}\right\rfloor\left\lfloor\frac{s+1}{2}\right\rfloor}x_{2m-1}=A^{n(n-1)}B^{n^2}x_{2m-1}.$$

\textbf{Case 4:} $s$ is odd and $k$ is even. let $s=2n+1$ and $k+s=2m-1$. We have $n\geq0$. Then
$$A^{g(s,k)}B^{g(s,k+1)}x_{k+s}=A^{\left\lfloor\frac{s}{2}\right\rfloor\left\lfloor\frac{s+1}{2}\right\rfloor}B^{\left\lfloor\frac{s-1}{2}\right\rfloor\left\lfloor\frac{s}{2}\right\rfloor}x_{2m-1}=A^{n(n+1)}B^{n^2}x_{2m-1}.$$

Cases 1 and 2 can be merged by letting $n\in\Z$ instead of just $n\in\Z_{\geq0}$. Similarly cases 3 and 4 can be merged. Finally, we conclude that all cluster variables generated by toric mutations can be written as either
$$A^{n^2}B^{n(n-1)}x_{2m},\quad A^{n(n-1)}B^{n^2}x_{2m-1}\qquad\text{where }m,n\in\Z.$$
\end{proof}

\section{Subgraphs of the Brane Tiling}\label{sec:weight}
For our purpose, every graph we consider is a subgraph of the dP$_2$ brane tiling so it is bipartite, planar and weighted. For such a graph $G$, which is bipartite, let $V_1$ and $V_2$ be its corresponding vertex sets. For any vertex set $V_0\subset V_1\cup V_2$, define $G-V_0$ to be the graph obtained by removing each vertex in $V_0$, as well as the edges that are incident to it, from $G$. These notations will be used for the rest of the paper.

We want to find a subgraph for each cluster variable that appears through toric mutations, such that the subgraph's weight equals the cluster variable. We use the weighting scheme utilized in \cite{lai2015beyond}, \cite{leoni2014aztec}, \cite{speyer2007perfect}, \cite{zhangcluster}, and elsewhere. 
\begin{definition}[Weight of Subgraphs]\label{def:weighting}
We associate a weight $\frac{1}{x_ix_j}$ to each edge bordering block labeled $i$ and $j$. For a set of edges $M$, define its weight $w(M)$ to be the product of the weights of the edges. For a subgraph $G$ of the brane tiling, let $\mathcal{M}(G)$ be the collection of its perfect matchings where each perfect matching is represented as a set of edges. Then, we define the weight of $G$ as
$$w(G)=\sum_{M\in\mathcal{M}(G)}w(M).$$
\end{definition}

In order to get recursive relations on the variables which correspond to subgraphs, we need lemmas that help us represent the weight of a large graph in terms of the weights of smaller graphs. Below we state Kuo's condensation theorems \cite{kuo2006graphical}, \cite{kuo2004applications}. 

\begin{lemma}[Balanced Kuo Condensation; Theorem 5.1 in \cite{kuo2004applications}]
\label{lem:KuoBal}
Let $G$ be a weighted planar bipartite graph discussed above with $|V_1|=|V_2|$. Assume that $p_1,p_2,p_3,p_4$ are four vertices appearing in a cyclic order on a face of $G$ with $p_1,p_3\in V_1$ and $p_2,p_4\in V_2$. Then
\begin{align*}
w(G)w(G-\{p_1,p_2,p_3,p_4\})=&w(G-\{p_1,p_2\})w(G-\{p_3,p_4\})\\
&+w(G-\{p_1,p_4\})w(G-\{p_2,p_3\}).
\end{align*}
\end{lemma}

\begin{lemma}[Unbalanced Kuo Condensation; Theorem 5.2 in \cite{kuo2004applications}]
\label{lem:KuoUnb}
Let $G$ be a weighted planar bipartite graph discussed above with $|V_1|=|V_2|+1$. Assume that $p_1,p_2,p_3,p_4$ are four vertices appearing in a cyclic order on a face of $G$ with $p_1,p_2,p_3\in V_1$ and $p_4\in V_2$. Then
\begin{align*}
w(G-\{p_2\})w(G-\{p_1,p_3,p_4\})=&w(G-\{p_1\})w(G-\{p_2,p_3,p_4\})\\
&+w(G-\{p_3\})w(G-\{p_1,p_2,p_4\}).
\end{align*}
\end{lemma}

\begin{lemma}[Non-alternating Kuo Condensation; Theorem 5.3 in \cite{kuo2004applications}]
\label{lem:KuoNal}
Let $G$ be a weighted planar bipartite graph discussed above with $|V_1|=|V_2|$. Assume that $p_1,p_2,p_3,p_4$ are four vertices appearing in a cyclic order on a face of $G$ with $p_1,p_2\in V_1$ and $p_3,p_4\in V_2$. Then
\begin{align*}
w(G-\{p_1,p_4\})w(G-\{p_2,p_3\})=&w(G)w(G-\{p_1,p_2,p_3,p_4\})\\
&+w(G-\{p_1,p_3\})w(G-\{p_2,p_4\}).
\end{align*}
\end{lemma}

\section{Contours for Cluster Variables}\label{sec:contour}
\label{sec:contour}

In this section, we describe a method to get the subgraph corresponding to any cluster variable obtained by toric mutations for the dP$_2$ quiver. Specifically, we use 5-sided contours to cut our brane tiling. We will define the rules to cut-out the subgraphs, and the formula of the contours. 

\subsection{Graphs from Contours}
Given a 5-tuple $(a,b,c,d,e)\in\Z^5$ with $a+b=d$ and $a+e=c$ (see Figure~\ref{fig:fund} right for those relations), we consider a 5-sided \textbf{contour} whose side-lengths are $a,b,c,d,e$ in clockwise order, starting from the upper right corner. Figure~\ref{fig:fund} (left) shows the fundamental shape of the contour, with each length being positive. In the case of negative side-lengths, we draw the corresponding side in the opposite direction. 

See Figure~\ref{fig:length} (left) for an example of a $5$-tuple and its contour.  We abuse notation and denote a geometric contour by its corresponding $5$-tuple.

\begin{figure}[h!]
\includegraphics[scale=0.1]{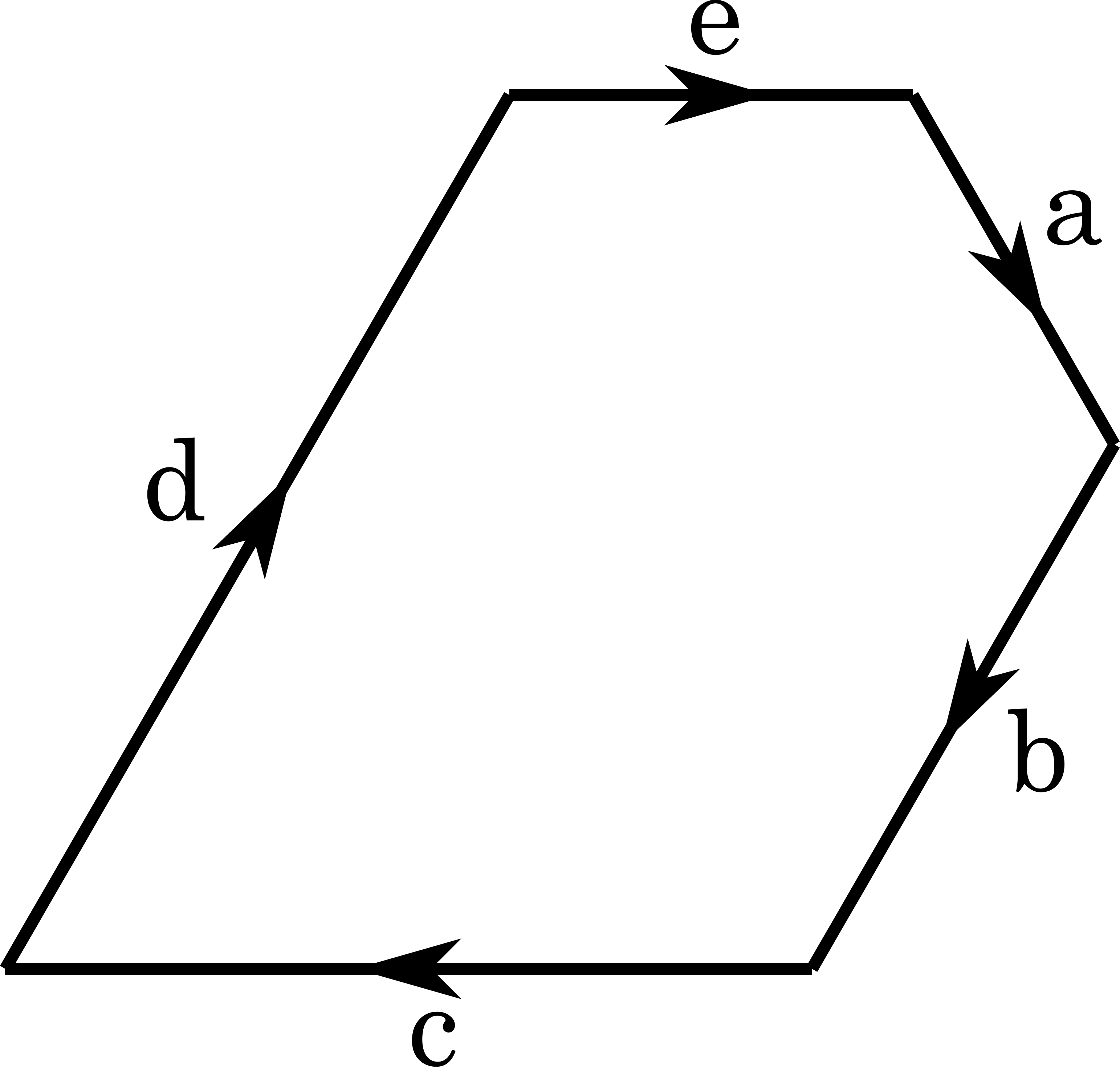}
\qquad
\includegraphics[scale=0.1]{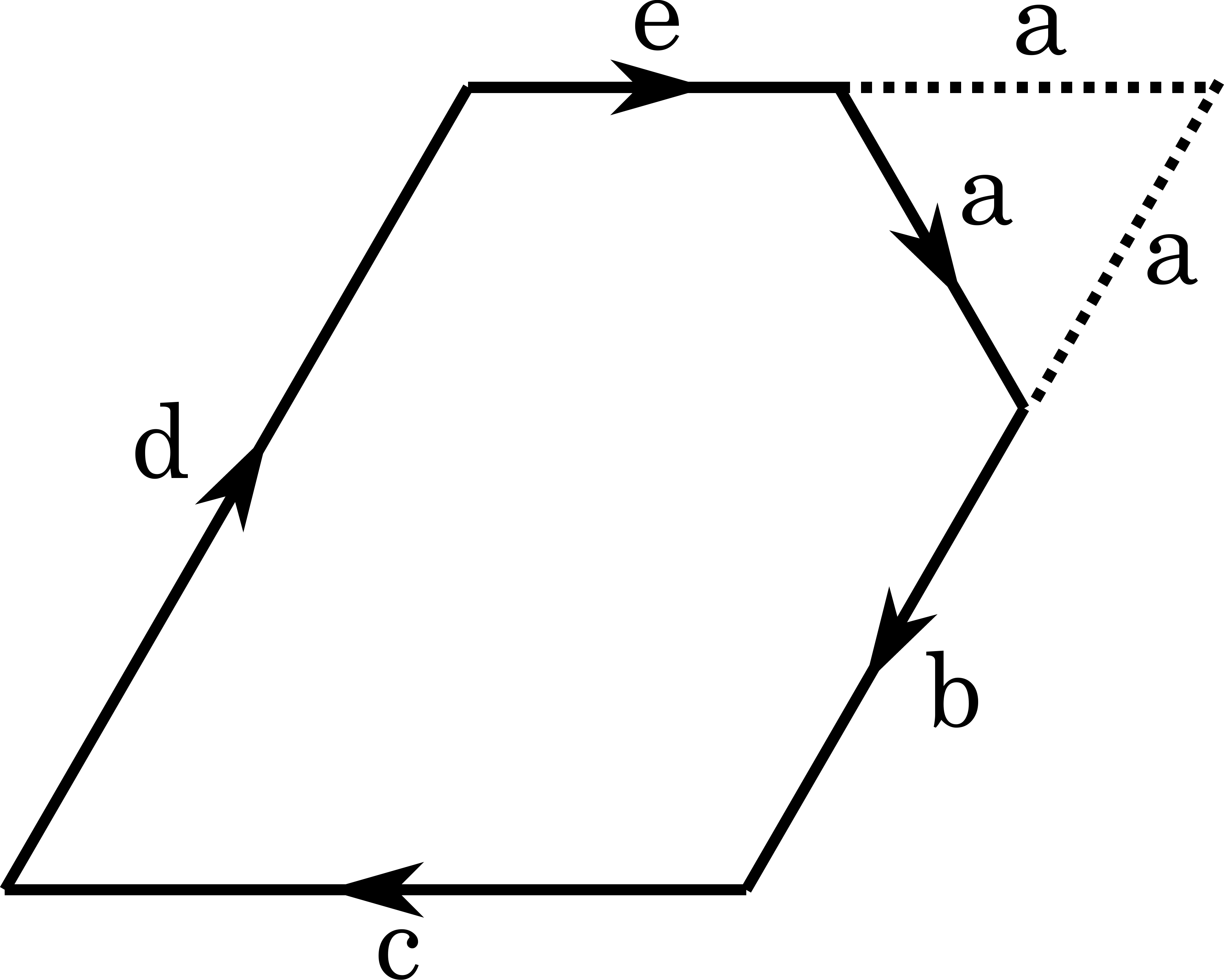}
\caption{\textbf{Left}: 5-sided fundamental shape; \textbf{Right}: relations between side lengths.}
\label{fig:fund}
\end{figure} 

Now we define the rule to get a subgraph from a contour.

\begin{definition} [Rules to Get Subgraph] \label{def:get-subgraph}

\ 

The white vertex between edges $c$ and $d$ is called the $\textbf{special vertex}$.

$\textbf{Step 1:}$ Given a 5-sided contour $C = (a,b,c,d,e)\in\mathbb{Z}^5$, we superimpose the contour on the brane tiling $\mathcal{T}$ such that the vertex between side $a$ and $e$ sits on any white vertex of degree 5, while each side follows.

$\textbf{Step 2:}$ On each side of positive length, we keep the black points while removing the white points; on each side of negative length, we keep the white points while removing the black points; on each side of zero length, we remove the single white point if it is not the special vertex. 

$\textbf{Step 3:}$ Each corner vertex is white. If the two adjacent sides of a corner vertex are both non-positive, we keep the vertex; otherwise, we remove it. As for the special vertex, if $a$ is even, we keep the special vertex; if $a$ is odd, we remove the special vertex.  Call the graph that remains inside the contour $\cal{G}(C)$.

$\textbf{Step 4:}$ In the resulting graph, we connect any vertex of valence 1 to its adjacent vertex. Call the edge of this connection a \textbf{forced matching}. Then delete these two vertices from the graph. Repeat this step until every vertex in the subgraph has valence at least 2. 

$\textbf{Step 5:}$ Call the resulting graph $\hatcal{G}(C)$ the $\textbf{subgraph}$ of contour $C$.  Often we may refer to $\hatcal{G}(C)$ as either $\hatcal{G}(a,b,c,d,e)$ or simply $\hatcal{G}$.  
\end{definition}

\begin{definition}
For any graph $G$, let $\hat{G}$ denote the graph obtained by removing all forced matchings.
\end{definition}

\begin{remark}
Note that our notation of graphs $\cal{G}$ and $\hatcal{G}$ for a contour is the opposite of the notation in \cite{lai2015beyond}.
\end{remark}

\begin{figure}
\includegraphics[scale = 0.18]{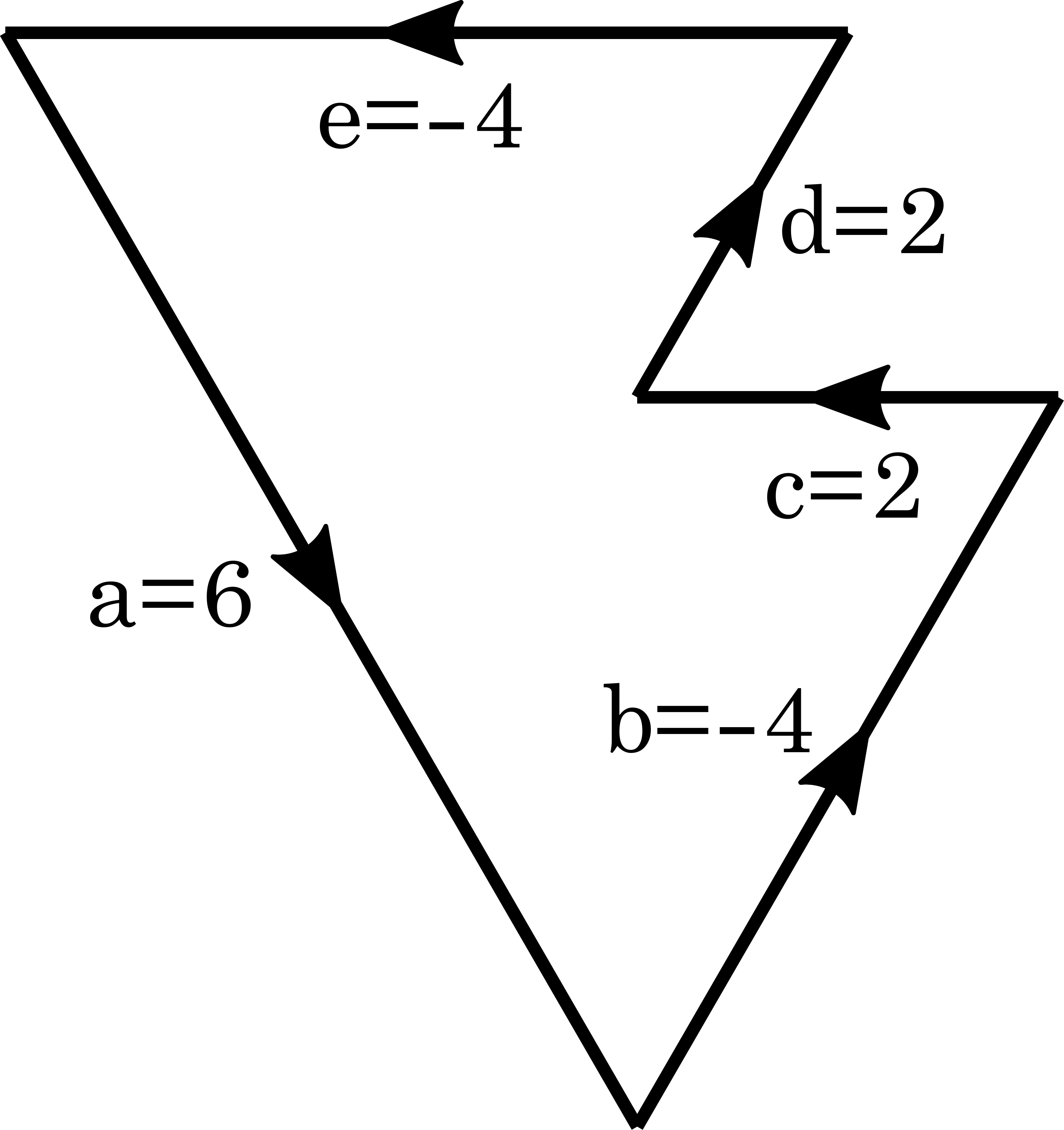}
\includegraphics[scale = 0.18]{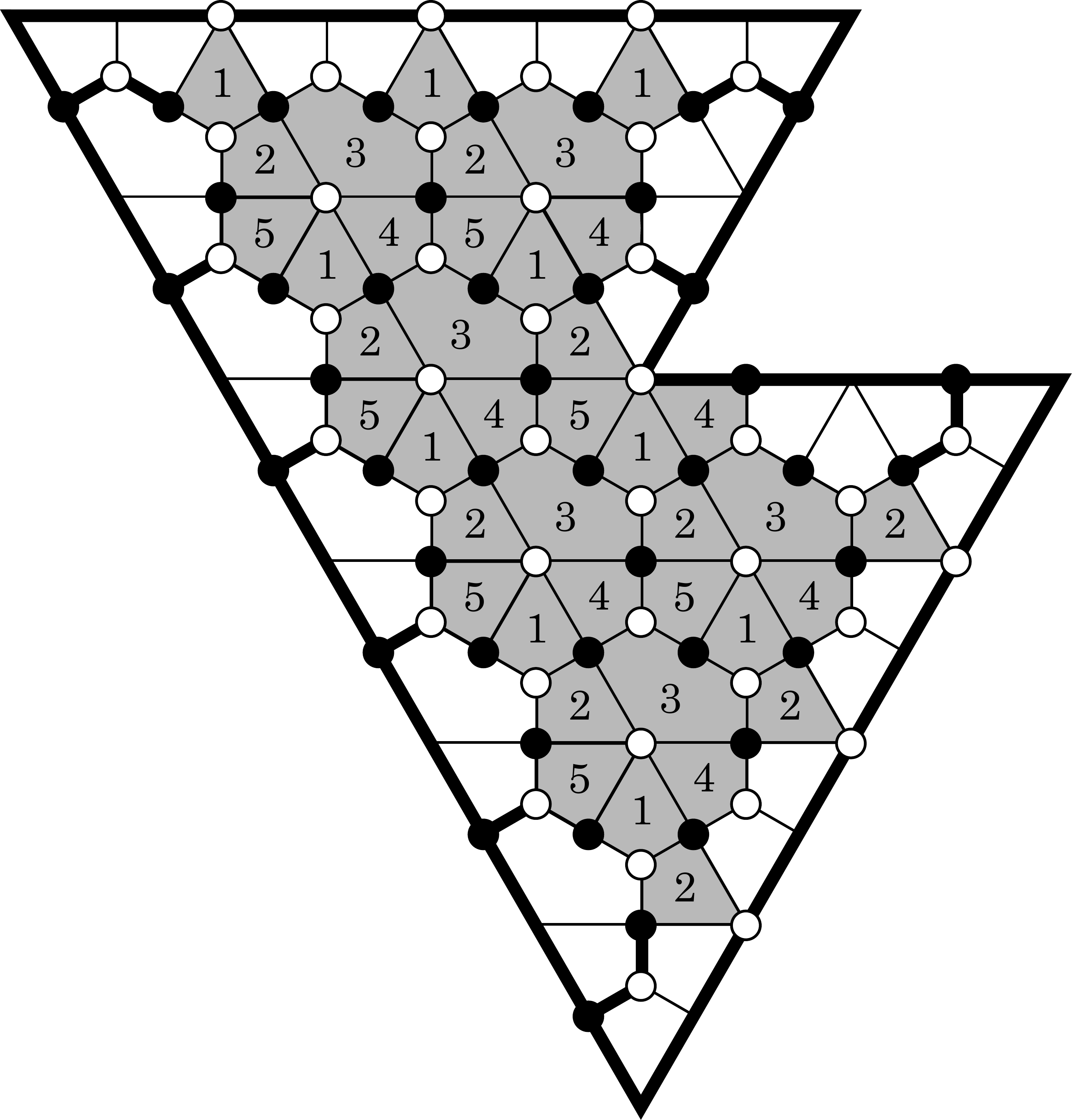}
\includegraphics[scale = 0.18]{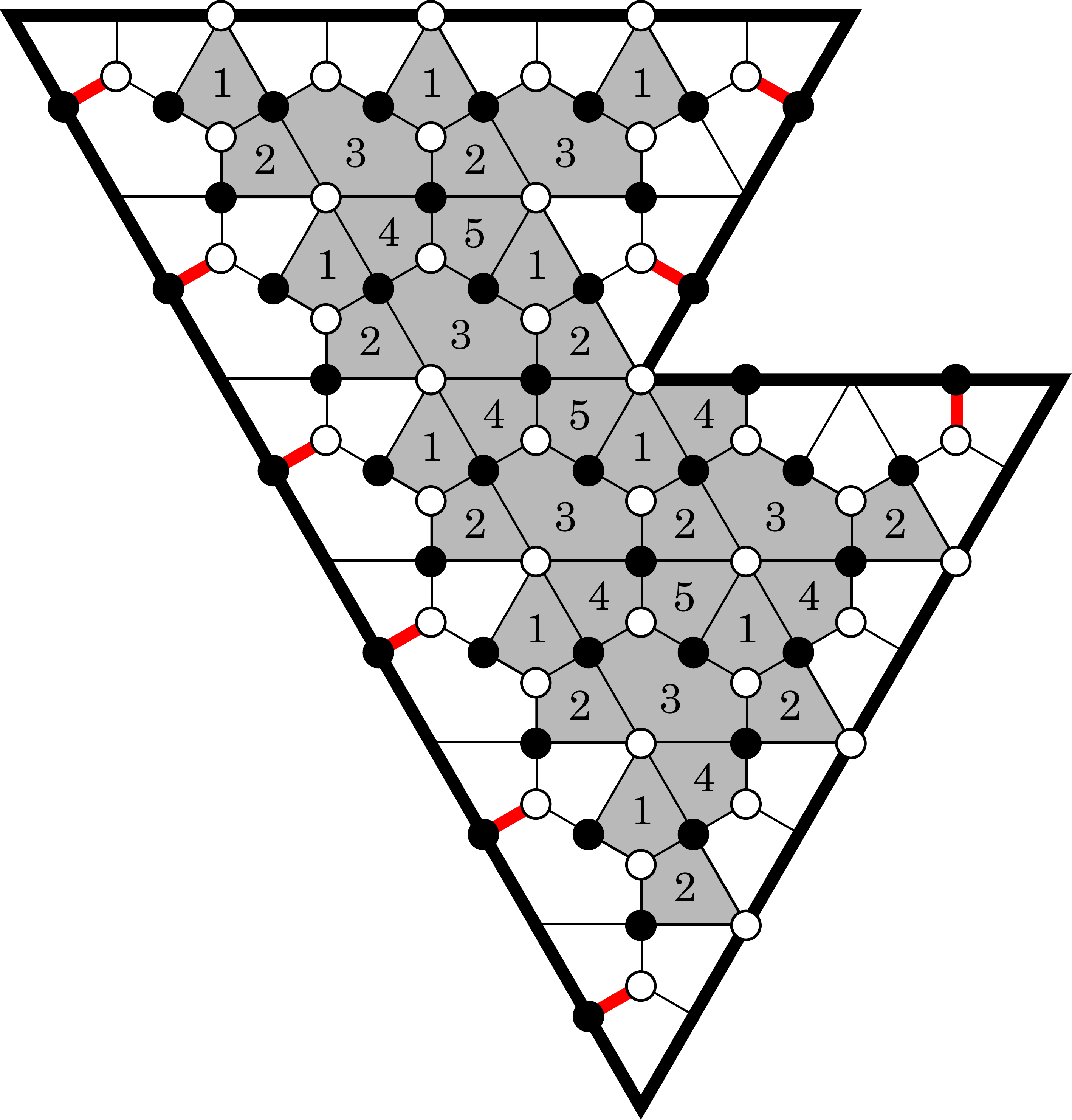}
\caption{Example of a contour $C = (6, -4, 2, 2, -4)$ and its subgraphs $\cal{G}(C)$ (shaded region and darkened edges) and $\hatcal{G}(C)$ (shaded region).}
\label{fig:length}
\end{figure}

We have already defined the weighting $w(G)$ of a graph $G$ in Definition~\ref{def:weighting}.  To fully recover the cluster variables from graphs, we define covering monomials for this specific brane tiling. The covering monomial has a more general definition in \cite{jeong2000brane} and \cite{jeong2011bipartite}. 

\begin{definition}[Covering Monomial]\label{def:cm}
For this definition, we think of every block labeled $3$ as two separate blocks labeled $3$. 
Given a contour $C$, let $a_j$ be the number of blocks labeled $j$ enclosed in $C$.  Let $b_j$ be the number of blocks labeled $j$ adjacent to a forced matching in $C$.  If the special vertex is kept (i.e. if $a$ is even) and the contour passes through the middle of a $3$-block near the special vertex (see Figure~\ref{fig:cov-mon-rule}), let $c_3 = 1$.  Otherwise, let $c_3 = 0$.  The covering monomial of graph $\cal{G}(C)$, denoted as $m(\cal{G}(C))$, is the product $x_1^{a_1}x_2^{a_2}x_3^{a_3+c_3}x_4^{a_4}x_5^{a_5}$. The covering monomial of graph $\hatcal{G}(C)$, denoted as $m(\hatcal{G}(C))$, is the product $x_1^{a_1-b_1}x_2^{a_2-b_2}x_3^{a_3-b_3+c_3}x_4^{a_4-b_4}x_5^{a_5-b_5} = \frac{m(\cal{G}(C))}{x_1^{b_1}x_2^{b_2}x_3^{b_3}x_4^{b_4}x_5^{b_5}}$.
\end{definition}

\begin{figure}[h!]
\includegraphics[scale = 0.2]{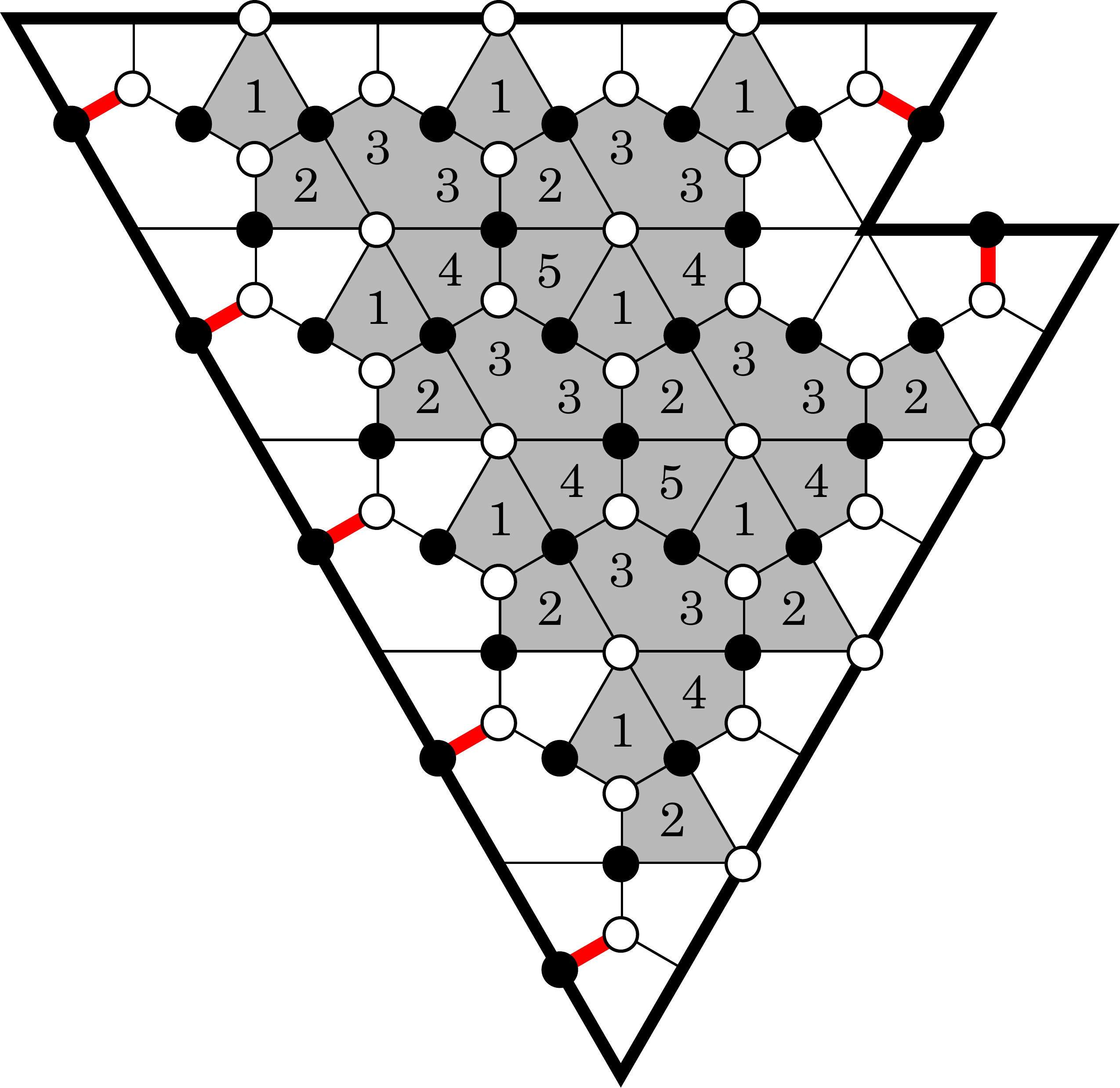}
\includegraphics[scale = 0.2]{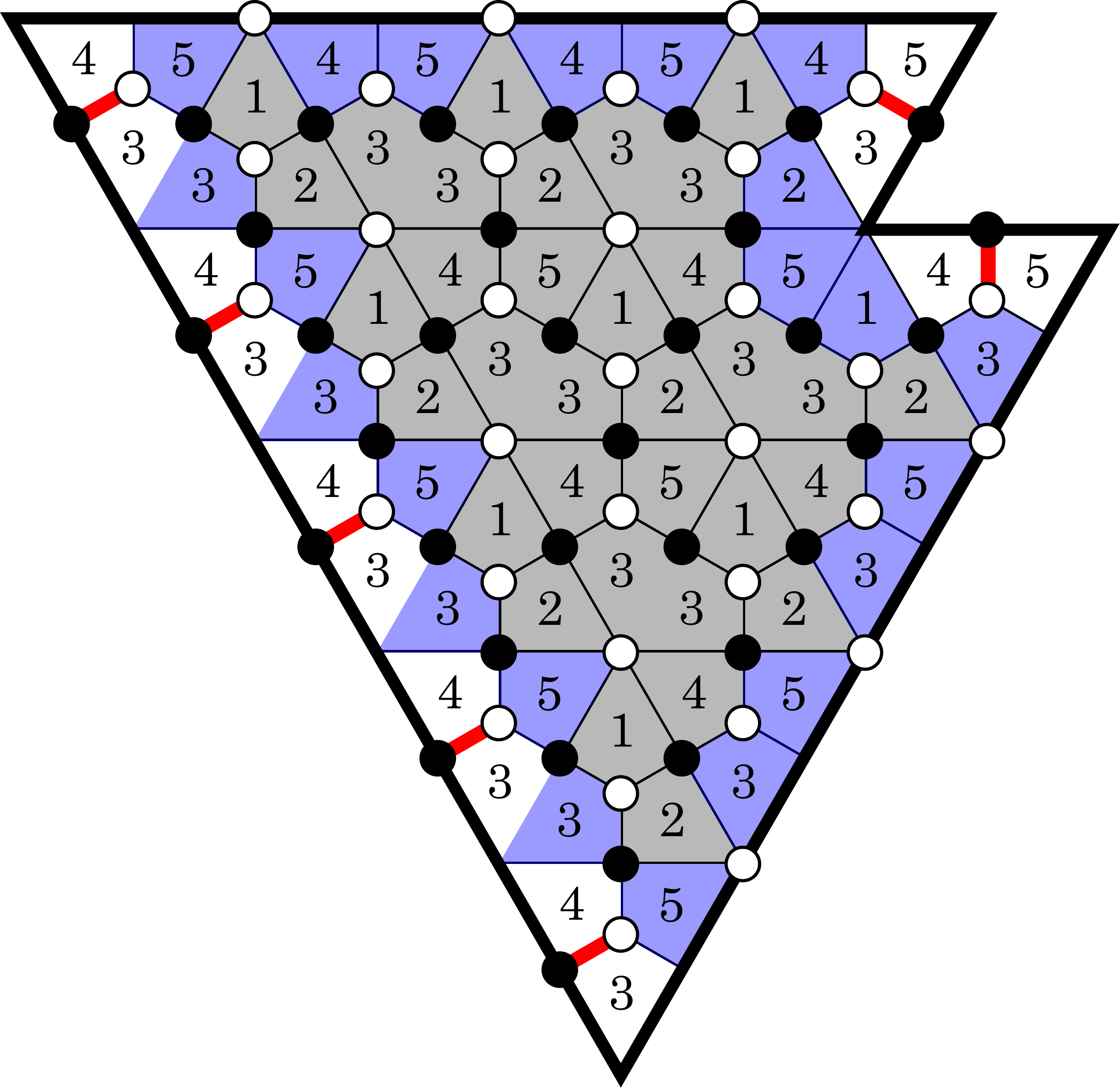}
\caption{Example of a subgraph $\hatcal{G}(5,-4,1,1,-4)$ and its covering monomial $m(\hatcal{G}(5,-4,1,1,-4))$.  The gray and purple blocks are included in the covering monomial.}
\label{fig:cov-mon-rule}
\end{figure}

\begin{remark}
Our definitions of weight and covering monomial remain unchanged if we think of each six sided $3$-block as two separate four sided blocks without an edge between them.  Each $3$-block will be drawn as two separate $3$-blocks if they appear on the boundary of our contour for sake of visualizing weight and covering monomial. 
\end{remark}

For any graph $G$ with an associated contour, denote the product of its weight and its covering monomial as 
\begin{align*}
c(G) := w(G)m(G).
\end{align*}

\subsection{Contours of Cluster Variables}

By Corollary 4.5, we have that all the cluster variables are of the form $A^{n^2}B^{n^2-n}x_{2k}$ or $A^{n^2+n}B^{n^2}x_{2k-1}$ where $n, k\in\Z$. Now we state the main result of this section that gives a formula of the contours of these two families. 

\begin{theorem}\label{thm:contours}
For $k\geq2$, we associate the following contours to the cluster variables such that if $C$ is the contour associated with a cluster variable, then $c(\hatcal{G}(C))$ equals the Laurent polynomial of that cluster variable.
\begin{align*}
A^{n^2}B^{n^2-n}x_{2k} &= c\left(\hatcal{G} \left(k-2+n,-\left\lceil \frac{k-4+5n}{2}\right\rceil, 2n-1, \left\lfloor \frac{k-3n}{2} \right\rfloor, 1+n-k\right)\right),\\
A^{n^2+n}B^{n^2}x_{2k-1} &= c\left(\hatcal{G} \left(k-2+n,-\left\lceil \frac{k-2+5n}{2} \right\rceil, 2n, \left\lfloor\frac{k-2-3n}{2} \right\rfloor, 2+n-k\right)\right).
\end{align*}
\end{theorem}

Notice that when $k\leq 1$, we can reflect the subgraph of $A^pB^qx_{6-2k}$ ($p, q\in \Z$) along $x_3$, which means we replace $x_2$ with $x_4$ and $x_1$ with $x_5$ to get the subgraph of $A^pB^qx_{2k}$ since block 2 and block 4, and block 1 and block 5 are symmetric with respect to $x_3$ in the brane tiling and $A,\ B$ are also fixed if we interchange $x_2$ with $x_4$ or/and $x_1$ with $x_5$. Therefore we only need to consider the situation where $k \geq 2$.

Before proving this result, we first look at the six possible shapes of the contours based on the relationship between $n$ and $k$, as is shown in Figure \ref{fig:possible_shapes}.

\begin{figure}[h!]
\includegraphics[scale=0.5]{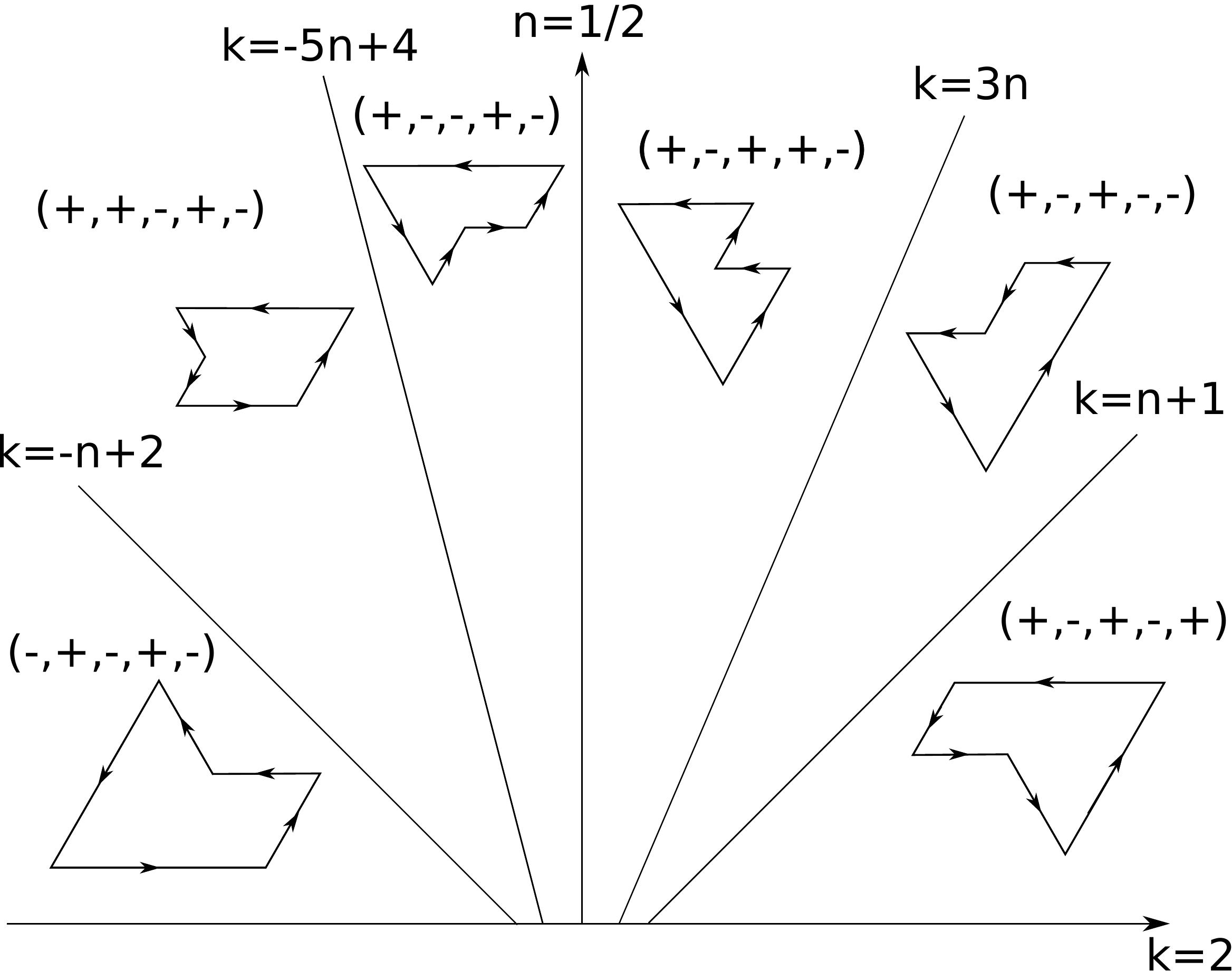}
\caption{Possible shapes of 5-sided contour}
\label{fig:possible_shapes}
\end{figure}

We use the entire next section to prove the main theorem.

\section{Proof of Main Theorem (Theorem~\ref{thm:contours})}\label{sec:proof}

\subsection{Overview of induction procedure}\label{sub:overview}
We use Kuo's condensation to inductively prove that multiplying the weight and covering monomial of these contours yields the Laurent polynomials of our cluster variables.  First we show that the weights satisfy the desired recurrences.  Then we show that for any form of recurrence, the covering monomials will be correct in the sense that multiplying the weight and covering monomials of our subgraphs gives the Laurent polynomial.  We abuse notation by saying a graph $G$ equals a cluster variable when we mean the weight of $G$, $w(G)$, will give us the cluster variable's Laurent polynomial with the appropriate covering monomial. 

The base case is $n=0$, which is proved in Section~\ref{sub:base}. Notice that when $n=0$, our formula for the contour in the main theorem contains two families: $\{x_{2k-1}\}_{k\geq2}$ and $\{x_{2k}\}_{k\geq2}$. So essentially, there are two different cases here.

After proving the base case when $n=0$, we split our way to consider the families of variables with $n>0$ and families of variables with $n<0$ separately. 

For $n\geq1$, by induction hypothesis, assume that we already have the contours for variables $A^{m^2}B^{m(m-1)}x_{2k}$ and $A^{m(m+1)}B^{m^2}x_{2k-1}$ for all $k\geq2$ and $0\leq m\leq n-1$. Then for each $k\geq2$, consider the following identity (recurrence):
\begin{align*}
&(A^{n^2}B^{n(n-1)}x_{2k})(A^{(n-1)^2}B^{(n-1)(n-2)}x_{2k+2}) \\
=&(A^{(n-1)n}B^{(n-1)^2}x_{2k-1})(A^{(n-1)n}B^{(n-1)^2}x_{2k+3}) + (A^{(n-1)n}B^{(n-1)^2}x_{2k+1})^2.
\end{align*}
It is clear that among all the five terms appeared above, $A^{n^2}B^{n(n-1)}x_{2k}$ is the only term that we do not have already. Therefore, it suffices to find some graph $G$ and points $p_1,p_2,p_3,p_4$ and use some version of Kuo's condensation theorem on it to prove that this term actually equals the weight of the subgraph that we described in the main theorem, correspondingly. Note that the graph $G$ we use and the points $p_1,p_2,p_3,p_4$ we choose depend on some relations between $n$ and $k$. Now that we have the terms $\{A^{n^2}B^{n(n-1)}x_{2k-1}\}$ for $k\geq2$, consider the following identity (recurrence) for each $k\geq3$:
\begin{align*}
&(A^{n(n+1)}B^{n^2}x_{2k-1})(A^{(n-1)^2+(n-1)}B^{(n-1)^2}x_{2k+1}) \\
=&(A^{n^2}B^{n(n-1)}x_{2k-2})(A^{n^2}B^{n(n-1)}x_{2k+2}) + (A^{n^2}B^{n(n-1)}x_{2k})^2.
\end{align*}
Similarly, there is only one term $A^{n(n+1)}B^{n^2}x_{2k-1}$ that we do not currently have. And by some Kuo's condensation theorem, we will get the desired result. One thing to notice here is that the above recurrence cannot be applied to $k=2$ since we do not have the term $A^{n^2}B^{n(n-1)}x_2$ in our theorem. To solve this problem and get the contour formula for $A^{n(n+1)}B^{n^2}x_3$, we use the following recurrence:
\begin{align*}
&(A^{n(n+1)}B^{n^2}x_3)(A^{n^2}B^{n(n-1)}x_8) \\
=&(A^{n(n+1)}B^{n^2}x_5)(A^{n^2}B^{n(n-1)}x_6) + (A^{n(n+1)}B^{n^2}x_7)(A^{n^2}B^{n(n-1)}x_4).
\end{align*}
Once this step is done, our inductive step is finished.

For $n\leq-1$, the argument is very similar. By induction hypothesis, assume that we already have the contours for variables $A^{m^2}B^{m(m-1)}x_{2k}$ and $A^{m(m+1)}B^{m^2}x_{2k-1}$ for all $n+1\leq m\leq 0$. The recurrence
\begin{align*}
&(A^{n^2}B^{n(n-1)}x_{2k})(A^{(n+1)^2}B^{n(n+1)}x_{2k+2}) \\
=&(A^{n(n+1)}B^{n^2}x_{2k-1})(A^{n(n+1)}B^{n^2}x_{2k+3}) + (A^{n(n+1)}B^{n^2}x_{2k+1})^2
\end{align*}
will give us contours for all variables $A^{n^2}B^{n(n-1)}x_{2k}$ for all $k\geq2$. After that, the recurrence
\begin{align*}
&(A^{n(n+1)}B^{n^2}x_{2k-1})(A^{(n+1)^2+(n+1)}B^{(n+1)^2}x_{2k+1}) \\
=&(A^{(n+1)^2}B^{(n+1)n}x_{2k-2})(A^{(n+1)^2}B^{(n+1)n}x_{2k+2}) + (A^{(n+1)^2}B^{(n+1)n}x_{2k})^2
\end{align*}
will give us contours for all variables $A^{n(n+1)}B^{n^2}x_{2k-1}$ for all $k\geq 3$. For the missing variables in the form of $A^{n(n+1)}B^{n^2}x_3$, we use the recurrence
\begin{align*}
&(A^{n(n+1)}B^{n^2}x_3)(A^{(n+1)^2}B^{n(n+1)}x_8) \\
=&(A^{n(n+1)}B^{n^2}x_5)(A^{(n+1)^2}B^{n(n+1)}x_6) + (A^{n(n+1)}B^{n^2}x_7)(A^{(n+1)^2}B^{n(n+1)}x_4).
\end{align*}
This completes the inductive step.

Section~\ref{sub:base} proves the base case $(n=0)$ and Section~\ref{sub:induction} proves one case of the inductive step. Notice that for the inductive step, we have 28 cases in total and we will not present explicit proofs for all cases. The cases are divided, generally speaking, by whether side lengths of the contour are greater or smaller than 0 and by some parity conditions on $n$ and $k$. Section~\ref{sub:technique} gives a summary of the techniques used to prove the remaining cases. All of them can be proved in exactly the same format, and in Appendix~\ref{sec:appendix}, we provide the necessary data for readers to verify the correctness of these remaining cases. 

\subsection{Overview of Proof Techniques}\label{sub:technique}
We have different cases to prove depending on the relations between $n$ and $k$ since different relations will lead to different shapes of our contour. In this section, we introduce the general format while the details are given in Appendix~\ref{sec:appendix}.

\

\noindent\textbf{Step 1.} 
Consider a contour $C=(a,b,c,d,e)$ with the special vertex kept or removed, and 4 points $p_1,p_2,p_3,p_4$ inside the contour. Depending on the whether the graph $\cal{G}(C)$ is balanced or not and depending on the colors and positions of $p_1,p_2,p_3,p_4$, we use a particular version of Kuo's condensation theorems which are always of the form: $$w\left(\cal{G}(C)-S_1\right)w(\cal{G}(C)-S_2)=w\left(\cal{G}(C)-S_3\right)w(\cal{G}(C)-S_4)+w\left(\cal{G}(C)-S_5\right)w(\cal{G}(C)-S_6),$$ where each $S_i$ is a subset of $\{p_1,p_2,p_3,p_4\}$. Notice that in this step, $\cal{G}(C)-S_i$ may include many forced matchings. 
Now we multiply both sides of the equation by $m(\cal{G}(C))^2$, the square of the covering monomial of the graph $G$. Each term in the equation is then of the form $m(\cal{G}(C))w(\cal{G}(C)-S_i).$

\

\noindent\textbf{Step 2.}
For each $i=1,\ldots,6$, we find a contour $C_i$ inside $C$ such that $\hat{\cal{G}(C)-S_i}=\hat{\cal{G}}(C_i)$. Recall that $\hat{G}$ is graph $G$ with all forced matchings removed. We find $C_i$ by first describing points $p_1,p_2,p_3,p_4$ and how removing each point separately will change the contour $C$. Then we can add these effects together to get the total effect of removing $S_i$. Notice that the additivity of such effects is nontrivial in general, but it is easy to verify for each of our cases.

This is the core step of our proof. The effects of removing each point $p_i$ from $\hat{\cal{G}}(C)$ will be stated and justified through diagrams.

\

\noindent\textbf{Step 3.}
Now we want to relate $m(\cal{G}(C))w(\cal{G}(C)-S_i)$ to $c(\hat{\cal{G}}(C_i))$. Consider $\cal{G}(C_i)$. By definition, we know that $\cal{G}(C_i)$ and $\cal{G}(C)-S_i$ only differ by a set of forced matchings of $\cal{G}(C)-S_i$ inside contour $C$ and outside contour $C_i$. Meanwhile, $m(\cal{G}(C))$ and $m(\cal{G}(C_i))$ differ by a factor of the product of all the blocks (the product of variables corresponding to the blocks) inside $C$ but outside $C_i$. As each block can be in only one forced matching (otherwise the matching would not be forced), the quotient
$$\frac{m(\cal{G}(C))w(\cal{G}(C)-S_i)}{m(\cal{G}(C_i))w(\cal{G}(C_i))}$$
is the product of all the blocks inside $C$ and outside $C_i$ that are not adjacent to any forced matchings inside $C$ and outside $C_i$. Let these blocks form set $T_i$. We are also using the notation $T(S_i)$ with $T_i$ interchangeably. For each case, we will explicitly provide $T_1,\ldots,T_6$ for a choice of points $p_1,p_2,p_3,p_4$ and check that
\begin{equation}
\left(\prod_{j\in T_1}x_j\right)\left(\prod_{j\in T_2}x_j\right)=\left(\prod_{j\in T_3}x_j\right)\left(\prod_{j\in T_4}x_j\right)=\left(\prod_{j\in T_5}x_j\right)\left(\prod_{j\in T_6}x_j\right).
\label{eqn:T_i}
\end{equation}

Also, notice that
$$m(\cal{G}(C_i))w(\cal{G}(C_i))=m(\hat{\cal{G}}(C_i))w(\hat{\cal{G}}(C_i))=:c(\hat{\cal{G}}(C_i))$$
since by definition, both $m(\cal{G}(C_i))/m(\hat{\cal{G}}(C_i))$ and $w(\hat{\cal{G}}(C_i))/w(\cal{G}(C_i))$ equals the product of blocks adjacent to the forced matchings of $\cal{G}(C_i).$ Combining these arguments, we conclude that
$$c(G_1)c(G_2)=c(G_3)c(G_4)+c(G_5)c(G_6)$$
where $G_i=\cal{G}(C_i).$

In this step, we are essentially checking that the covering monomials match up with the weights used in Kuo's condensation theorems to give the correct Laurent polynomials. 

\

\noindent\textbf{Step 4.} Using the induction hypothesis, we can identify five of the expressions $c(G_i)$ as the Laurent polynomials of cluster variables. Therefore, the sixth expression is the Laurent polynomial of the next cluster variable in the sequence.

\

We provide the details of these steps in Section~\ref{sub:base} and Section~\ref{sub:induction}.

\begin{definition}[Notation]\label{def:notation}
We establish the following notations before presenting the proof.

Let $(a, b, c, d, e)-K$ be the contour of side lengths $a,b,c,d,e$ with the special vertex kept and $(a,b,c,d,e)-R$ be the contour of side lengths $a,b,c,d,e$ with the special vertex removed. We write $\cal{G}(a, b, c, d, e,)-K$ (resp. $-R$) to denote the subgraph obtained from contour $(a, b, c, d, e) -K$ (resp. $-R$).  Similarly for $\hatcal{G}$.

We say that point $p_i$ is a white (or black) point on edge $a$ (or $b,c,d,e$) if it is one of the white (or black) points on the boundary of $\hatcal{G}(C)$ facing edge $a$, where $C$ is some contour. This notation follows from \cite{lai2015beyond} and it does not necessarily mean that $p_i$ is on edge $a$ (or $b,c,d,e$) of the contour.
\end{definition}

\subsection{Base case $(n=0)$}\label{sub:base}
When $n = 0$, the cluster variables $A^{n^2}B^{n^2-n}x_{2k}$ and $A^{n^2+n}B^{n^2}x_{2k-1}$ where $n\in\Z, k \in\Z_{\ge 0}$ become the terms $\{x_m\}_{m \in \Z}$ of the Somos-5 sequence.  

For $1 \le i \le 5$, let $C_i$ be the contour defined in Theorem~\ref{thm:contours} for the initial cluster variable $x_i$.  We verify the weights and covering monomials of these contours.  As shown in Figure~\ref{fig:somos5}, the subgraphs for these cluster variables are empty so they have weight $1$. Recall that by definition, the covering monomials for $C_3$ and $C_4$ have an additional $x_3$ term.  We can see that $c(\hatcal{G}(C_i)) = x_i$ for $1 \le i \le 5$. 

\begin{figure}[h!]
\includegraphics[scale = 0.3]{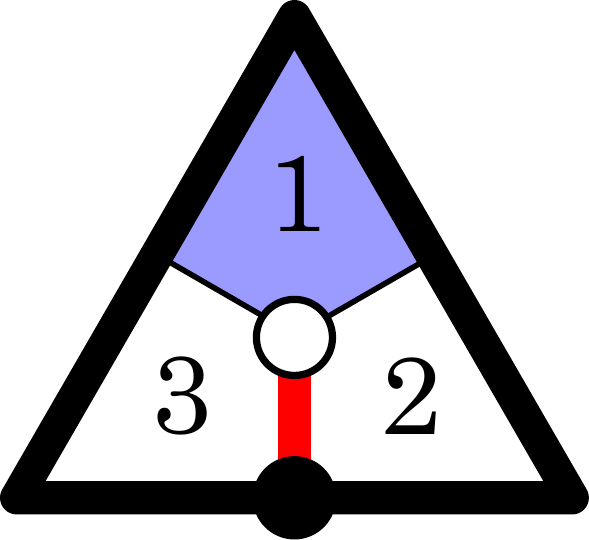}
\includegraphics[scale = 0.3]{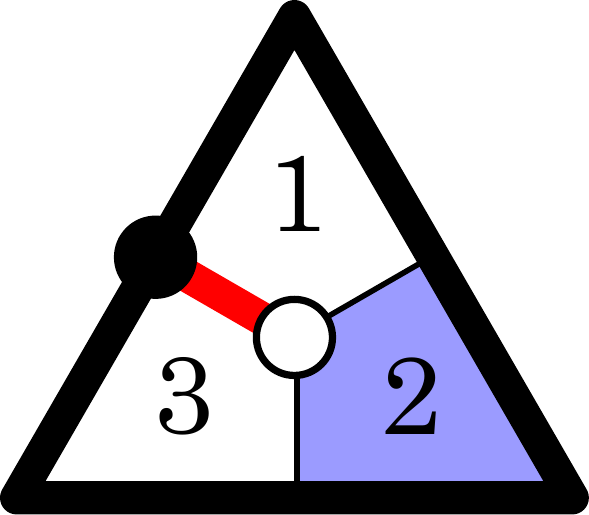}
\includegraphics[scale = 0.3]{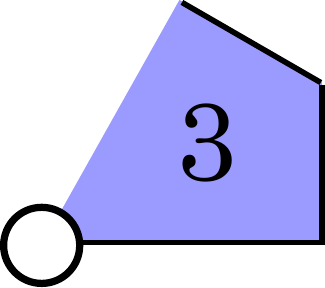}
\includegraphics[scale = 0.3]{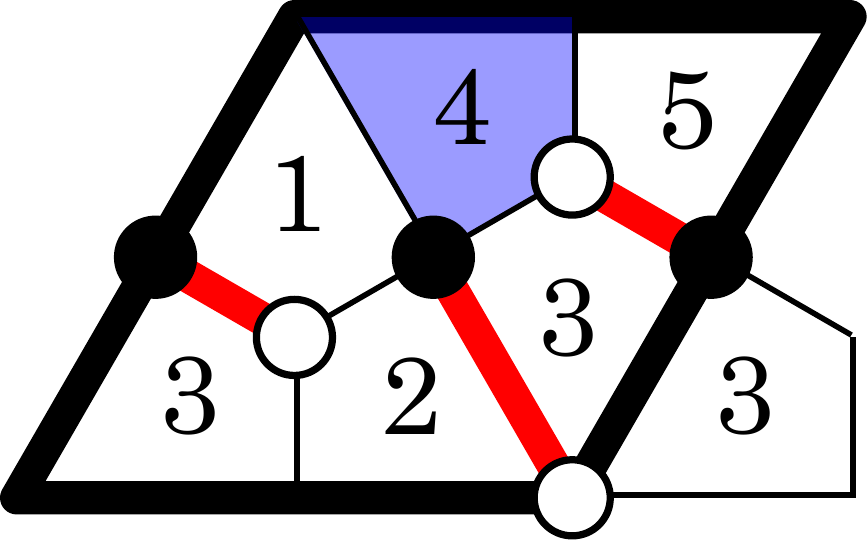}
\includegraphics[scale = 0.3]{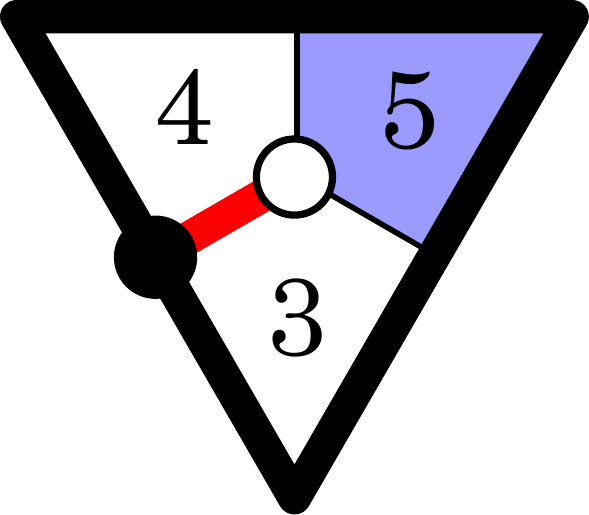}
\caption{For $1 \le i \le 5$, we give contours $C_i$ for terms $x_i$ of the Somos-5 sequence.  The purple blocks are what remain after multiplying the weights and covering monomials of these graphs.}
\label{fig:somos5}
\end{figure}

Now assume the contours for $x_i$ for all $i \le m-1$ give the correct Laurent polynomials for our cluster variables.  We show the contour defined in Theorem~\ref{thm:contours} for $x_m$ is correct.

\

\noindent\textbf{Case 1:} $m = 2k-1$. We take the following contour
\begin{align*}
C=(a,b,c,d,e)
=\left(k-2, -\left\lceil \frac{k-2}{2} \right\rceil, 0, \left\lfloor \frac{k-2}{2} \right\rfloor, 2-k\right).
\end{align*}
Since $k > 3$, we have $a>0, b<0, d \geq0$ and $e<0$.

Let $G=\hatcal{G}(C)$. We then follow the steps shown in Section~\ref{sub:technique}. 

\

\noindent\textbf{Step 1.} 
We apply balanced Kuo's condensation theorem (Lemma~\ref{lem:KuoBal}) and write down
\begin{align*}
w(\cal{G}(C))w(\cal{G}(C)-\{p_1,p_2,p_3,p_4\})=&w(\cal{G}(C)-\{p_1,p_2\})w(\cal{G}(C)-\{p_3,p_4\})\\
&+w(\cal{G}(C)-\{p_1,p_4\})w(\cal{G}(C)-\{p_2,p_3\}).
\end{align*}
where we let $S_1=\emptyset$, $S_2=\{p_1,p_2,p_3,p_4\}$, $S_3=\{p_1,p_2\}$, $S_4=\{p_3,p_4\}$, $S_5=\{p_1,p_4\}$, $S_6=\{p_2,p_3\}$. Then we multiply both sides by $m(\cal{G}(C))^2$. 

\

\noindent\textbf{Step 2.} We define the black points $p_1,p_3$ and white points $p_2,p_4$ as follows.

\begin{itemize}
\item Let $p_1$ be any black point on edge $e$.
\item Let $p_2$ be any white point on edge $a$.  
\item Let $p_3$ be any black point on edge $b$.
\item Let $p_4$ be a white point near edge $c$ defined as follows: 
\begin{itemize}
\item If $k \equiv 0$ (mod 2), then $a\equiv 0$ (mod 2) so the special vertex is kept.  Let $p_4$ be the kept special white point between edges $c$ and $d$.
\item If $k \equiv 1$ (mod 2), then $a\not \equiv 0$ (mod 2) so the special vertex is removed.  Let $p_4$ be the other white point on the 5-block which contains the removed white point between edges $c$ and $d$.
\end{itemize}
\end{itemize}

We also give the effects of removing each point separately:
\begin{itemize}
\item The effect of removing $p_1$ is $(a,b,c,d,e)\rightarrow(a-1,b,c,d-1,e+1).$ We may also write this succinctly as $-\{p_1\}=(-1,0,0,-1,1).$ This effect is equivalent to deleting a trapezoid along edge $e$ of the original contour.
\item The effect of removing $p_2$ is $(a,b,c,d,e)\rightarrow(a-1,b+1,c,d,e+1)$. It is equivalent to deleting a trapezoid along edge $a$.
\item The effect of removing $p_3$ is $(a,b,c,d,e)\rightarrow(a-1,b+1,c-1,d,e).$
\item The effect of removing $p_4$ is $(a,b,c,d,e)-K\rightarrow(a,b,c,d,e)-R$ and $(a,b,c,d,e)-R\rightarrow(a,b+1,c,d+1,e)=K$ depending on the parity of $k$.
\end{itemize}

The position of each point and the effect of removing each point can be seen in Figure~\ref{fig:effect_base1_K} (special point kept) and Figure~\ref{fig:effect_base1_R} (special point removed). In the figures, we use big red dots to indicate point $p_i$ and red edges to indicate forced matchings. The shadowed region is what's removed from the original contour after deleting the corresponding point. We also use black letters $K/R$ to indicate whether the special point is kept or removed in the original contour and use blue letters for the new contour. 
\begin{figure}[h!]
\includegraphics[scale=0.1]{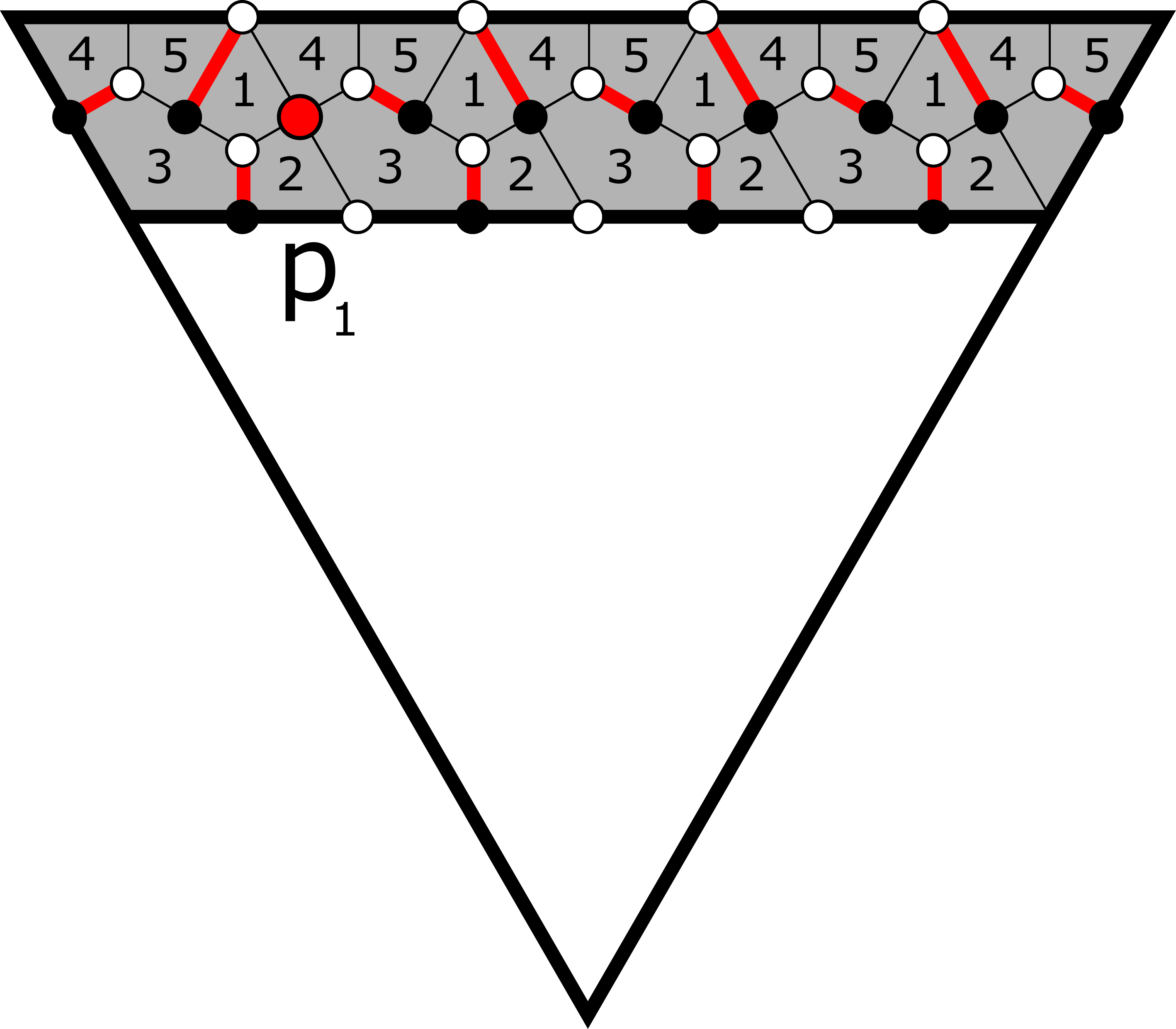}
\includegraphics[scale=0.1]{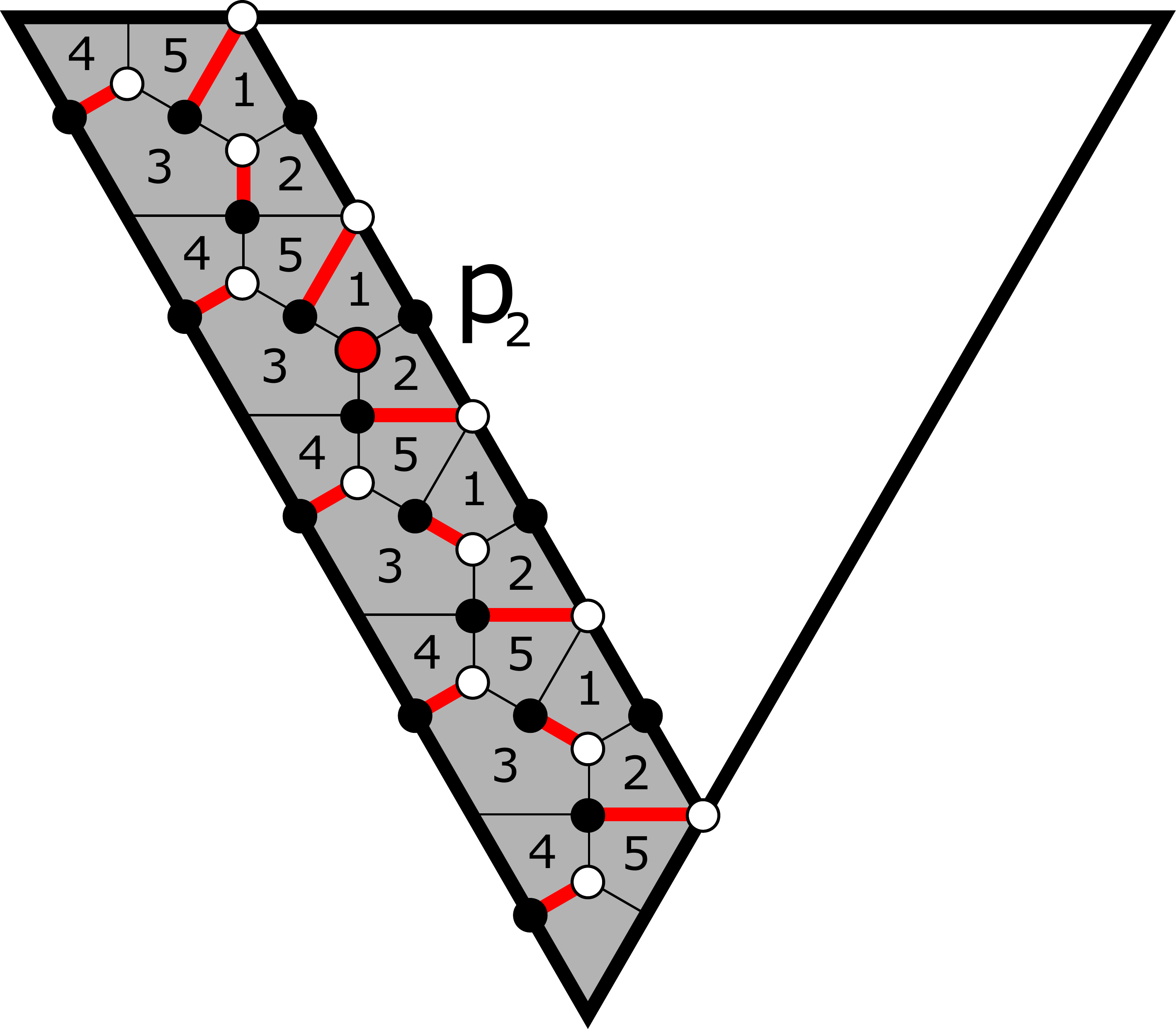}
\includegraphics[scale=0.1]{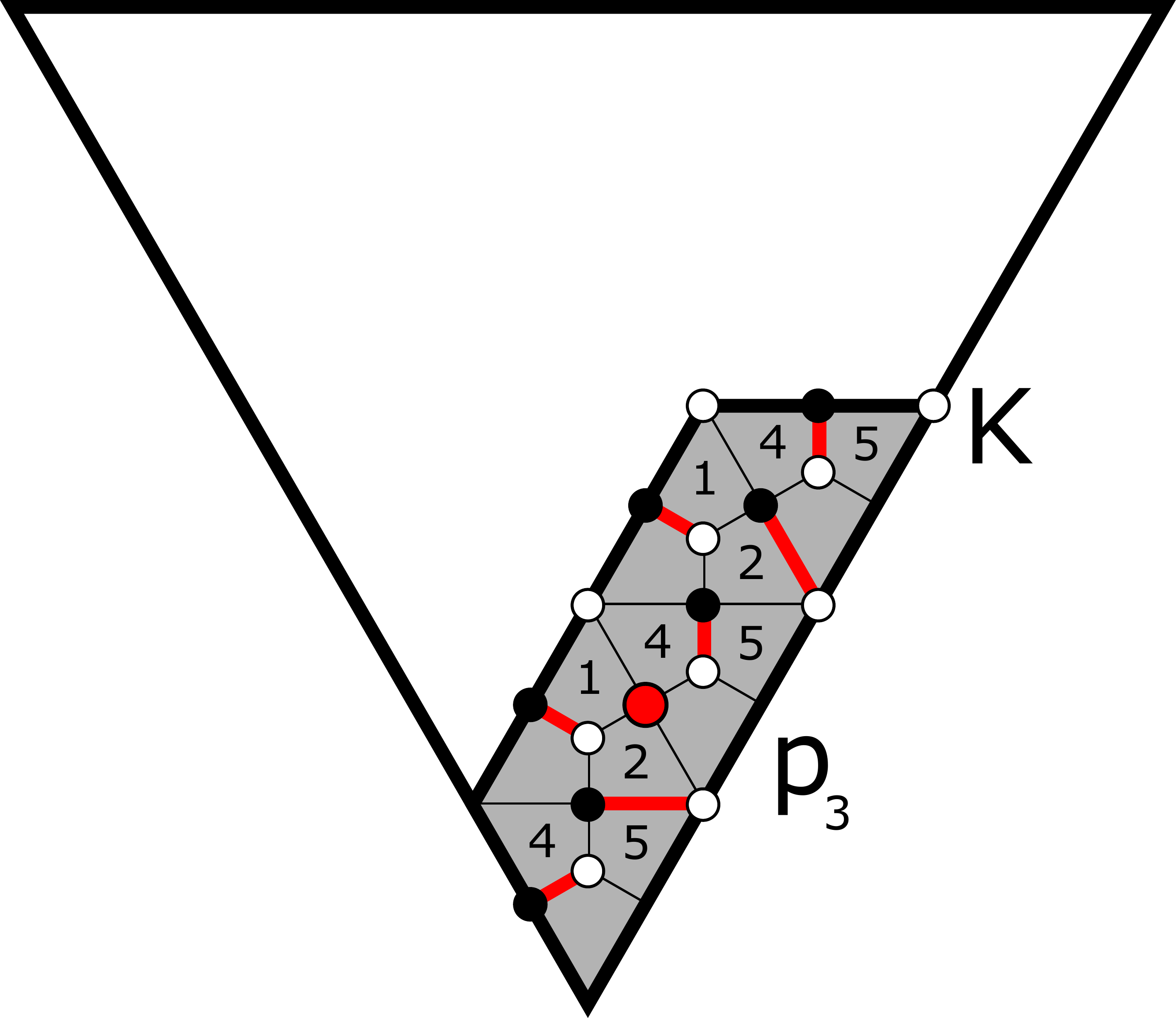}
\includegraphics[scale=0.1]{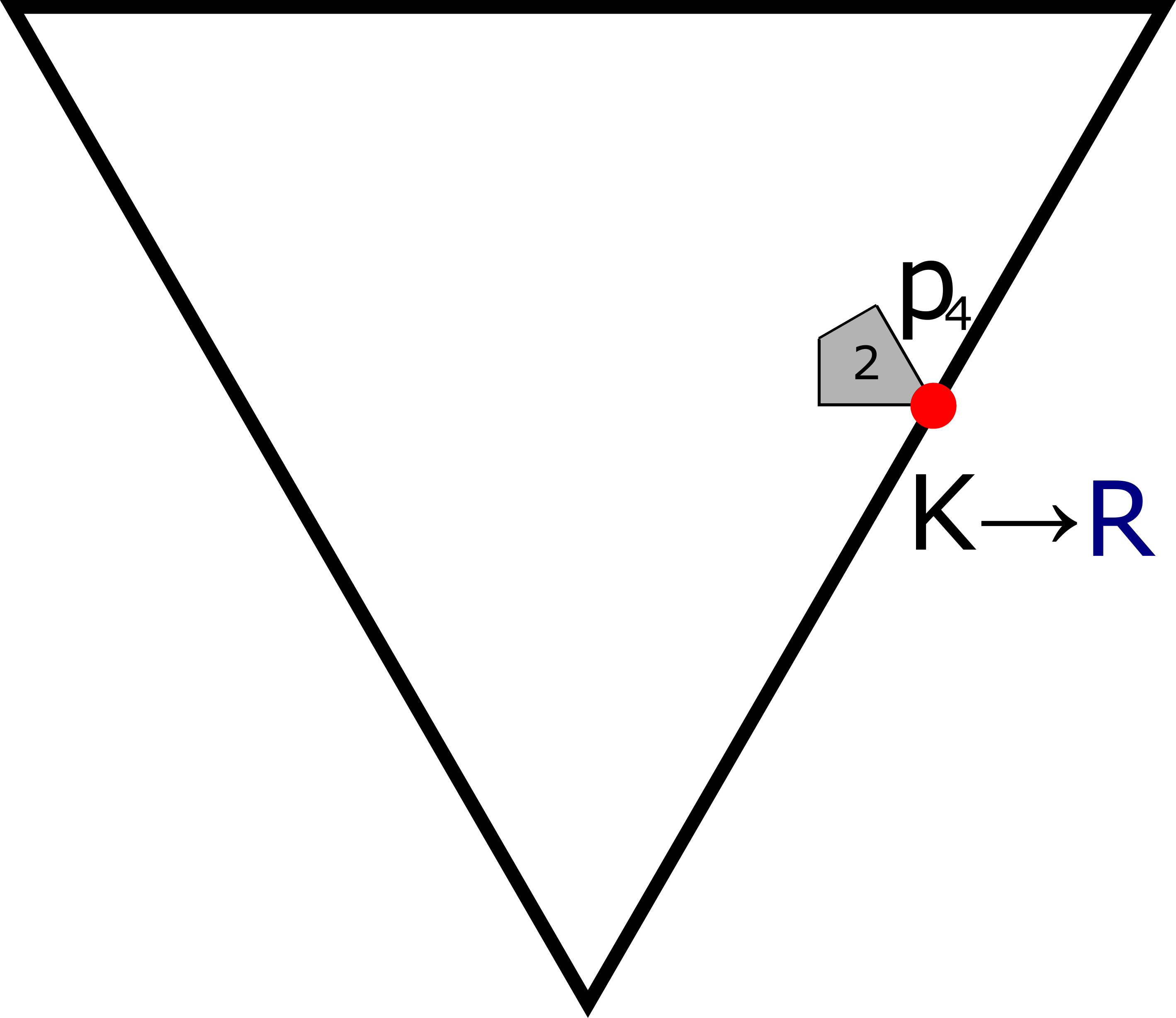}
\caption{Effects of removing points for $x_{2k-1}$, $k$ even.}
\label{fig:effect_base1_K}
\end{figure}

\begin{figure}[h!]
\includegraphics[scale=0.1]{somos_case1_remove_p1.pdf}
\includegraphics[scale=0.1]{somos_case1_remove_p2.pdf}
\includegraphics[scale=0.1]{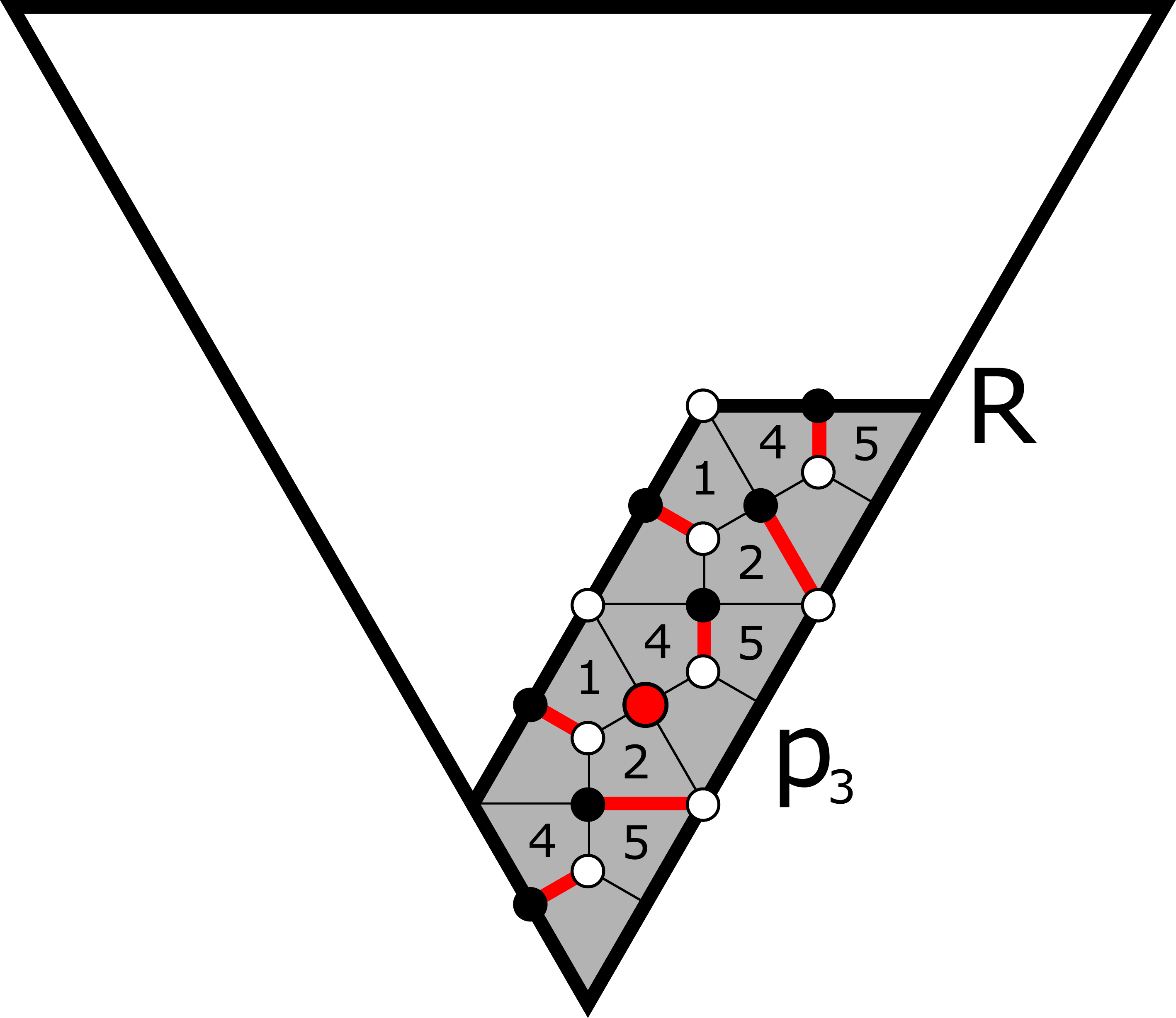}
\includegraphics[scale=0.1]{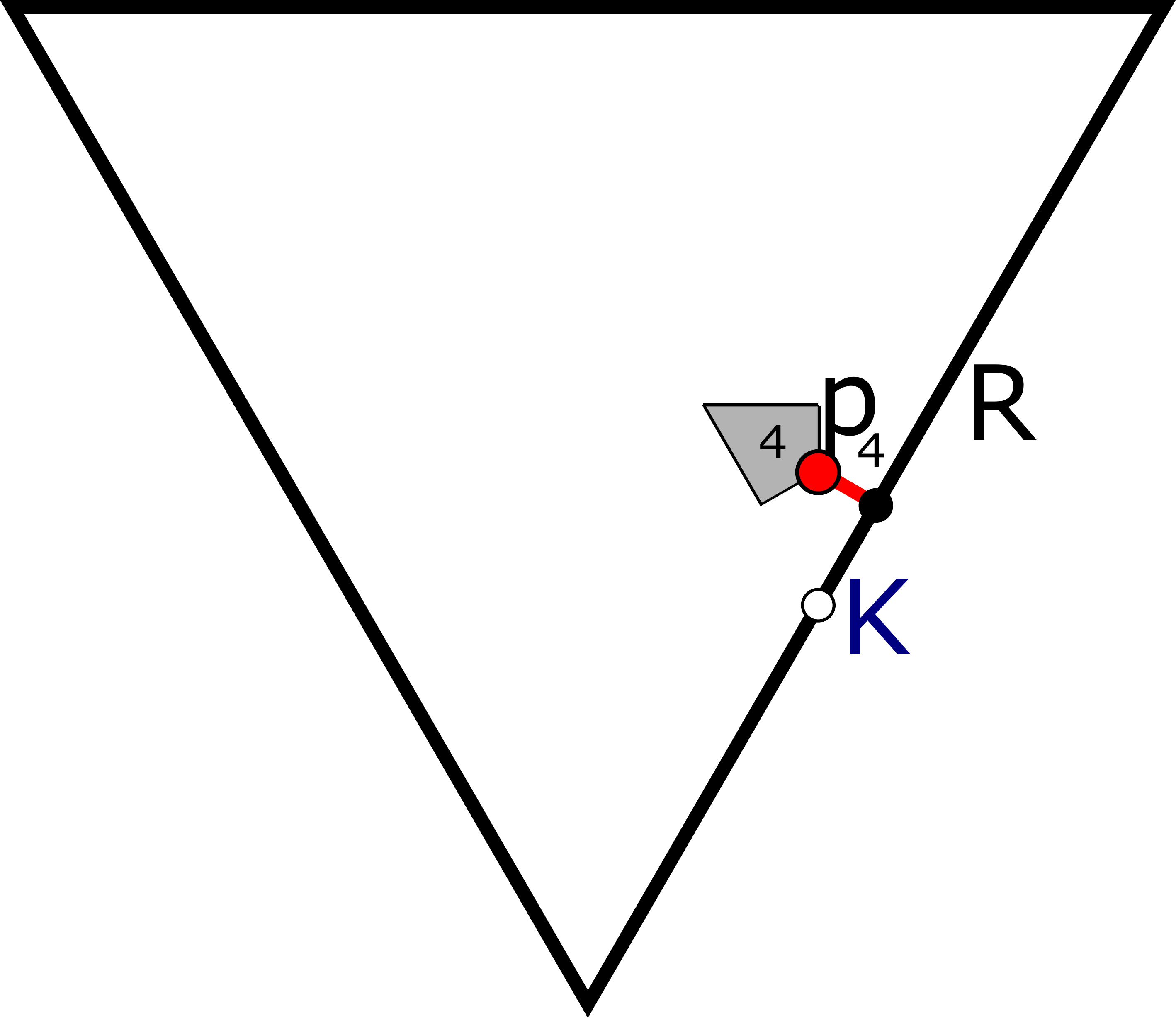}
\caption{Effects of removing points for $x_{2k-1}$, $k$ odd.}
\label{fig:effect_base1_R}
\end{figure}

Below, we explicitly write down the contour $C_i$ satisfying $\hat{\cal{G}(C)-S_i}=\hat{\cal{G}}(C_i)$ for each $S_i$, with the corresponding cluster variable, followed from induction hypothesis. Here, we omit the details of doing explicit calculation of adding and subtracting 1's.

\noindent\textbf{Subcase 1:} $k$ is even, i.e. $m\equiv3$ (mod 4). We have $C=(a,b,c,d,e)-K$. 
\begin{align*}
\hat{G-\{p_1,p_2,p_3,p_4\}}-K&= \hatcal{G}(a-3, b+2, 0-1, d-1, e+2)-R\\
&=\hatcal{G}(C_2),\ \text{graph of }x_{2k-6}\\
\hat{G-\{p_1,p_2\}}-K&= \hatcal{G}(a-2, b+1, 0, d-1, e+2)-K\\
&=\hatcal{G}(C_3),\ \text{graph of }x_{2k-5}\\
\hat{G-\{p_3,p_4\}}-K &= \hatcal{G}(a-1, b+1, c-1, d, e)-R\\
&=\hatcal{G}(C_4),\ \text{graph of }x_{2k-2}\\
\hat{G-\{p_1,p_4\}}-K
&= \hatcal{G}(a-1,b, c, d-1, e+1)-R\\ 
&=\hatcal{G}(C_5),\ \text{graph of }x_{2k-3}\\
\hat{G-\{p_2,p_3\}}-K&= \hatcal{G}(a-2,b+2, 0-1, d, e+1)-K\\
&=\hatcal{G}(C_6),\ \text{graph of }x_{2k-4}
\end{align*}

\noindent\textbf{Subcase 2:} $k$ is odd, i.e. $m\equiv1$ (mod 4). We have $C=(a,b,c,d,e)-R$. 
\begin{align*}
\hat{G-\{p_1,p_2,p_3,p_4\}}-R&= \hatcal{G}(a-3, b+3, c-1, d, e+2)-K\\
&=\hatcal{G}(C_2),\ \text{graph of }x_{2k-6}\\
\hat{G-\{p_1,p_2\}}-R&= \hatcal{G}(a-2, b+1, c, d-1, e+2)-R\\
&=\hatcal{G}(C_3),\ \text{graph of }x_{2k-5}\\
\hat{G-\{p_3,p_4\}}-R &= \hatcal{G}(a-1, b+2, c-1, d+1, e)-K\\
&=\hatcal{G}(C_4),\ \text{graph of }x_{2k-2}\\
\hat{G-\{p_1,p_4\}}-R
&= \hatcal{G}(a-1,b+1, c, d, e+1)-K\\ 
&=\hatcal{G}(C_5),\ \text{graph of }x_{2k-3}\\
\hat{G-\{p_2,p_3\}}-R&= \hatcal{G}(a-2,b+2, c-1, d, e+1)-R\\
&=\hatcal{G}(C_6),\ \text{graph of }x_{2k-4}
\end{align*}

By the Somos-5 recurrence $x_{2k-1}x_{2k-6} = x_{2k-5}x_{2k-2} + x_{2k-3}x_{2k-4}$ we conclude that $G = \hatcal{G}(C_1)$ is the graph of $x_{2k-1}$ (after verifying step 3).

\

\noindent\textbf{Step 3.} Now that we have all of the contours $C_i$, we specify the sets $T_i$ (defined in Section~\ref{sub:technique}) for a specific choice of $p_1,p_2,p_3,p_4$. 

$G-K$ (Special vertex kept): let $p_1$ be the rightmost (B) point on edge e (not in a forced matching), $p_2$ be the topmost (W) point on edge a (not in a forced matching), $p_3$ be the bottommost (B) point on edge b (in a forced matching), $p_4$ be the special vertex. See Figure~\ref{fig:cov-mon-base1-K}.
\begin{align*}
T(\emptyset) = 1, \quad T(\{p_1,p_2,p_3,p_4\}) = x_3x_3x_3x_4, \quad T(\{p_1,p_2\}) = x_3x_4, \\
T(\{p_3,p_4\}) = x_3x_3, \quad T(\{p_2,p_3\}) = x_3x_3, \quad T(\{p_1,p_4\}) = x_3x_4.
\end{align*}

\begin{figure}[h!]
\includegraphics[scale=0.11]{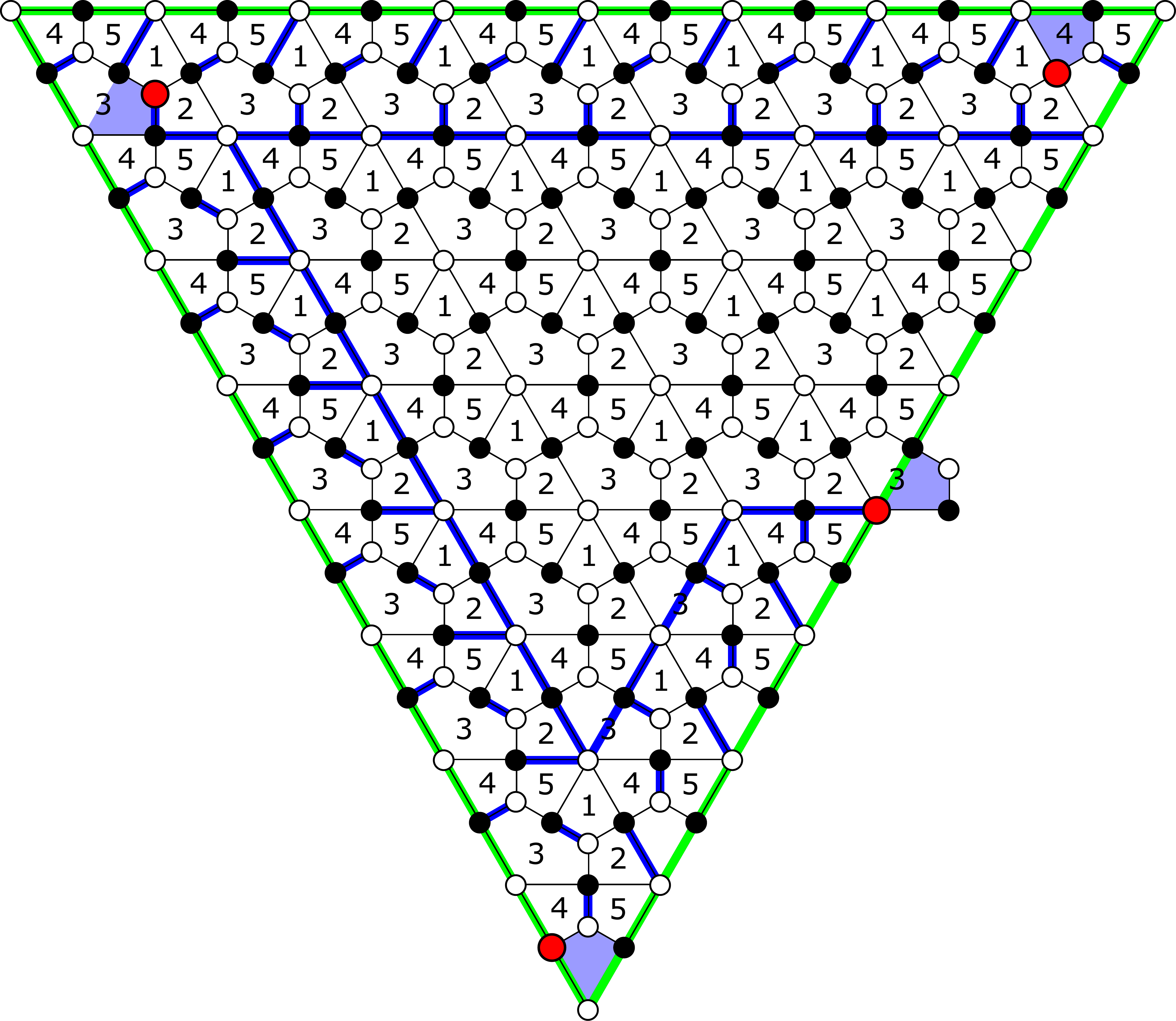}
\includegraphics[scale=0.11]{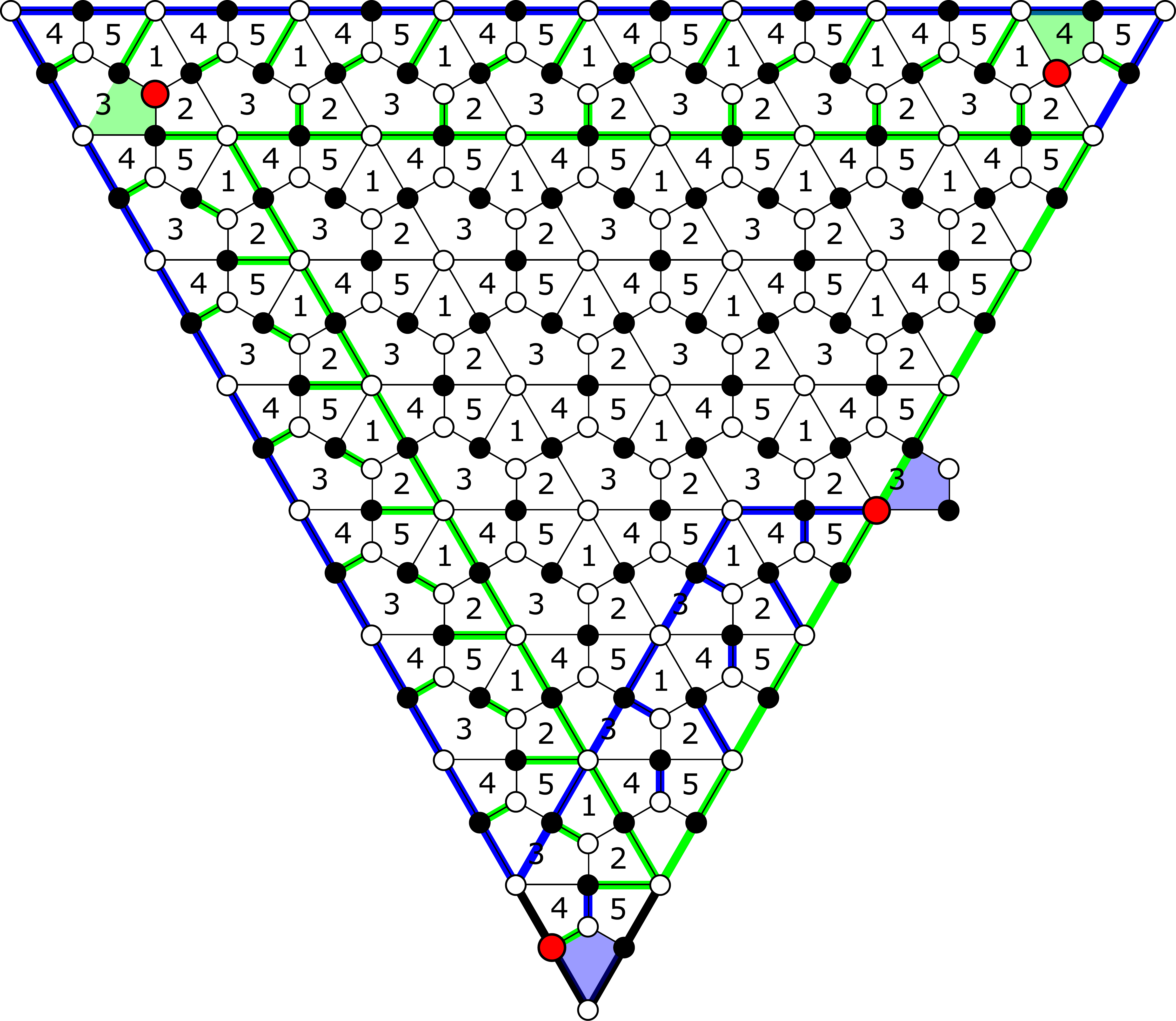}
\includegraphics[scale=0.11]{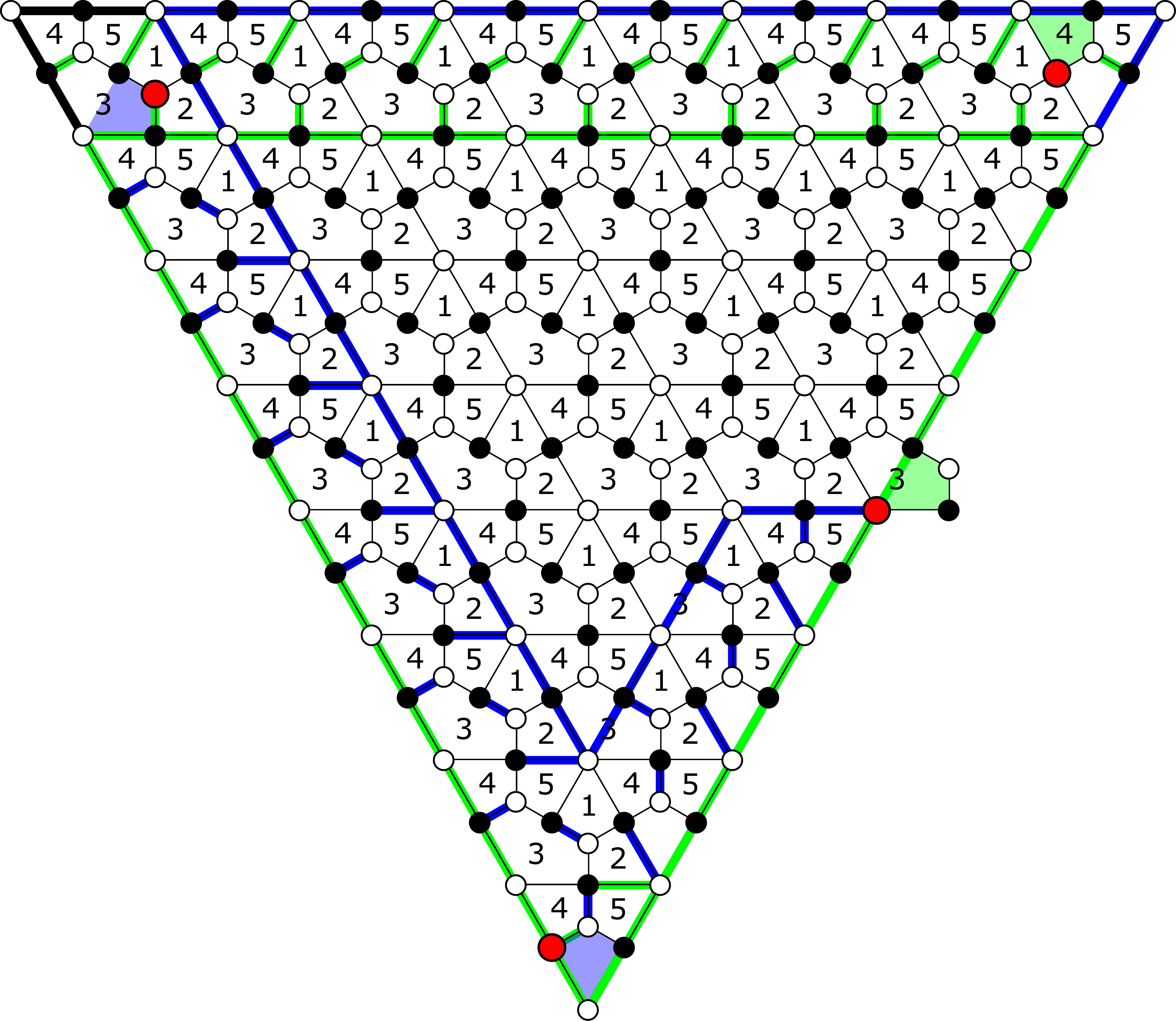}
\caption{Covering monomial for $x_{2k-1}$, $k$ even. Left: $T(\emptyset)$ and $T(\{p_1,p_2,p_3,p_4\})$. Middle: $T(\{p_1,p_2\})$ and $T(\{p_3,p_4\})$. Right $T(\{p_2,p_3\})$ and $T(\{p_1,p_4\})$.}
\label{fig:cov-mon-base1-K}
\end{figure}

$G-R$ (Special vertex removed): let $p_1$ be the rightmost (B) point on edge e (not in a forced matching), $p_2$ be the topmost (W) point on edge a (not in a forced matching), $p_3$ be the bottommost (B) point on edge b (in a forced matching), $p_4$ be the other white vertex on the 5-block below the special vertex. See Figure~\ref{fig:cov-mon-base1-R}.
\begin{align*}
T(\emptyset) = 1, \quad T(\{p_1,p_2,p_3,p_4\}) = x_3x_3x_3x_4, \quad T(\{p_1,p_2\}) = x_3x_4, \\
T(\{p_3,p_4\}) = x_3x_5, \quad T(\{p_2,p_3\}) = x_3x_3, \quad T(\{p_1,p_4\}) = x_4x_4.
\end{align*}

\begin{figure}[h!]
\includegraphics[scale=0.11]{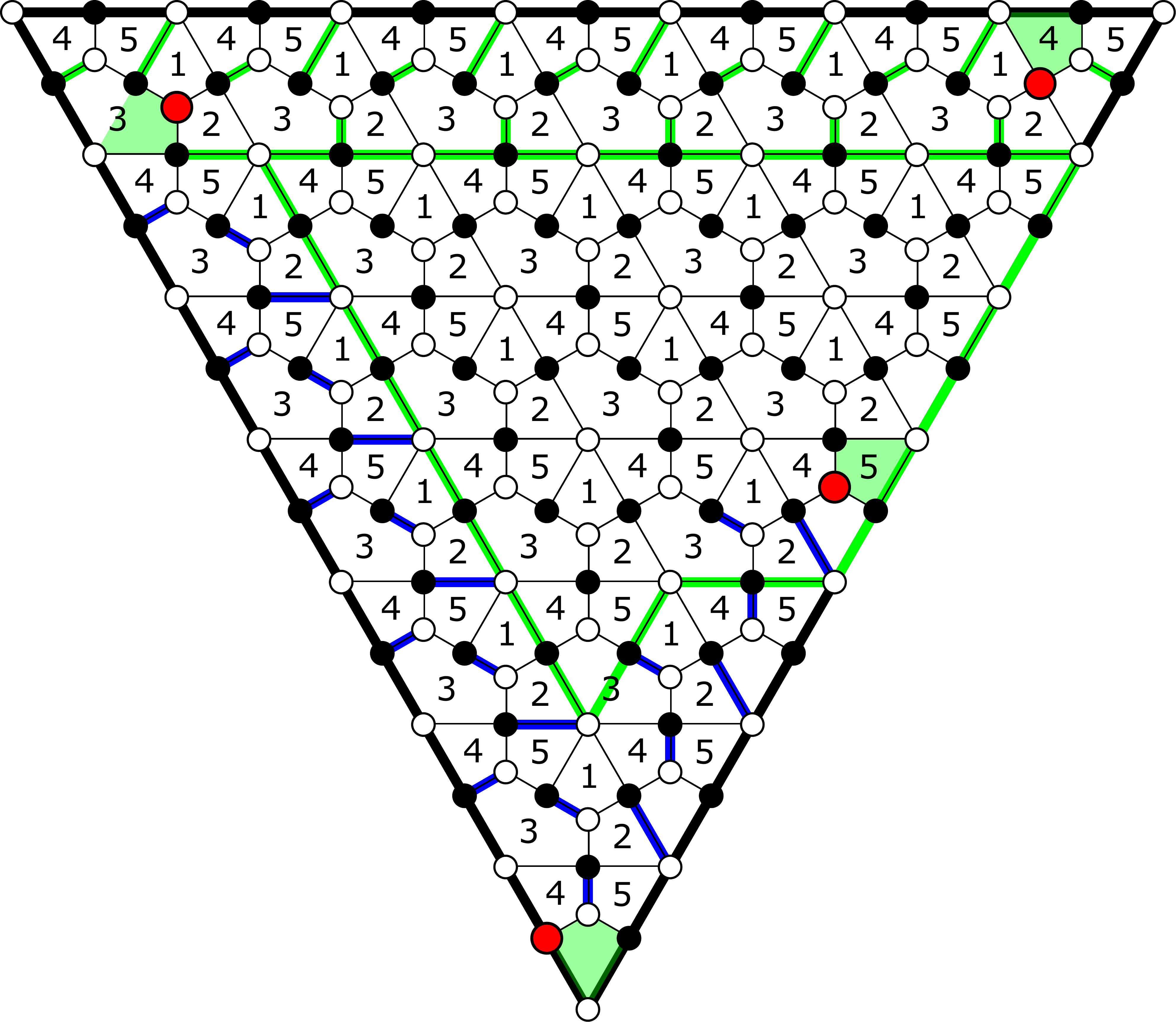}
\includegraphics[scale=0.11]{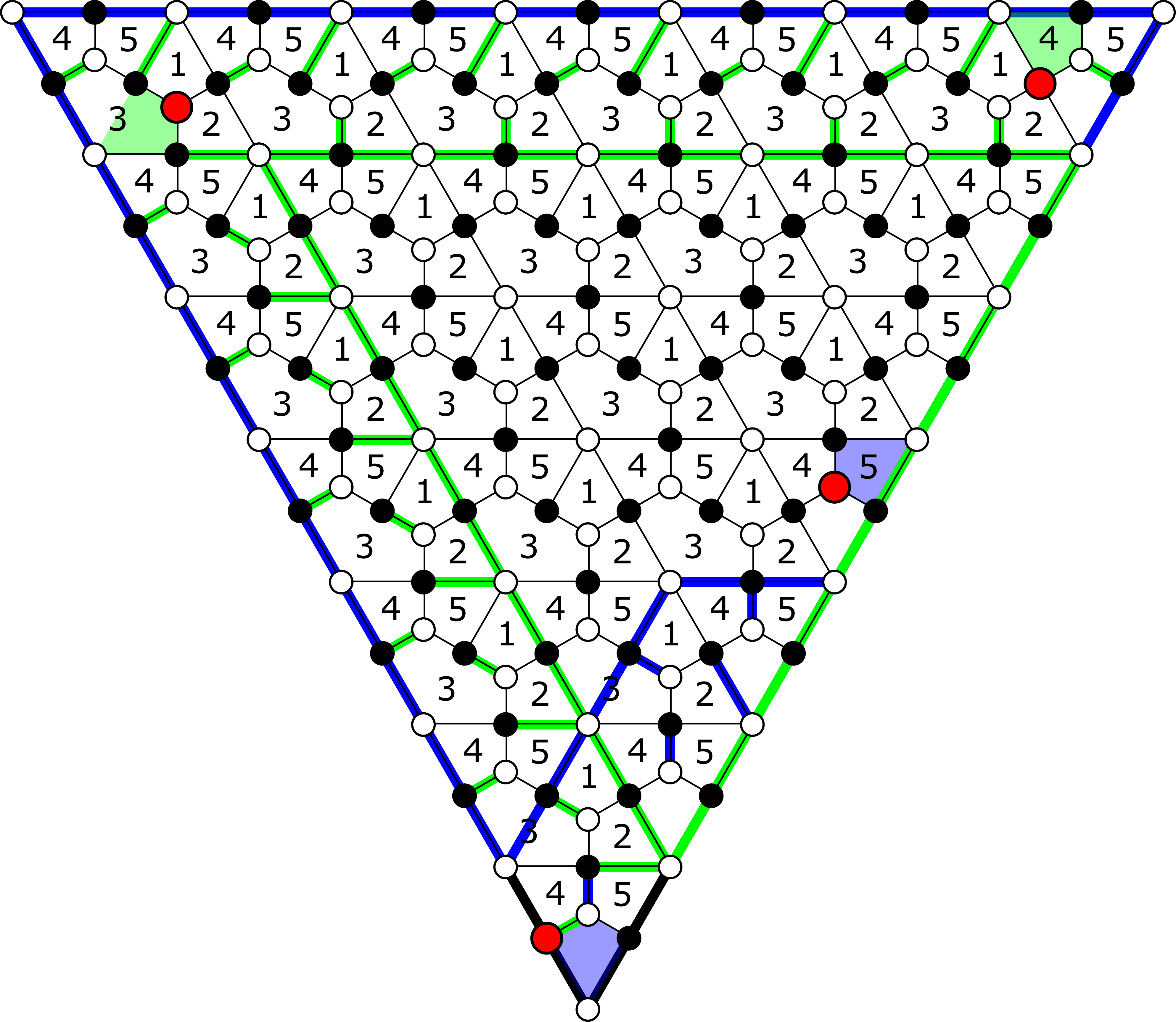}
\includegraphics[scale=0.11]{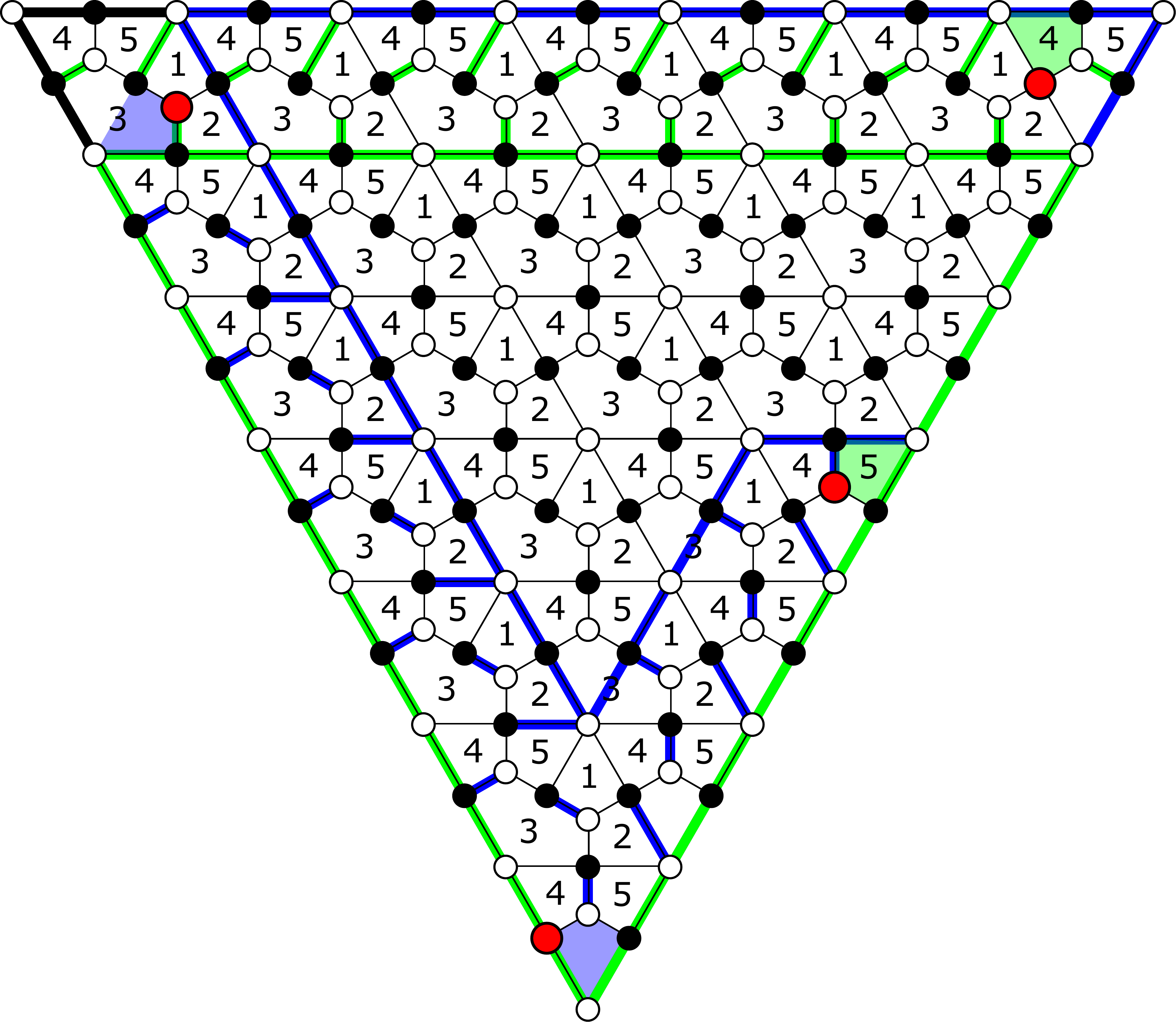}
\caption{Covering monomial for $x_{2k-1}$, $k$ even. Left: $T(\emptyset)$ and $T(\{p_1,p_2,p_3,p_4\})$. Middle: $T(\{p_1,p_2\})$ and $T(\{p_3,p_4\})$. Right $T(\{p_2,p_3\})$ and $T(\{p_1,p_4\})$.}
\label{fig:cov-mon-base1-R}
\end{figure}


We see that equation~\ref{eqn:T_i} is satisfied: $$\left(\prod_{j\in T_1}x_j\right)\left(\prod_{j\in T_2}x_j\right)=\left(\prod_{j\in T_3}x_j\right)\left(\prod_{j\in T_4}x_j\right)=\left(\prod_{j\in T_5}x_j\right)\left(\prod_{j\in T_6}x_j\right).$$

Finally, we conclude that $c(\hatcal{G}(C_1))$ is the Laurent polynomial of $x_{2k-1}$, as desired.

\

\noindent\textbf{Case 2:} $m = 2k$. Consider the following contour
\begin{align*}
C=(a,b,c,d,e)
=\left(k-2, -\left\lceil \frac{k-4}{2} \right\rceil, -1, \left\lfloor \frac{k}{2} \right\rfloor, 1-k\right).
\end{align*}
Since $k \geq 3$, we have $a > 0, b \ge 0, d>0$ and $e < 0$. The proof is similar to the first case. In Step 1 we again use balanced Kuo's condensation on $\cal{G}(C)$ and use the same notation for each $S_i$. In Step 2 we define the four points as follows.
\begin{itemize}
\item Let $p_1$ be any white point on edge $a$.
\item Let $p_2$ be any black point on edge $e$.
\item Let $p_3$ be any white point on edge $d$.
\item Let $p_4$ be a black point near edge $c$ on edge $d$ defined as follows:
\begin{itemize}
\item If $k \equiv 0$ (mod 2), then $a\equiv 0$ (mod 2) so the special point is kept.  Let $p_4$ be the black point on the edge between the 4-block and 5-block above the special point.
\item If $k \equiv 1$ (mod 2), then $a\not \equiv 0$ (mod 2) so the special point is removed.  Let $p_4$ be the lowest black point on edge $d$.
\end{itemize}
\end{itemize}

We also give the effects of removing each point separately:
\begin{itemize}
\item The effect of removing $p_1$ is $(a,b,c,d,e)\rightarrow (a-1, b+1, c, d, e+1)$.
\item The effect of removing $p_2$ is $(a,b,c,d,e) \rightarrow (a-1, b, c, d-1, e+1)$.
\item The effect of removing $p_3$ is $(a,b,c,d,e)-K\rightarrow(a,b-1,c+1,d-1,e+1)-R$ and $(a,b,c,d,e)-R\rightarrow(a,b,c+1,d,e+1)-K$ depending on the parity of $k$. 
\item The effect of removing $p_4$ is $(a,b,c,d,e)-K \rightarrow (a,b-1,c,d-1,e)-R$ and $(a,b,c,d,e)-R \rightarrow (a,b,c,d,e)-K$ depending on the parity of $k$.
\end{itemize}

The position of each point and the effect of removing each point is shown in Figure~\ref{fig:effect_base2_K} (special point kept) and Figure~\ref{fig:effect_base2_R} (special point removed).

\begin{figure}[h!]
\includegraphics[scale=0.1]{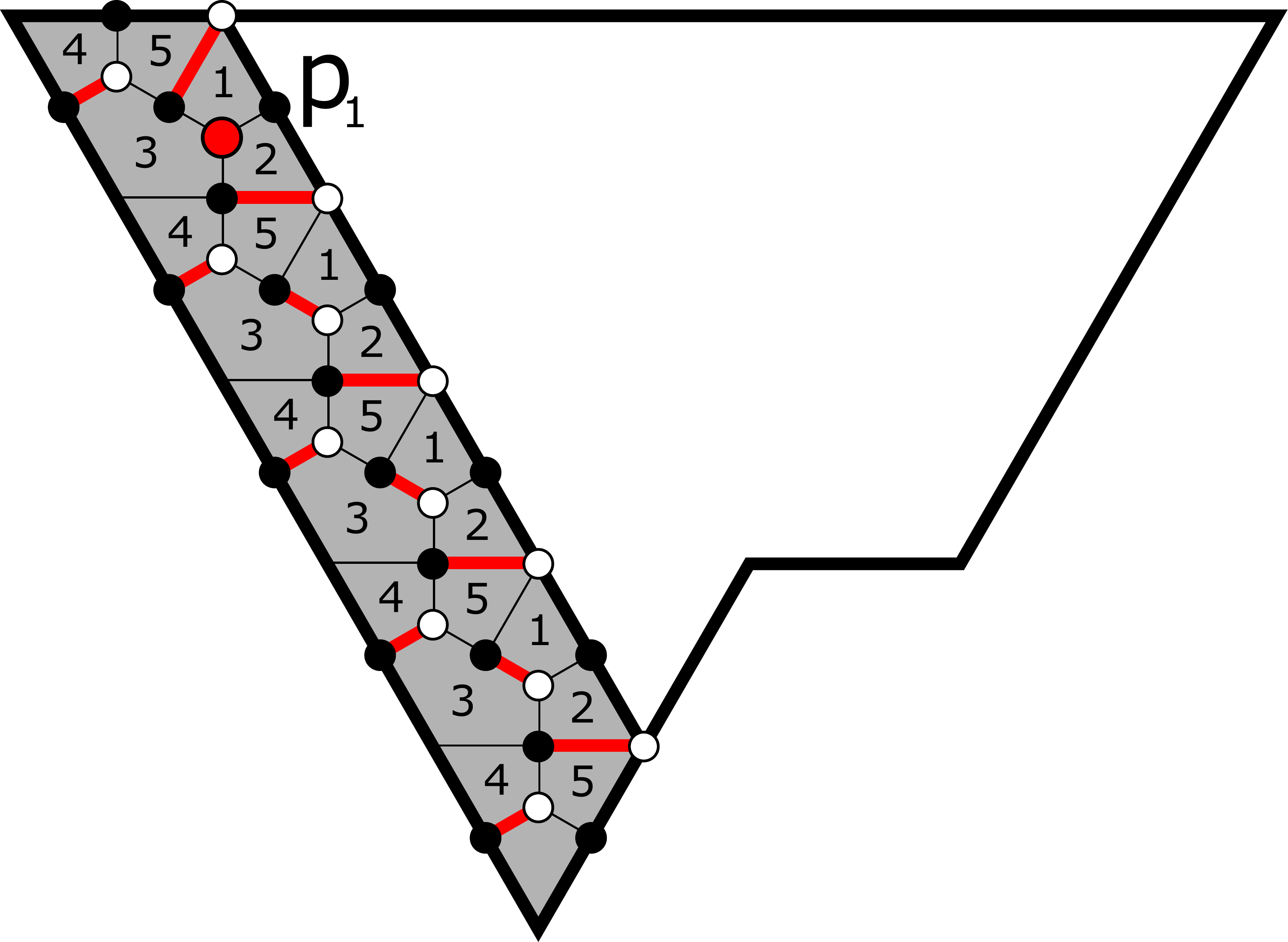}
\includegraphics[scale=0.1]{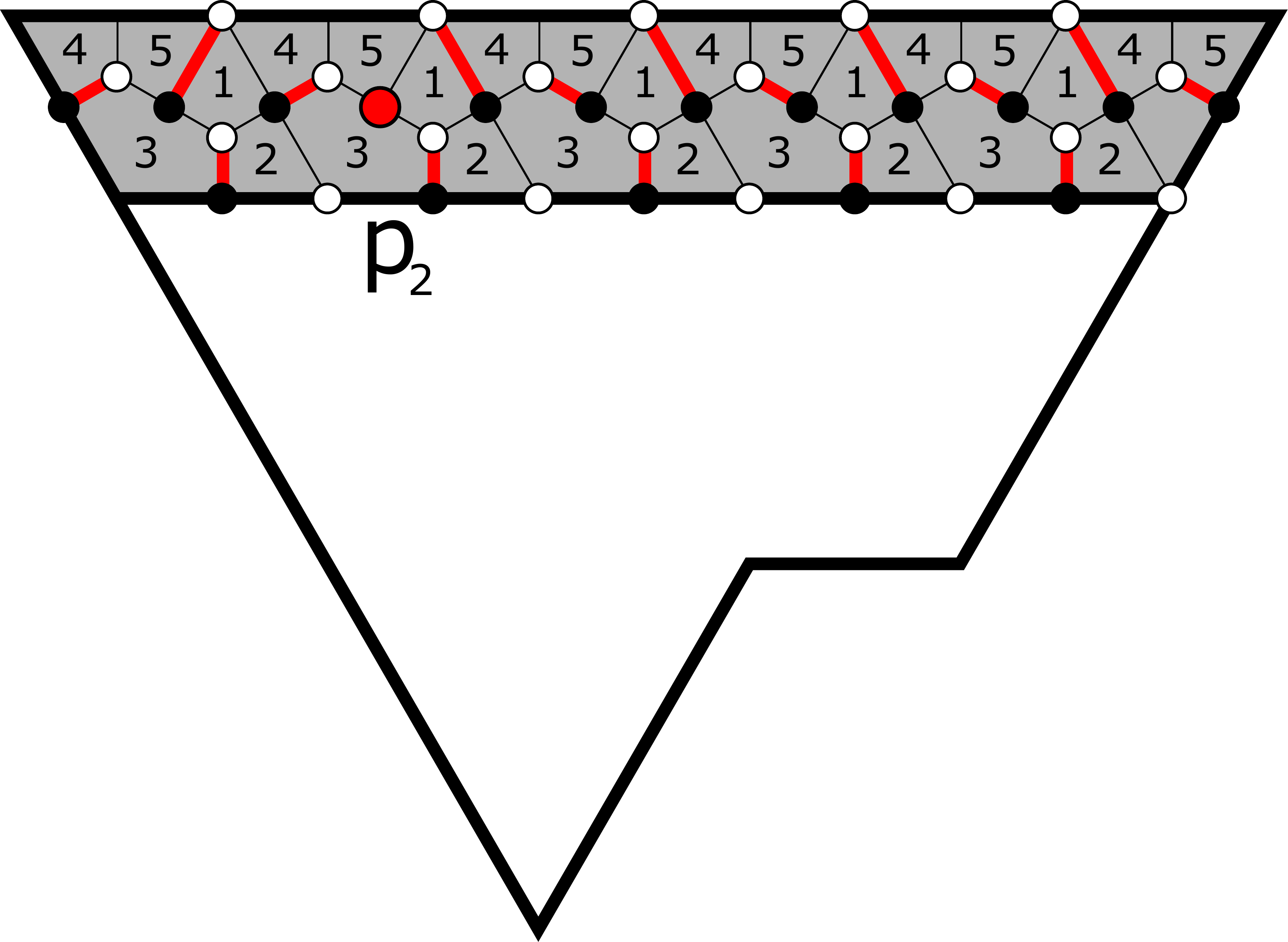}
\includegraphics[scale=0.1]{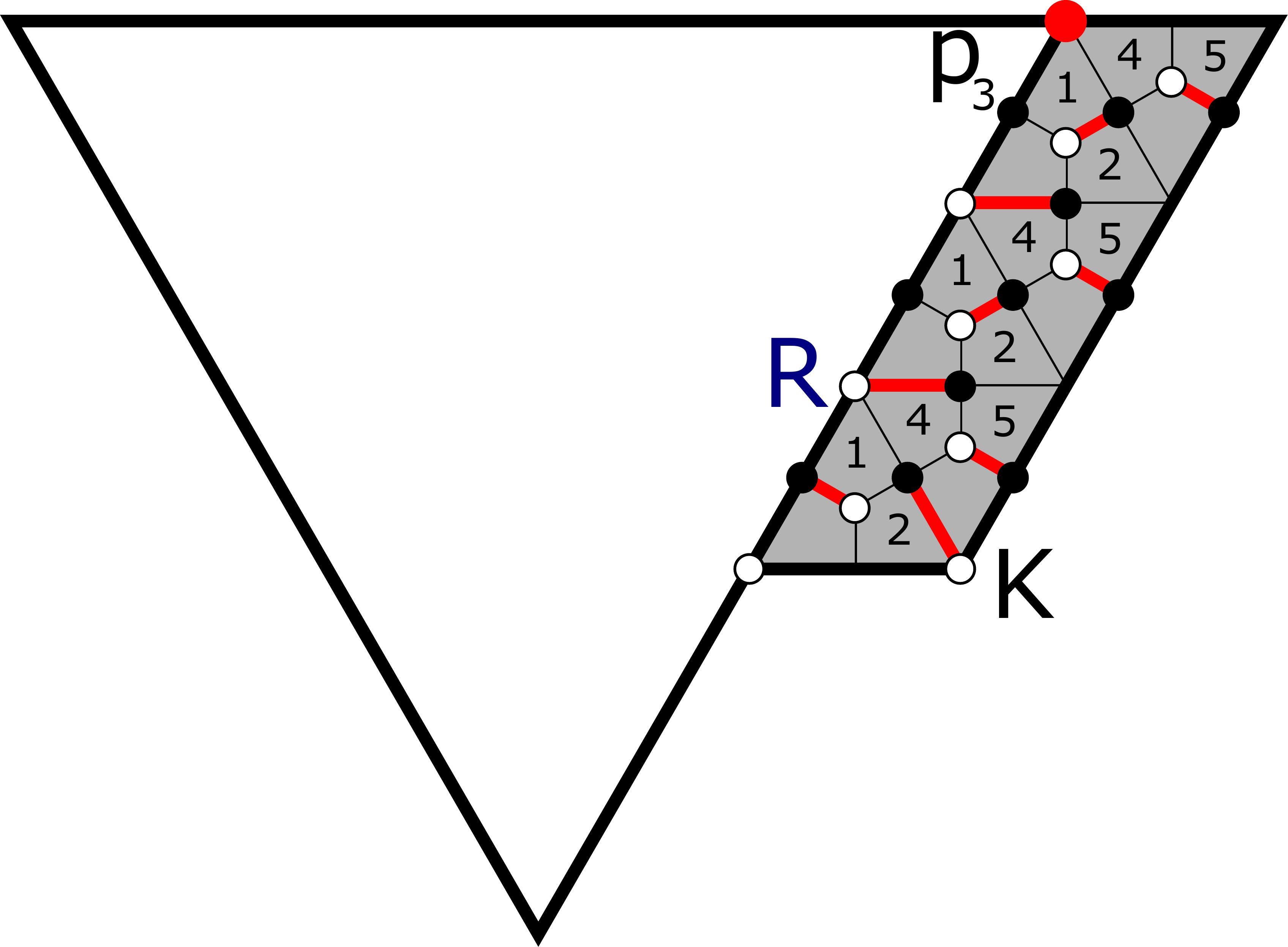}
\includegraphics[scale=0.1]{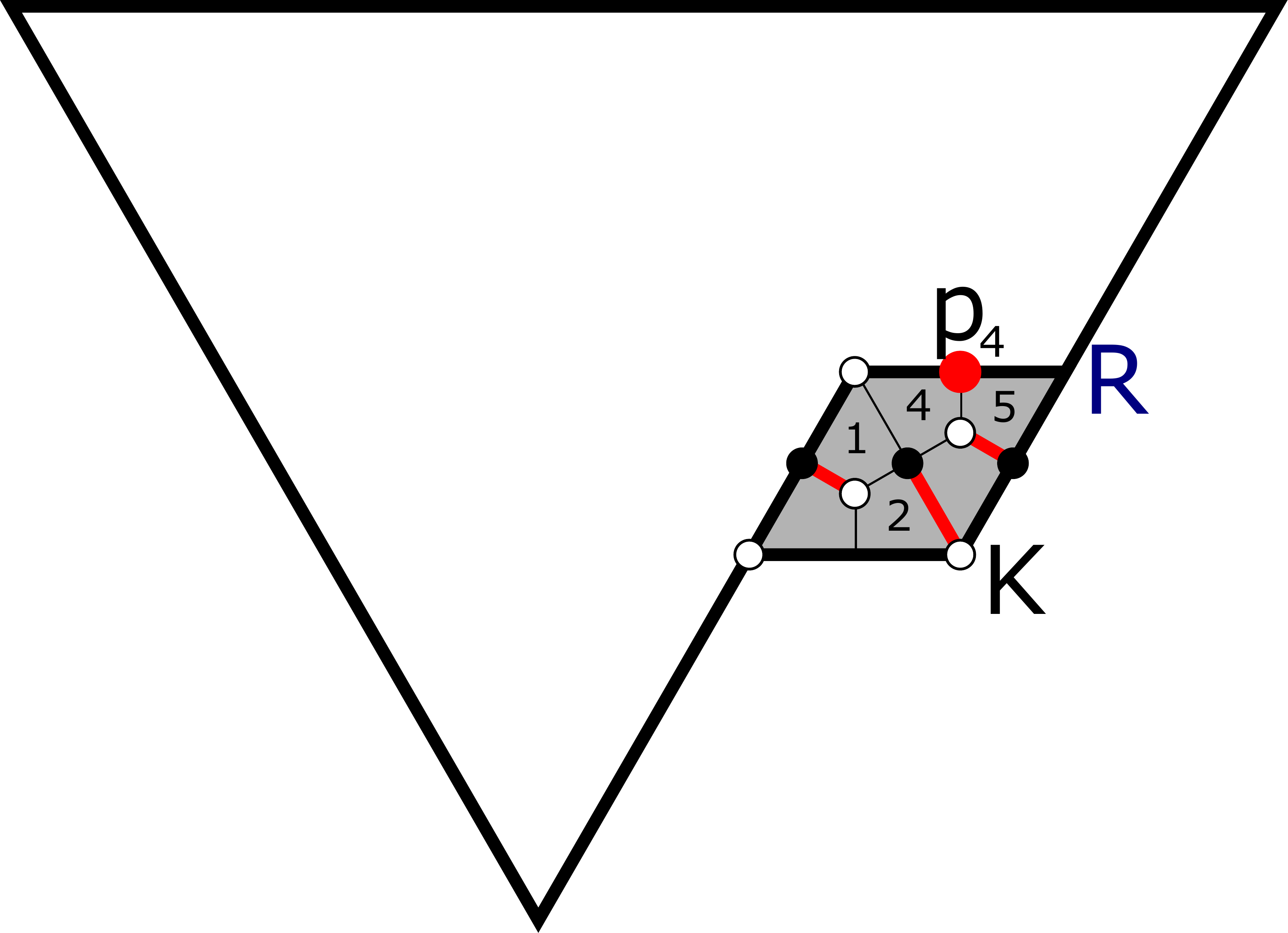}
\caption{Effects of removing points for $x_{2k}$, $k$ even.}
\label{fig:effect_base2_K}
\end{figure}

\begin{figure}[h!]
\includegraphics[scale=0.1]{somos_case2_remove_p1.pdf}
\includegraphics[scale=0.1]{somos_case2_remove_p2.pdf}
\includegraphics[scale=0.1]{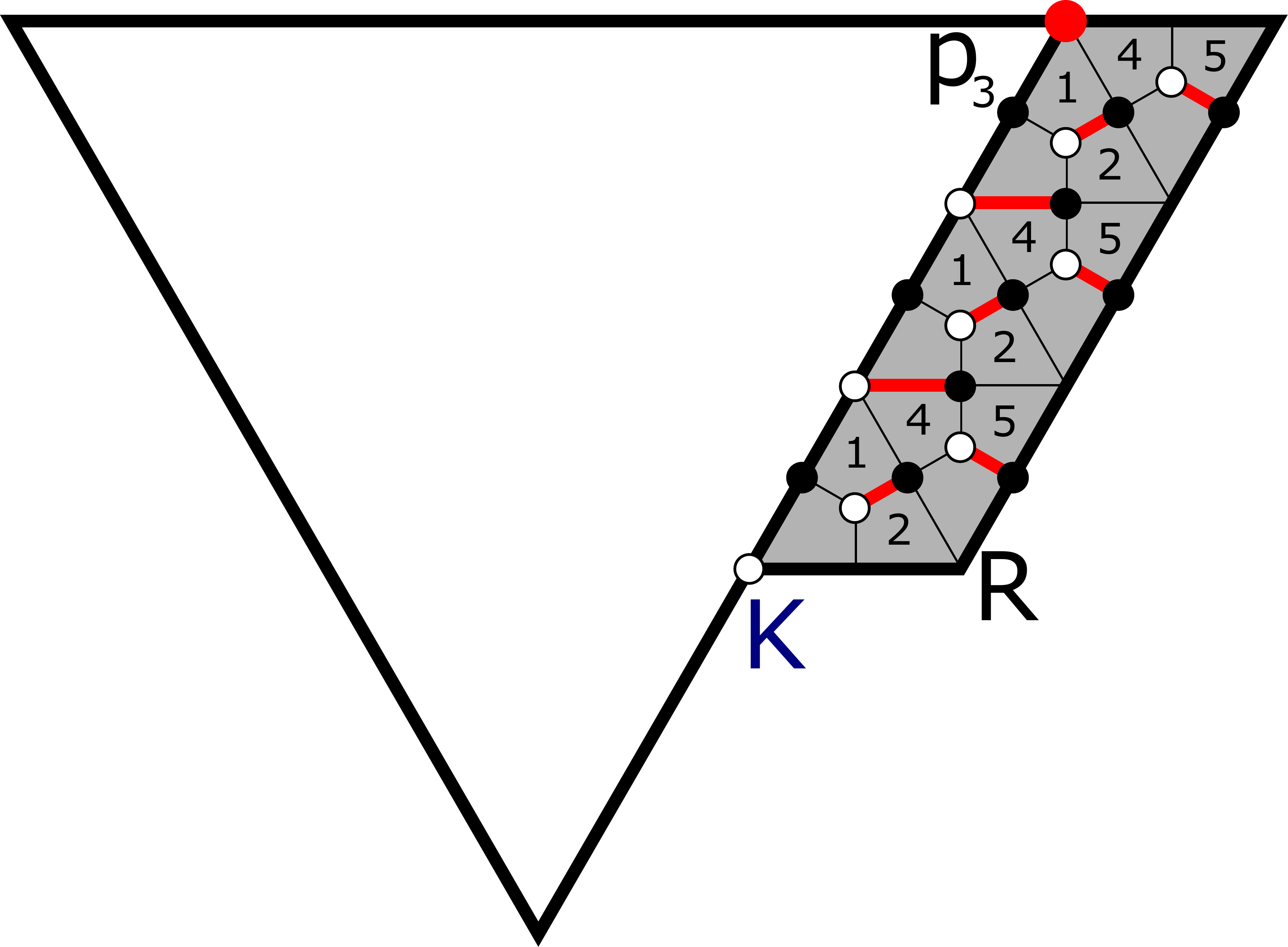}
\includegraphics[scale=0.1]{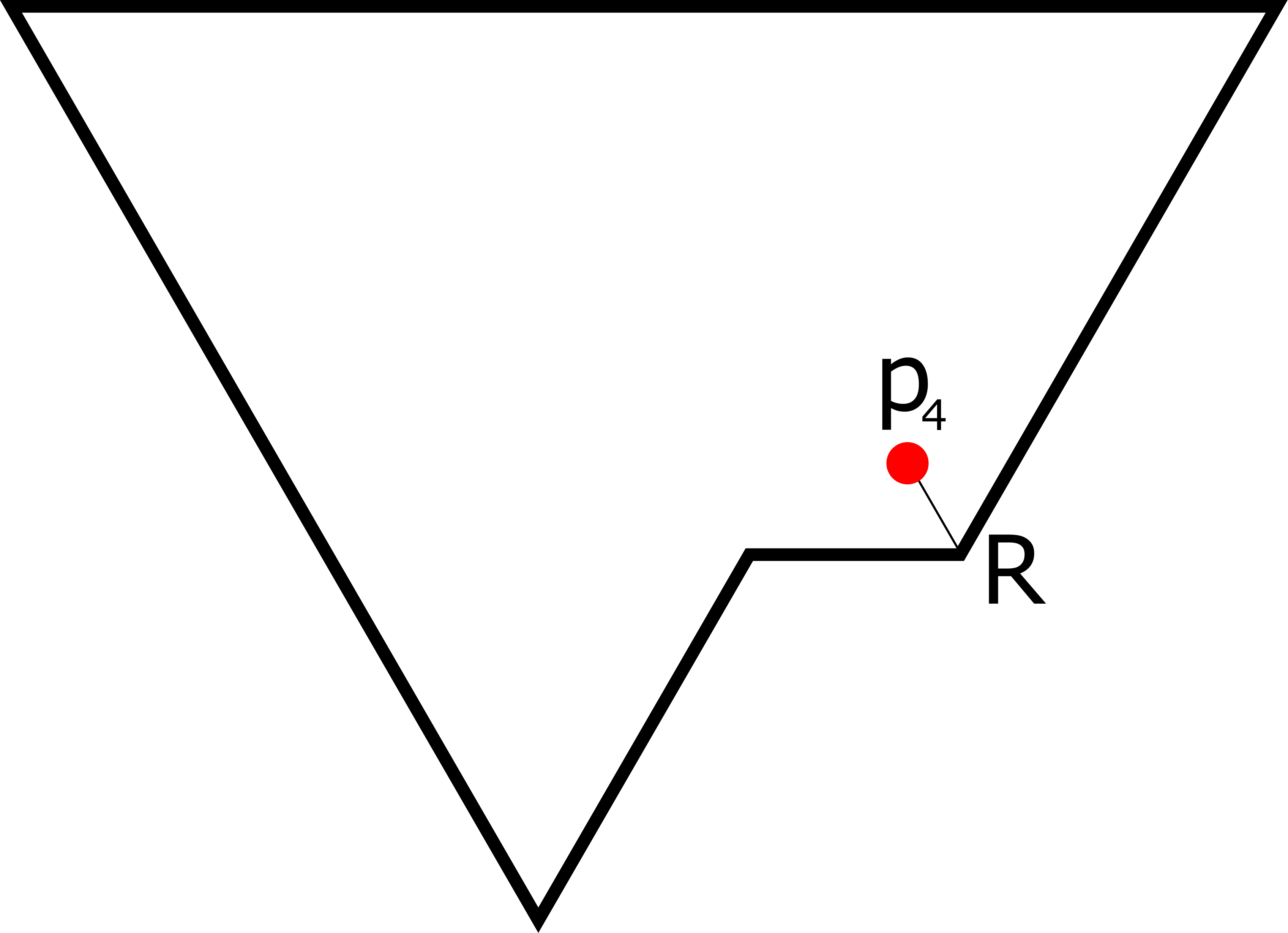}
\caption{Effects of removing points for $x_{2k}$, $k$ odd.}
\label{fig:effect_base2_R}
\end{figure}

Below, we explicitly write down the contour $C_i$ satisfying $\hat{\cal{G}(C)-S_i}=\hat{\cal{G}}(C_i)$ for each $S_i$, with the corresponding cluster variable (after verifying step 3).

\noindent\textbf{Subcase 1:} $k$ is even, i.e. $m\equiv0$ (mod 4). We have $C=(a,b,c,d,e)-K$. 
\begin{align*}
\hat{G-\{p_1,p_2,p_3,p_4\}}-K &= \hatcal{G}(a-2, b, c+1, d-2, e+3)-K\\ 
&=\hatcal{G}(C_2),\text{ graph of }x_{2k-5}\\
\hat{G-\{p_1,p_2\}}-K &= \hatcal{G}(a-2, b+1, c, d-1, e+2)-K\\
&=\hatcal{G}(C_3),\text{ graph of }x_{2k-4}\\
\hat{G-\{p_3,p_4\}}-K &= \hatcal{G}(a, b-1, c+1, d-1, e+1)-K\\
&=\hatcal{G}(C_4),\text{ graph of }x_{2k-1}\\
\hat{G-\{p_1,p_4\}}-K &= \hatcal{G}(a-1,b-1, c+1, d-2, e+2)-R\\
&=\hatcal{G}(C_5),\text{ graph of }x_{2k-3}\\
\hat{G-\{p_2,p_3\}}-K &= \hatcal{G}(a-1,b, c, d-1, e+1)-R\\
&=\hatcal{G}(C_6),\text{ graph of }x_{2k-2}
\end{align*}

\noindent\textbf{Subcase 2:} $k$ is odd, i.e. $m\equiv2$ (mod 4). We have $C=(a,b,c,d,e)-R$. 
\begin{align*}
\hat{G-\{p_1,p_2,p_3,p_4\}}-R &= \hatcal{G}(a-2, b, c+1, d-2, e+3)-R\\ 
&=\hatcal{G}(C_2),\text{ graph of }x_{2k-5}\\
\hat{G-\{p_1,p_2\}}-R &= \hatcal{G}(a-2, b+1, c, d-1, e+2)-R\\
&=\hatcal{G}(C_3),\text{ graph of }x_{2k-4}\\
\hat{G-\{p_3,p_4\}}-R &= \hatcal{G}(a, b-1, c+1, d-1, e+1)-R\\
&=\hatcal{G}(C_4),\text{ graph of }x_{2k-1}\\
\hat{G-\{p_1,p_4\}}-R &= \hatcal{G}(a-1,b, c+1, d-1, e+2)-K\\
&=\hatcal{G}(C_5),\text{ graph of }x_{2k-3}\\
\hat{G-\{p_2,p_3\}}-R &= \hatcal{G}(a-1,b+1, c, d, e+1)-K\\
&=\hatcal{G}(C_6),\text{ graph of }x_{2k-2}
\end{align*}

By the Somos-5 recurrence $x_{2k}x_{2k-5} = x_{2k-4}x_{2k-1} + x_{2k-2}x_{2k-3}$ we conclude that $\hatcal{G}(C_1)$ is the graph of $x_{2k}$.

In Step 3 we specify the sets $T_i$ and verify equation~\ref{eqn:T_i}.

$G-K$ (Special vertex kept): let $p_1$ be the bottommost (W) point on edge a (not in a forced matching), $p_2$ be the leftmost (B) point on edge e (not in a forced matching), $p_3$ be the topmost (W) point on edge d, $p_4$ be the (B) point on the edge between the 4-block and 5-block above the special vertex. See Figure~\ref{fig:cov-mon-base2-K}.
\begin{align*}
T(\emptyset) = 1, \quad T(\{p_1,p_2,p_3,p_4\}) = x_3x_4x_4x_5, \quad T(\{p_1,p_2\}) = x_3x_5, \\
T(\{p_3,p_4\}) = x_4x_4, \quad T(\{p_2,p_3\}) = x_4x_5, \quad T(\{p_1,p_4\}) = x_3x_4.
\end{align*}

\begin{figure}[h!]
\includegraphics[scale=0.1]{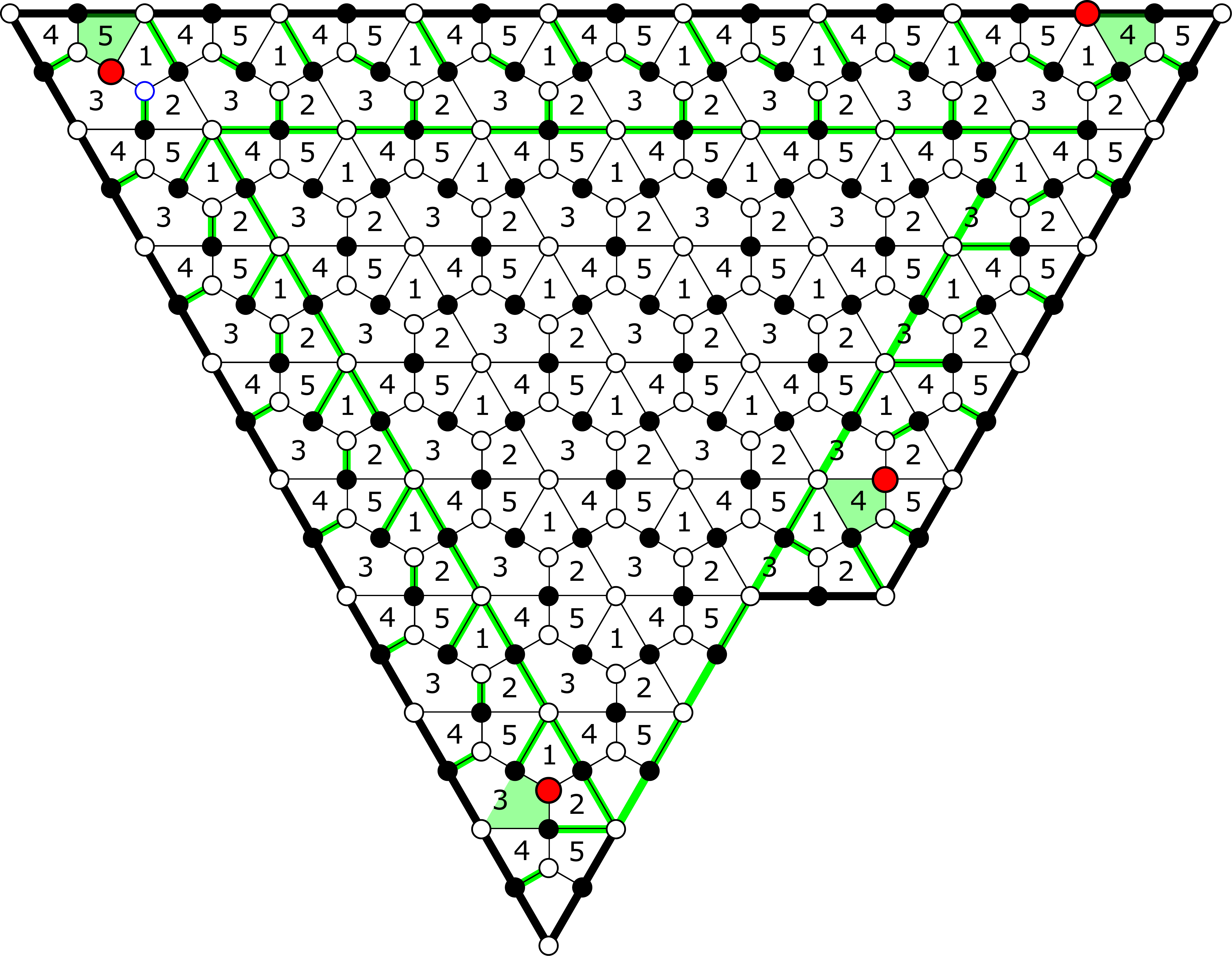}
\includegraphics[scale=0.1]{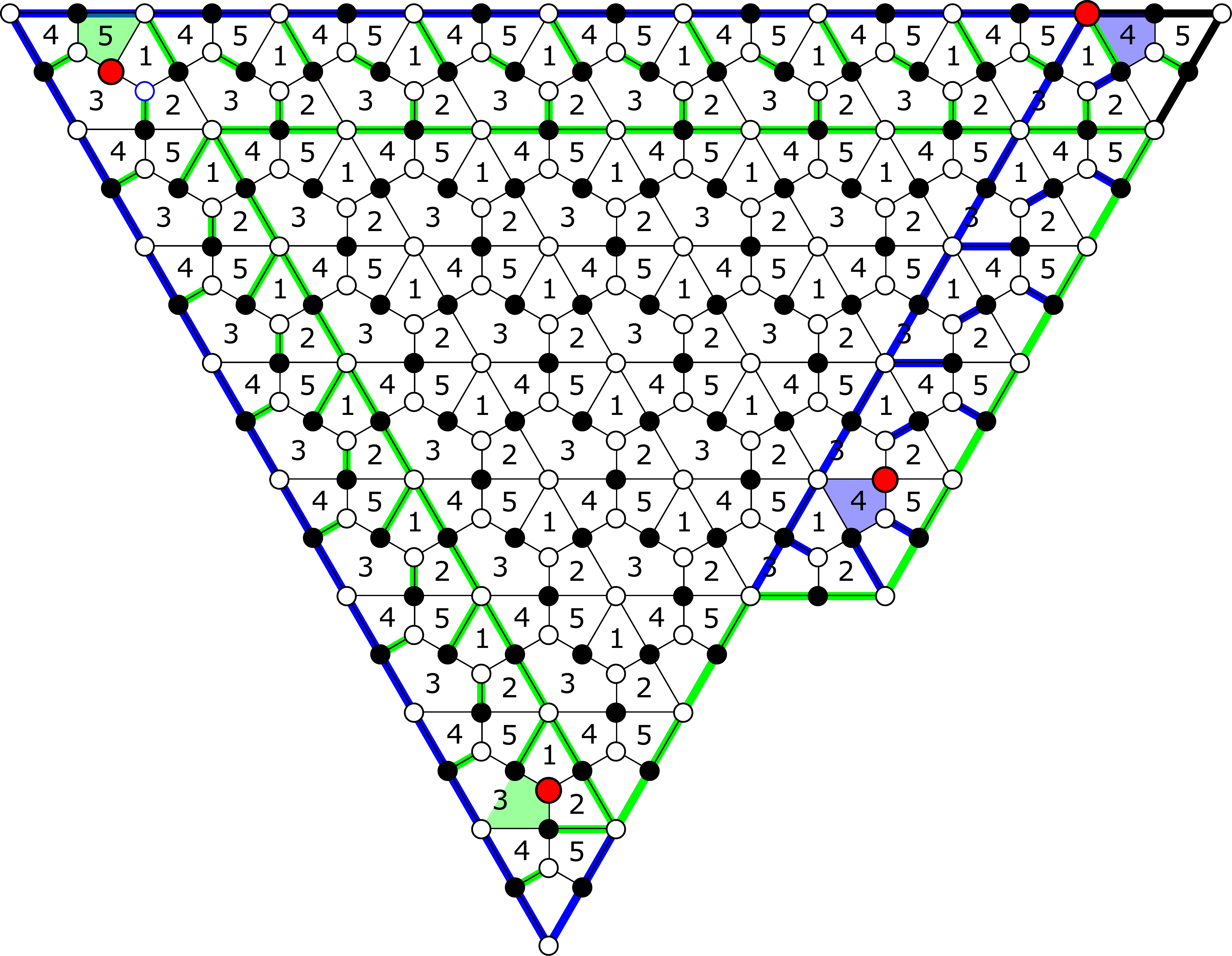}
\includegraphics[scale=0.1]{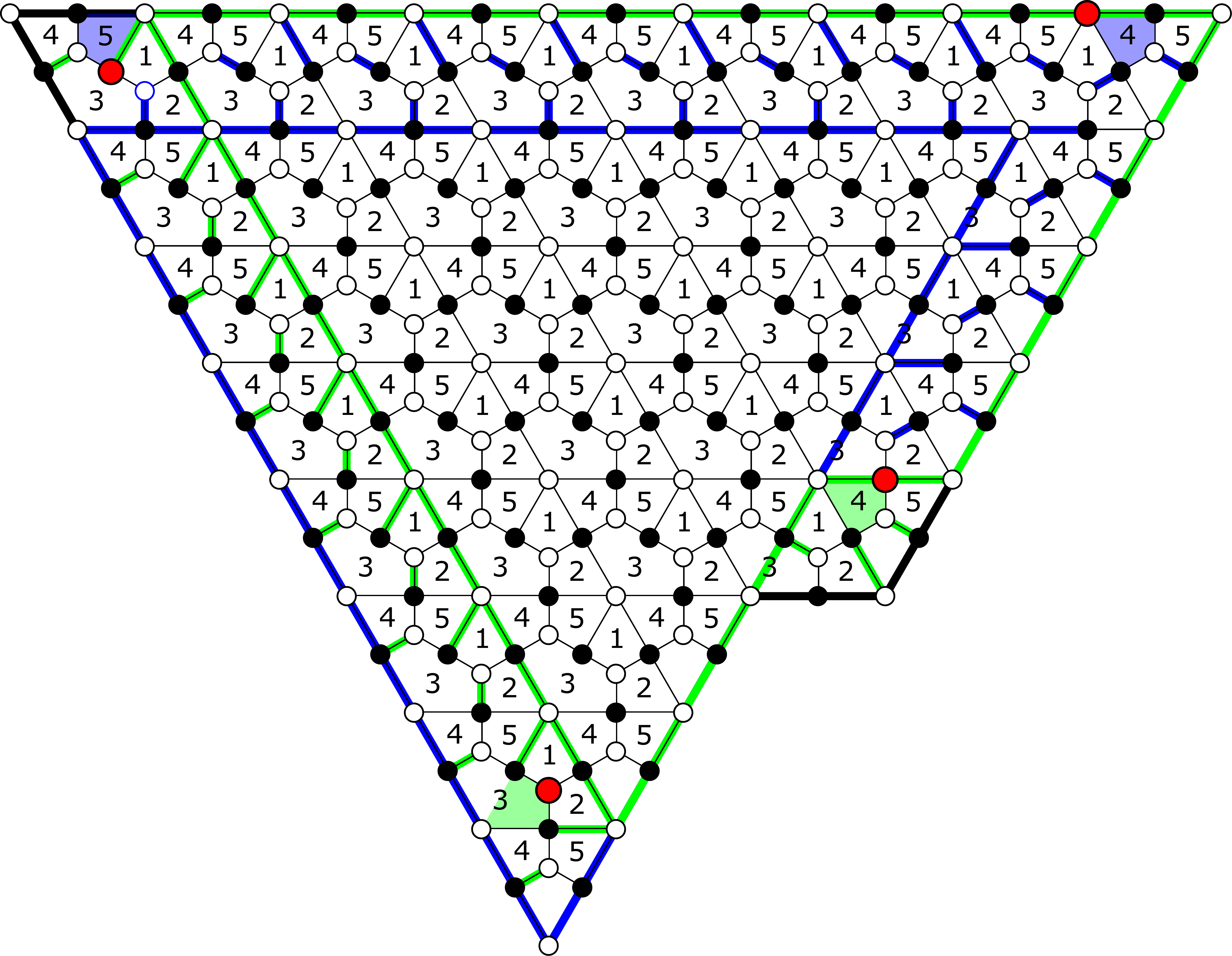}
\caption{Covering monomial for $x_{2k}$, $k$ even. Left: $T(\emptyset)$ and $T(\{p_1,p_2,p_3,p_4\})$. Middle: $T(\{p_1,p_2\})$ and $T(\{p_3,p_4\})$. Right $T(\{p_2,p_3\})$ and $T(\{p_1,p_4\})$.}
\label{fig:cov-mon-base2-K}
\end{figure}

$G-R$ (Special vertex removed): let $p_1$ be the bottommost (W) point on edge a (not in a forced matching), $p_2$ be the leftmost (B) point on edge e (not in a forced matching), $p_3$ be the topmost (W) point on edge d, $p_4$ be the (B) point on the edge between the 2-block and 3-block above the special vertex.
\begin{align*}
T(\emptyset) = 1, \quad T(\{p_1,p_2,p_3,p_4\}) = x_2x_3x_4x_5, \quad T(\{p_1,p_2\}) = x_3x_5, \\
T(\{p_3,p_4\}) = x_2x_4, \quad T(\{p_2,p_3\}) = x_4x_5, \quad T(\{p_1,p_4\}) = x_2x_3.
\end{align*}

\begin{figure}[h!]
\includegraphics[scale=0.1]{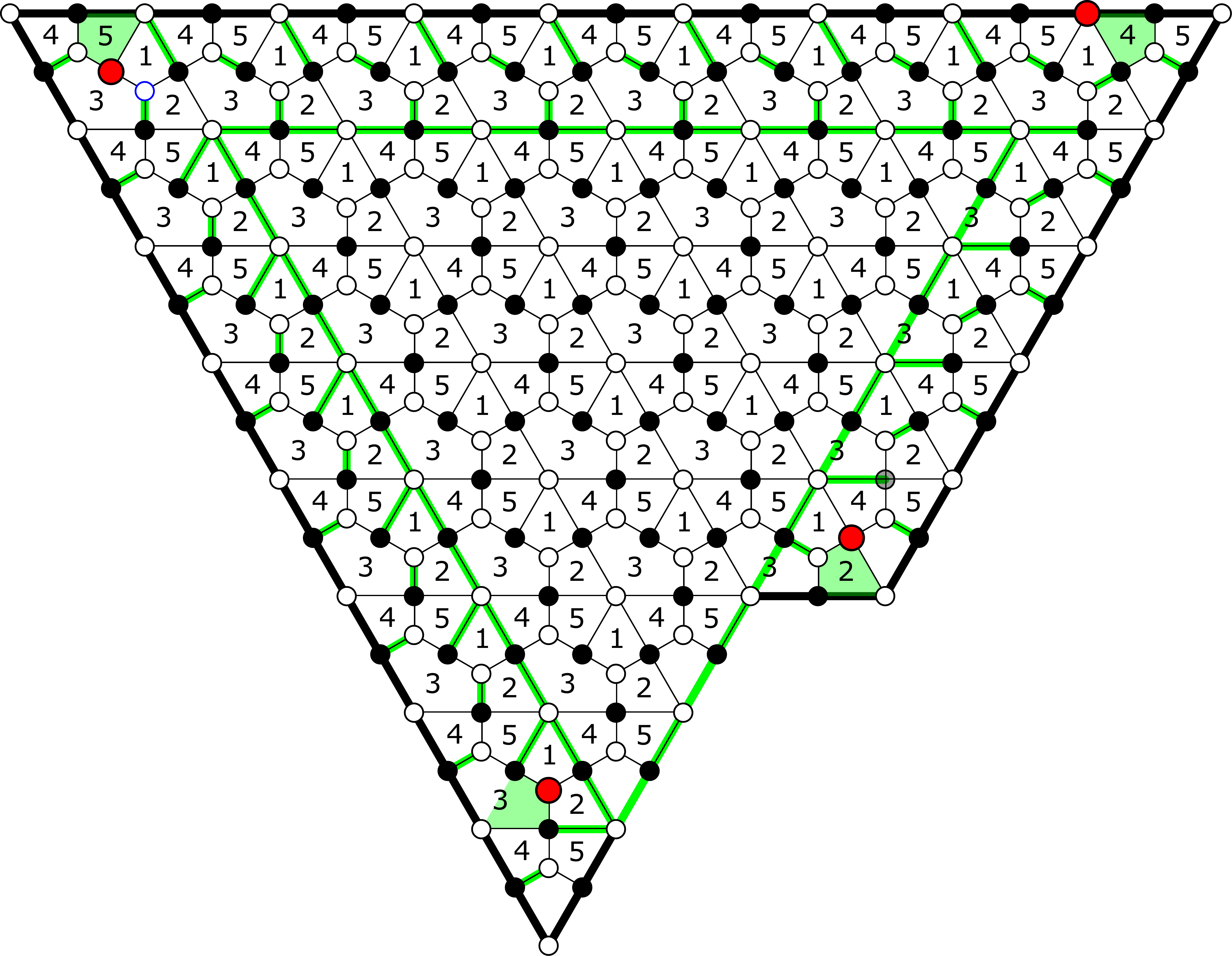}
\includegraphics[scale=0.1]{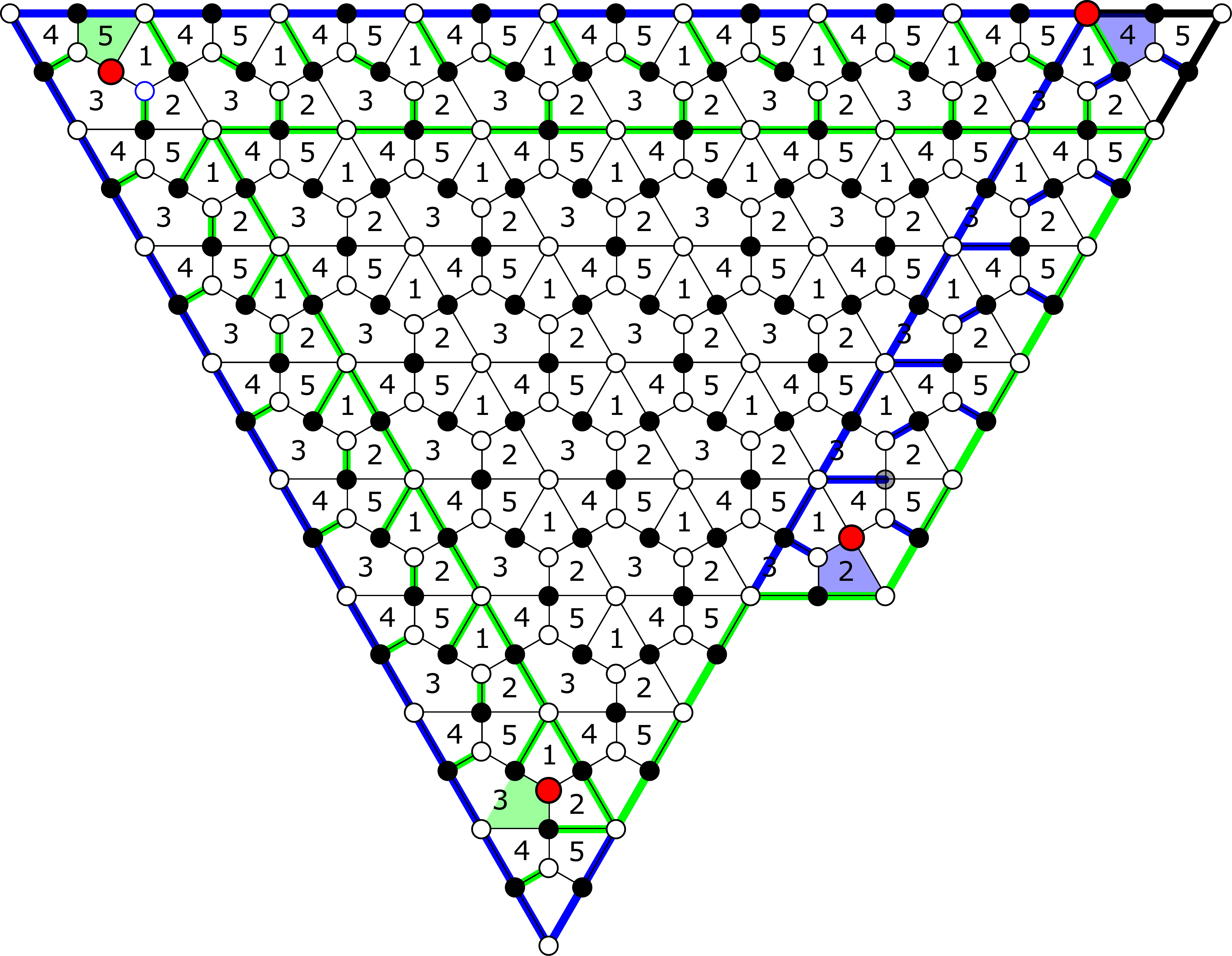}
\includegraphics[scale=0.1]{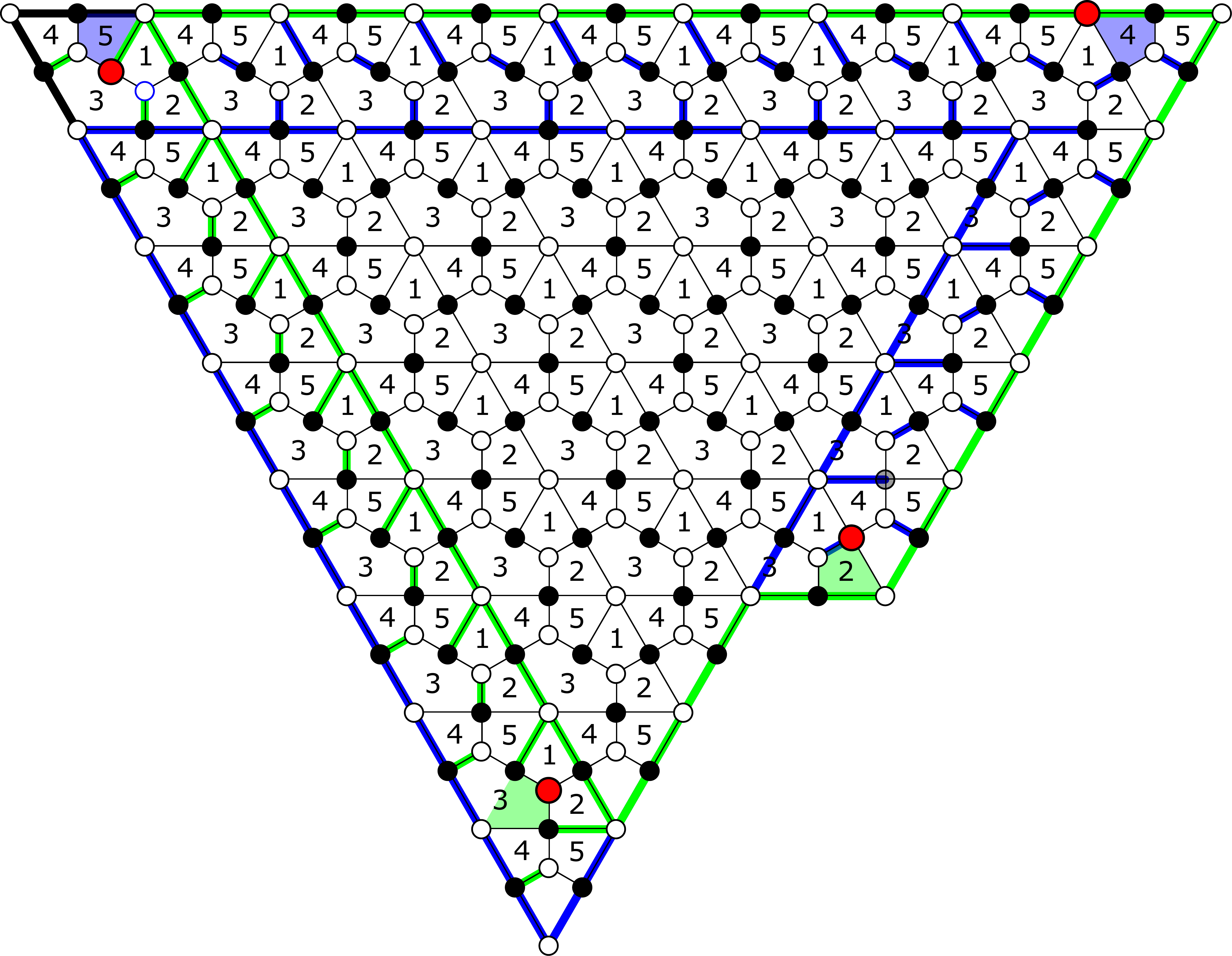}
\caption{Covering monomial for $x_{2k}$, $k$ odd. Left: $T(\emptyset)$ and $T(\{p_1,p_2,p_3,p_4\})$. Middle: $T(\{p_1,p_2\})$ and $T(\{p_3,p_4\})$. Right $T(\{p_2,p_3\})$ and $T(\{p_1,p_4\})$.}
\label{fig:cov-mon-base2-K}
\end{figure}

\subsection{Inductive Step for $A^{n^2}B^{n(n-1)}x_{2k}$, $n\geq1$, $k\geq3n-1$}\label{sub:induction}

\

As we have explained in Section~\ref{sub:overview}, we will only show the inductive step for the case $A^{n^2}B^{n(n-1)}x_{2k}$. All of the other cases can be proved in the same way and we provide the data in Appendix~\ref{sec:appendix} for doing so. 

Assume the contours of $A^{m^2}B^{m(m-1)}x_{2k}$ and $A^{m(m+1)}B^{m^2}x_{2k+1}$, as defined in Theorem~\ref{thm:contours}, give the correct cluster variables for any $m\leq n-1$ and $k\geq 3n-1$. Now we want to show that the contour of $A^{m^2}B^{m(m-1)}x_{2k}$ is correct for any $k \geq 3n-1$ and $m=n$. 

The recurrence we use is
\begin{align} \label{eqn:rec-1}
&(A^{n^2}B^{n(n-1)}x_{2k})(A^{(n-1)^2}B^{(n-1)(n-2)}x_{2k+2}) \\
=&(A^{(n-1)n}B^{(n-1)^2}x_{2k-1})(A^{(n-1)n}B^{(n-1)^2}x_{2k+3}) + (A^{(n-1)n}B^{(n-1)^2}x_{2k+1})^2
\end{align}

where by the induction hypothesis, we have the correctness of the contours for cluster variables $A^{(n-1)^2}B^{(n-1)(n-2)}x_{2k+2}$, $A^{(n-1)n}B^{(n-1)^2}x_{2k-1}$, $A^{(n-1)n}B^{(n-1)^2}x_{2k+1}$ and $A^{(n-1)n}B^{(n-1)^2}x_{2k+3}$.

For this case, let contour $C$ be the following:
$$C = (a,b,c,d,e) = \left(k-1+n, -\left\lceil \frac{k+5n-5}{2} \right\rceil, 2n-2, \left\lfloor \frac{k-3n+3}{2} \right\rfloor, n-k-1\right).$$


Since $k\geq 3n-1$, we have $a>0, b<0, c \ge 0,d>0, e<0$. Again, we use the steps described in Section~\ref{sub:technique}. Let $G=\hatcal{G}(C)$.

\

\noindent\textbf{Step 1:} We use non-alternating Kuo Condensation theorem (Lemma~\ref{lem:KuoNal}) and write down
\begin{align*}
w(G-\{p_1,p_2\})w(G-\{p_3,p_4\})=&w(G)w(G-\{p_1,p_2,p_3,p_4\})\\
&+w(G-\{p_1,p_3\})w(G-\{p_2,p_4\}). 
\end{align*}
where we let $S_1=\{p_1,p_2\}$, $S_2=\{p_3,p_4\}$, $S_3=\emptyset$, $S_4=\{p_1,p_2,p_3,p_4\}$, $S_5=\{p_1,p_3\}$, $S_6=\{p_2,p_4\}$. Then we multiply both sides by $m(\cal{G}(C))^2$. 

\

\noindent\textbf{Step 2.} We define the four points $p_1,p_2,p_3,p_4$ on edge $d,e,b,c$ respectively, where $p_1$, $p_4$ are white, while $p_2$, $p_3$ are black. We list the effect of each removal as follows. 
\begin{align*}
-\{p_1\} &= 
\begin{cases}
( 0,-1,1,-1,1 ) -R,& \text{if } G = (a,b,c,d,e)-K\\
    ( 0,0,1,0,1 ) -K,              & \text{if } G = (a,b,c,d,e)-R
\end{cases}
\\
-\{p_2\} &= ( -1,0,0,-1,1 )\\
-\{p_3\} &= ( -1,1,-1,0,0 )\\
-\{p_4\}&= 
\begin{cases}
    ( 0,0,0,0,0 ) -R,& \text{if } G = (a,b,c,d,e)-K\\
    ( 0,1,0,1,0 ) -K,              & \text{if } G = (a,b,c,d,e)-R
\end{cases}
\end{align*}

The position of these points and the effects of removing each point, when the special point is kept, is shown in Figure~\ref{fig:effect_1.1_p1_K} ($p_1$) and Figure~\ref{fig:effect_1.1_K} ($p_2,p_3,p_4$). Notice that after we remove $p_1$, as shown in Figure~\ref{fig:effect_1.1_p1_K}, some area gets deleted (grey) and some area gets added (pink). The position of these points and effects of removing each point, when the special point ire removed, is shown in Figure~\ref{fig:effect_1.1_R}.

\begin{figure}[h!]
\includegraphics[scale=0.1]{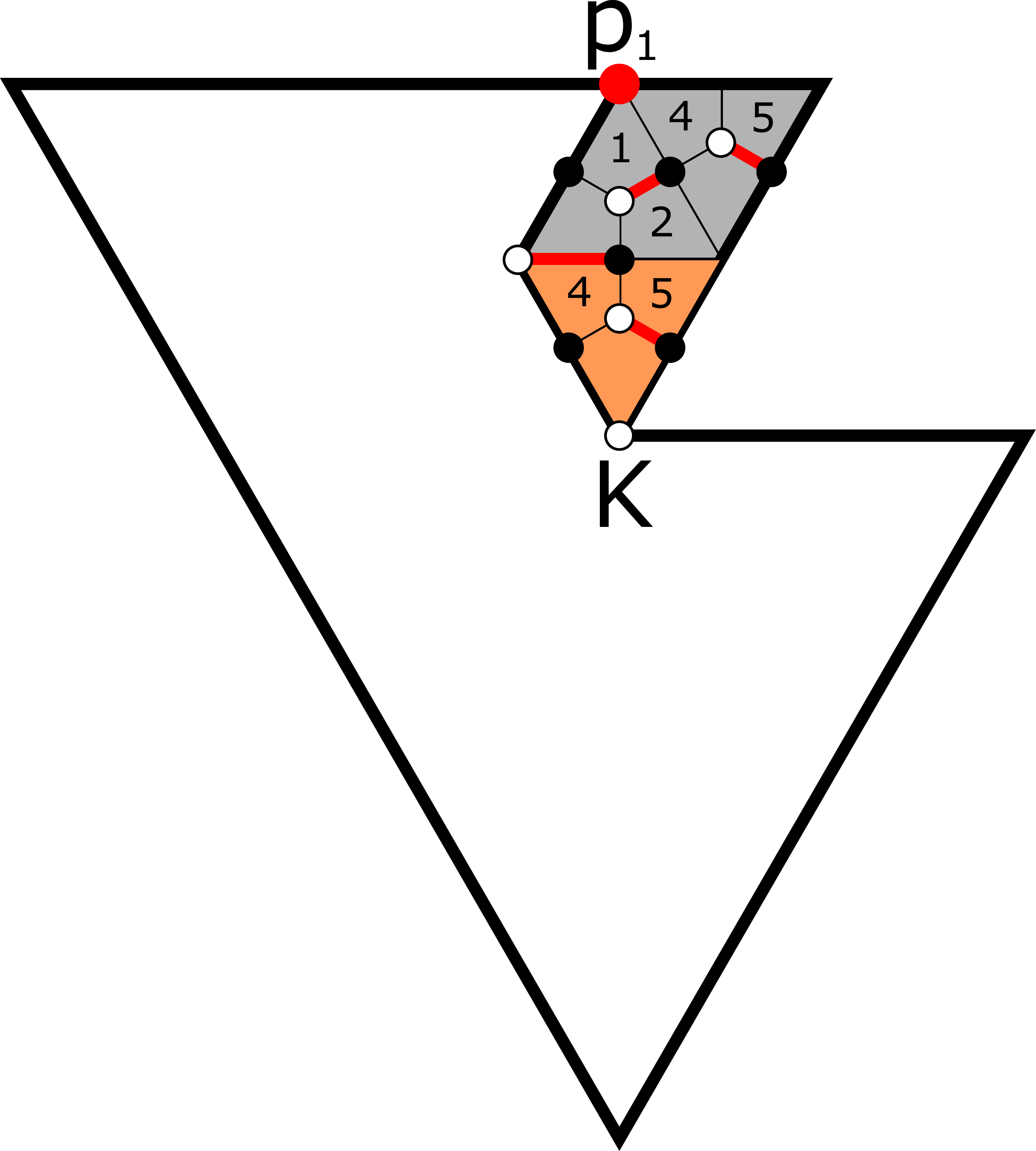}
\quad
\includegraphics[scale=0.1]{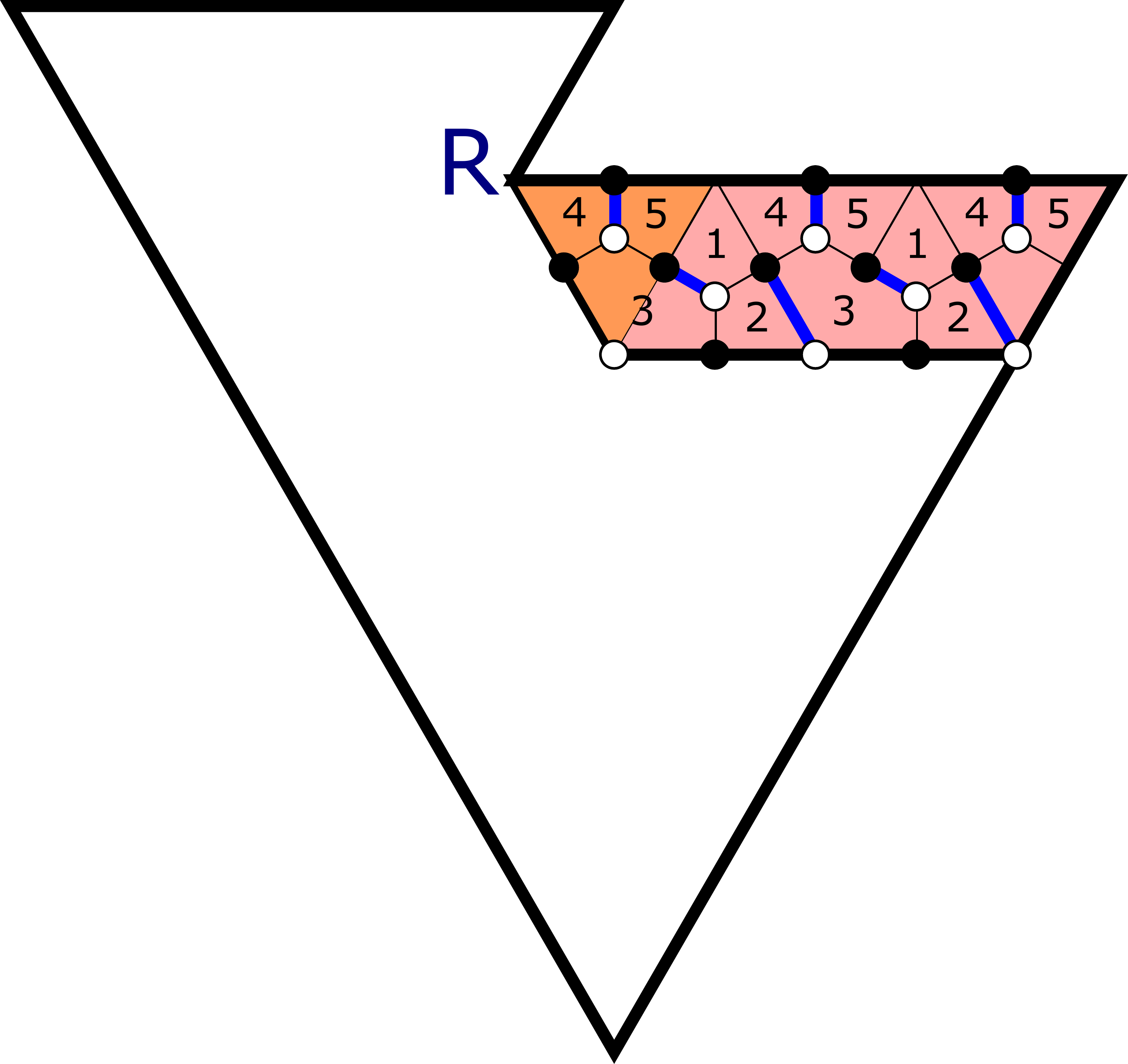}
\caption{Effect of removing $p_1$ for $A^{n^2}B^{n(n-1)}$ with $n\geq1$, $k\geq 3n-1$ and the special point kept. Left: before removal. Right: after removal.}
\label{fig:effect_1.1_p1_K}
\end{figure}

\begin{figure}[h!]
\includegraphics[scale=0.1]{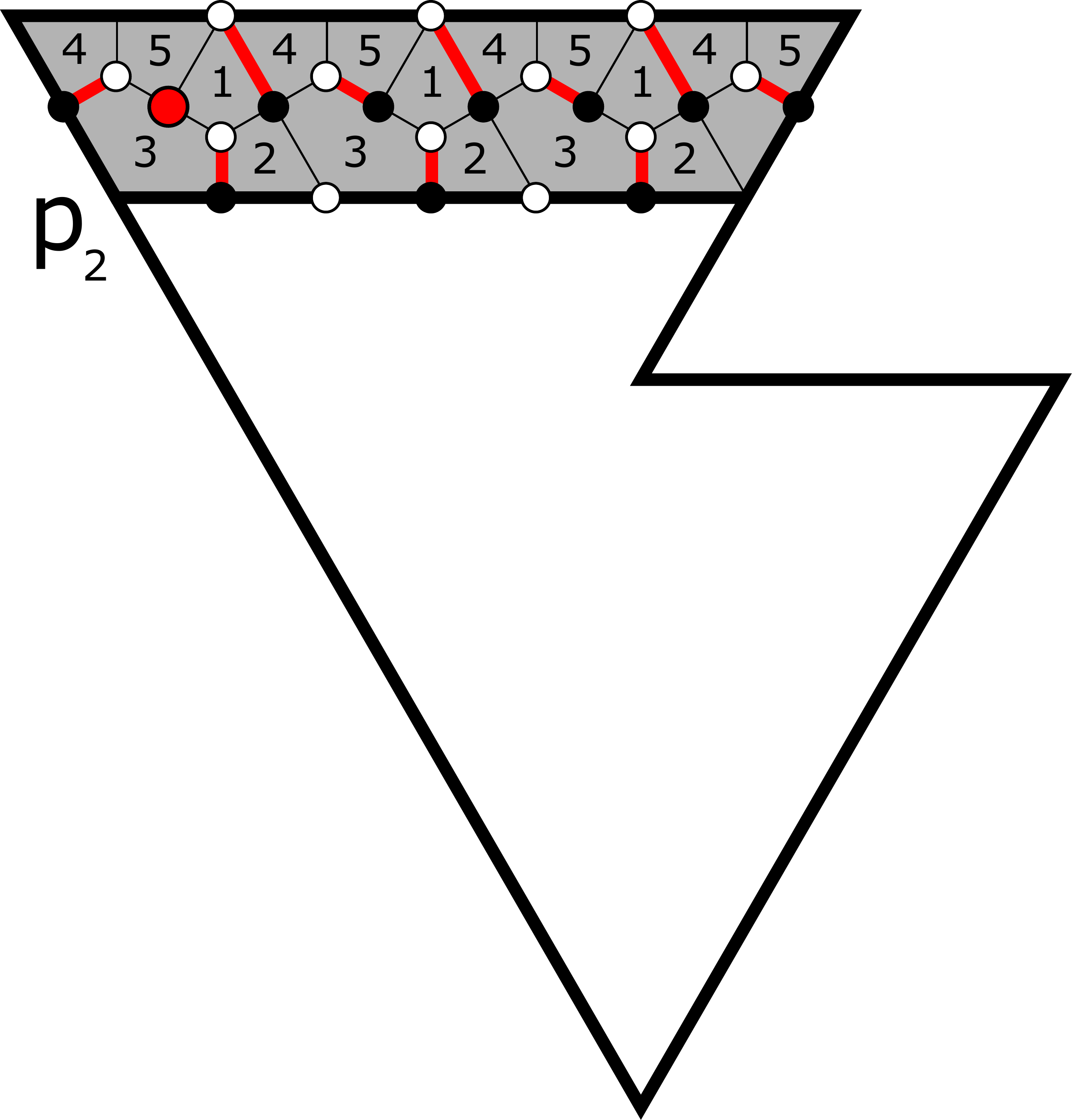}\quad
\includegraphics[scale=0.1]{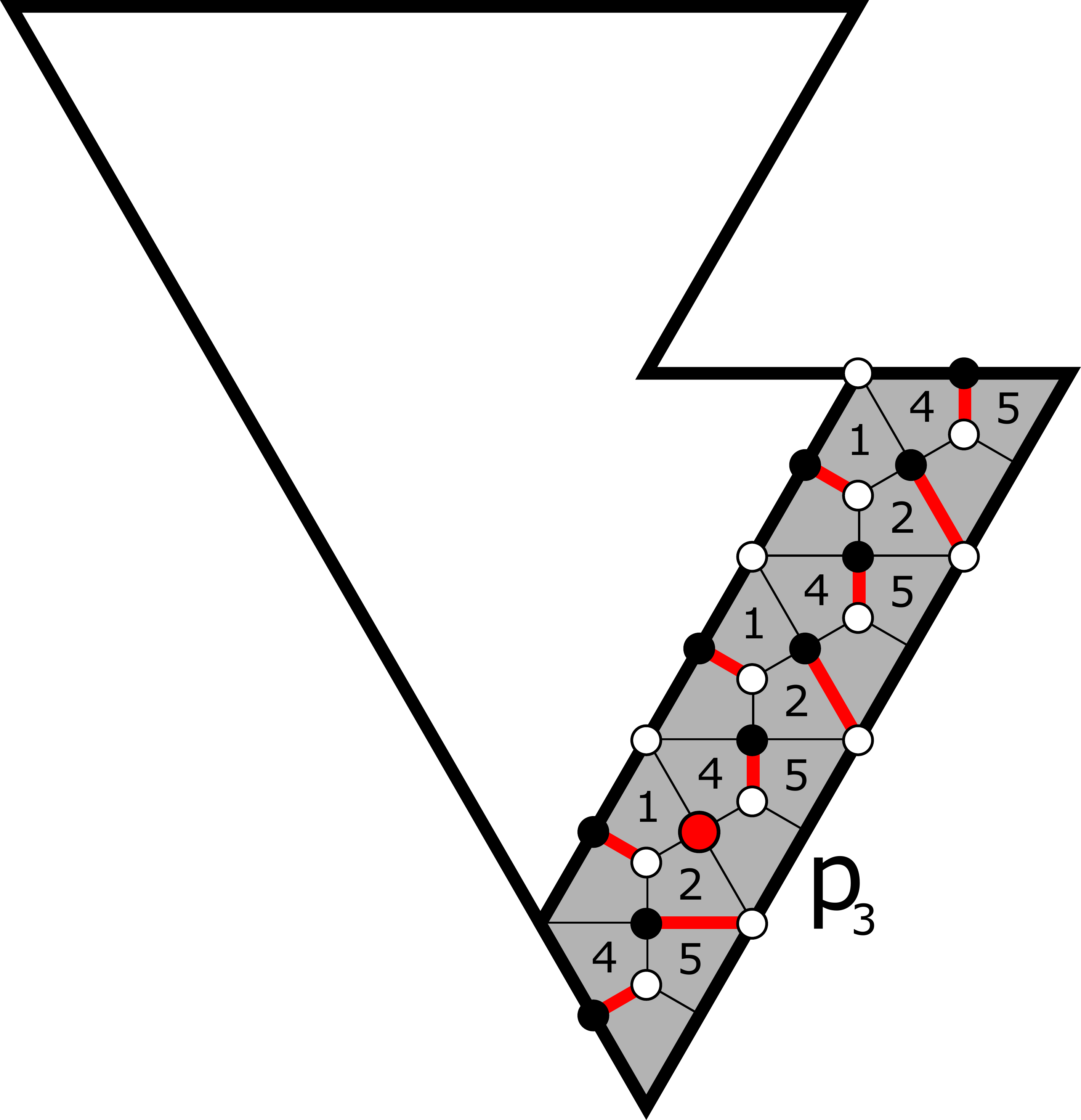}\quad
\includegraphics[scale=0.1]{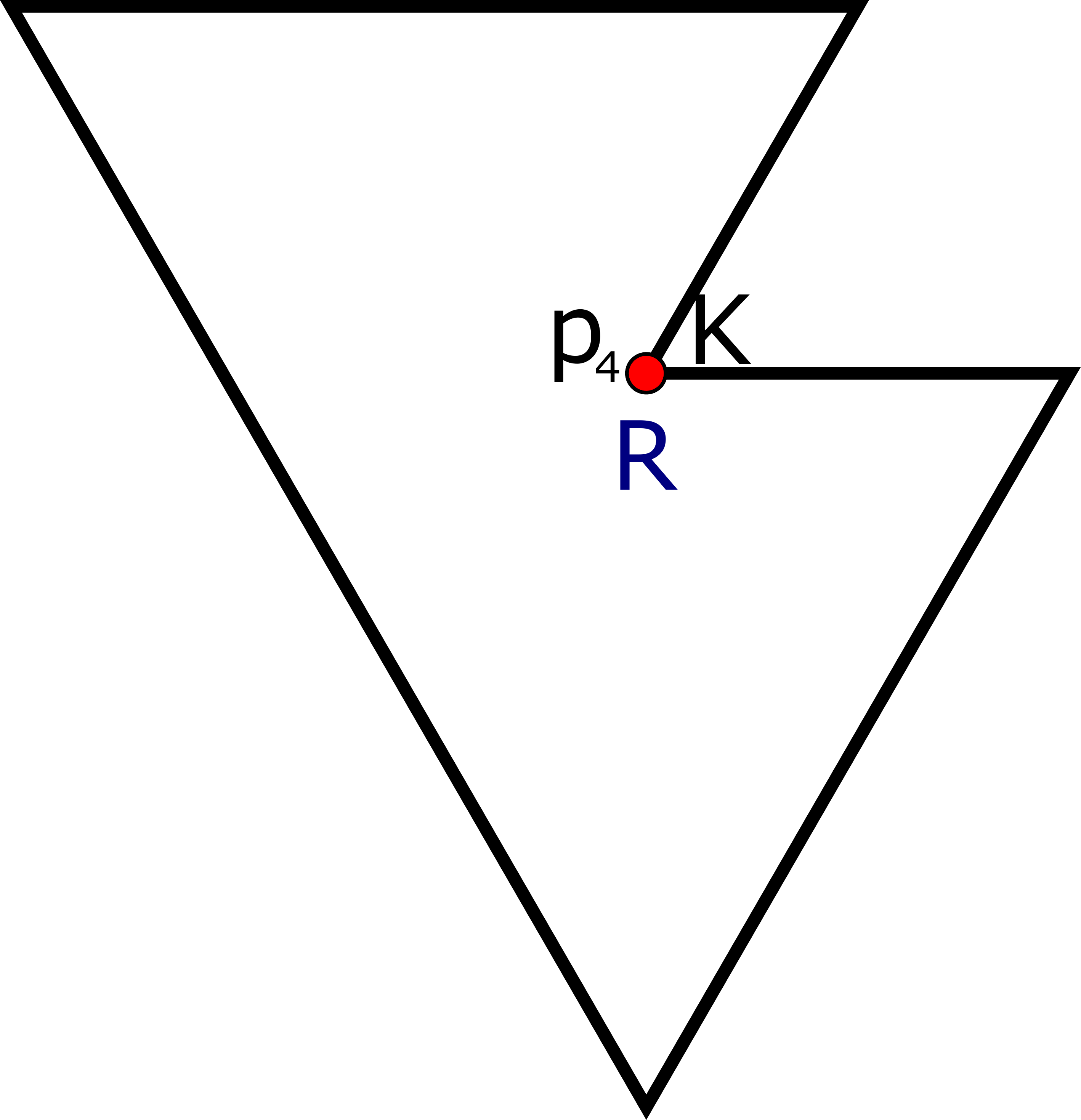}
\caption{Effects of removing $p_2,p_3,p_4$ for $A^{n^2}B^{n(n-1)}$ with $n\geq1$, $k\geq 3n-1$ and the special point kept.}
\label{fig:effect_1.1_K}
\end{figure}

\begin{figure}[h!]
\includegraphics[scale=0.1]{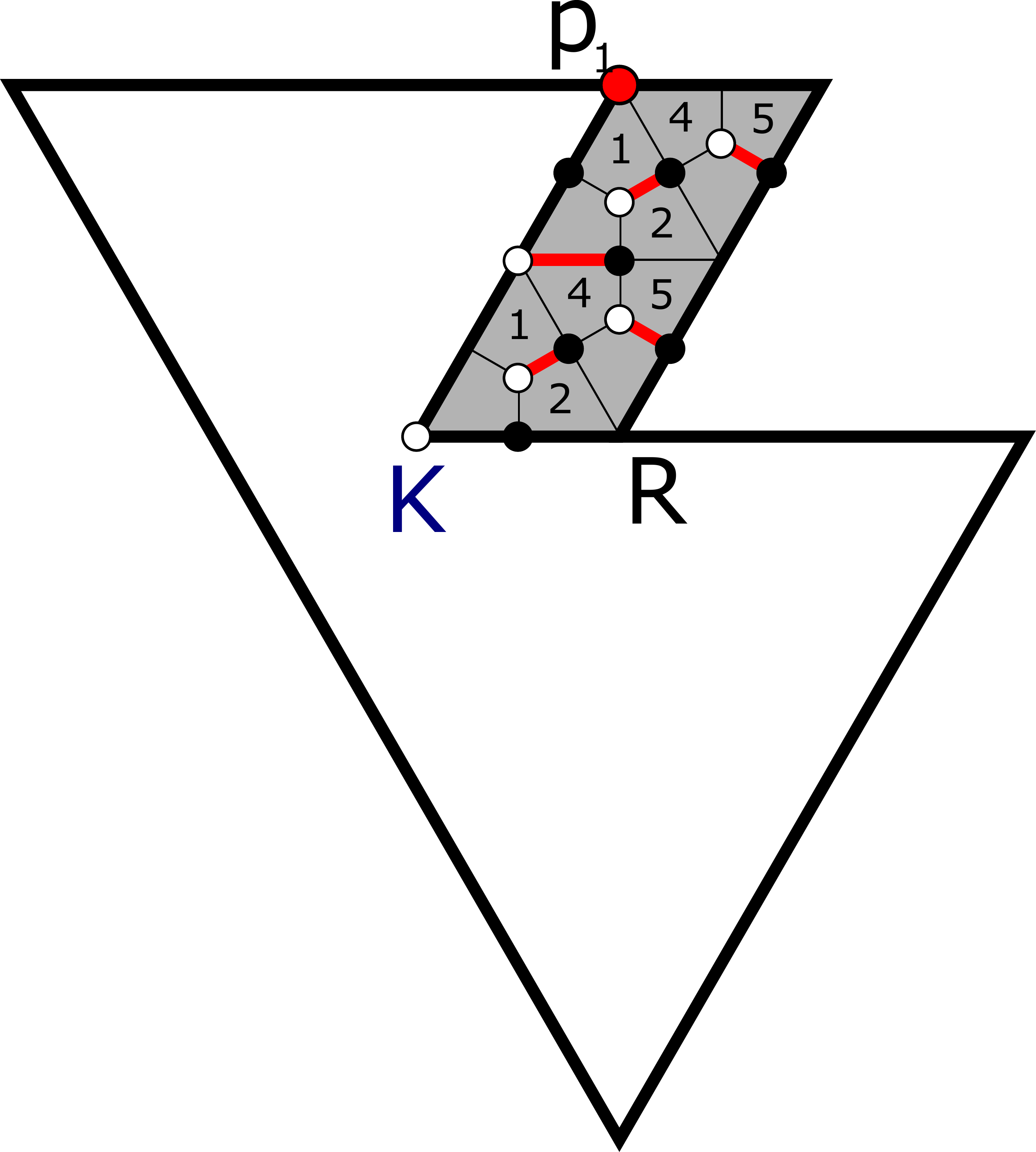}
\includegraphics[scale=0.1]{remove_p2_+_-_+_+_-_.pdf}
\includegraphics[scale=0.1]{remove_p3_+_-_+_+_-_.pdf}
\includegraphics[scale=0.1]{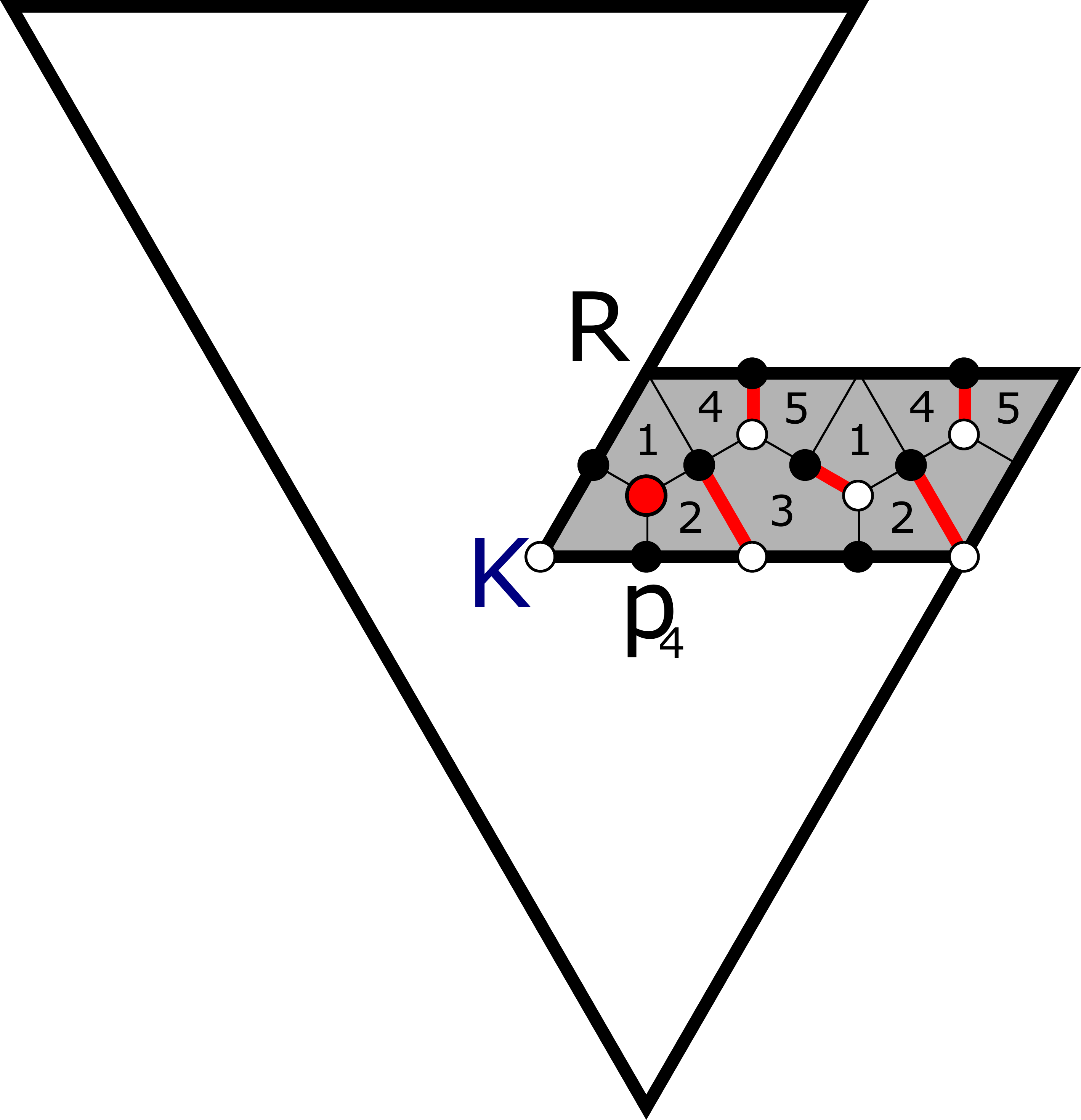}
\caption{Effects of removing $p_i$'s for $A^{n^2}B^{n(n-1)}$ with $n\geq1$, $k\geq 3n-1$ and the special point removed.}
\label{fig:effect_1.1_R}
\end{figure}

Below, we explicitly write down the contour $C_i$ satisfying $\hat{\cal{G}(C)-S_i}=\hat{\cal{G}}(C_i)$ for each $S_i$, with the corresponding cluster variable, followed from induction hypothesis. 

\noindent\textbf{Case 1:} $n+k$ is odd. Thus, the special vertex is kept, and $C=(a,b,c,d,e)-K$. 

\begin{align*}
\hat{G-\{p_1,p_2\}}&= \hatcal{G}(a-1, b-1, c+1, d-2, e+2)-R\\
&=\hatcal{G}(C_1)\\
\hat{G-\{p_3,p_4\}}&= \hatcal{G}(a-1, b+1, c-1, d, e)-R\\
&=\hatcal{G}(C_2),\ \text{graph of }A^{(n-1)^2}B^{(n-1)(n-2)}x_{2k+2}\\
\hat{G} &= \hatcal{G}(a, b, c, d, e)-K\\
&=\hatcal{G}(C_3),\ \text{graph of }A^{n(n-1)}B^{(n-1)^2}x_{2k+3}\\
\hat{G-\{p_1,p_2,p_3,p_4\}} &= \hatcal{G}(a-2, b+1, c, d-1, e+2)-K\\
&=\hatcal{G}(C_4),\ \text{graph of }A^{n^2-n}B^{(n-1)^2}x_{2k-1}\\
\hat{G-\{p_1,p_3\}}&= \hatcal{G}(a-1, b, c, d-1, e+1)-R\\ 
&=\hatcal{G}(C_5),\ \text{graph of }A^{n^2-n}B^{(n-1)^2}x_{2k+1}\\
\hat{G-\{p_2,p_4\}}&= \hatcal{G}(a-1, b, c, d-1, e+1)-R\\
&=\hatcal{G}(C_6),\ \text{graph of }A^{n^2-n}B^{(n-1)^2}x_{2k+1}
\end{align*}

\noindent\textbf{Case 2:} $n+k$ is even. Thus, the special vertex is removed, and $C=(a,b,c,d,e)-R$. 

\begin{align*}
\hat{G-\{p_1,p_2\}}&= \hatcal{G}(a-1, b, c+1, d-1, e+2)-K\\
&=\hatcal{G}(C_1)\\
\hat{G-\{p_3,p_4\}}&= \hatcal{G}(a-1, b+2, c-1, d+1, e)-K\\
&=\hatcal{G}(C_2),\ \text{graph of }A^{(n-1)^2}B^{(n-1)(n-2)}x_{2k+2}\\
\hat{G} &= \hatcal{G}(a, b, c, d, e)-R\\
&=\hatcal{G}(C_3),\ \text{graph of }A^{n(n-1)}B^{(n-1)^2}x_{2k+3}\\
\hat{G-\{p_1,p_2,p_3,p_4\}} &= \hatcal{G}(a-2, b+1, c, d-1, e+2)-R\\
&=\hatcal{G}(C_4),\ \text{graph of }A^{n^2-n}B^{(n-1)^2}x_{2k-1}\\
\hat{G-\{p_1,p_3\}}&= \hatcal{G}(a-1, b+1, c, d, e+1)-K\\ 
&=\hatcal{G}(C_5),\ \text{graph of }A^{n^2-n}B^{(n-1)^2}x_{2k+1}\\
\hat{G-\{p_2,p_4\}}&= \hatcal{G}(a-1, b+1, c, d, e+1)-K\\
&=\hatcal{G}(C_6),\ \text{graph of }A^{n^2-n}B^{(n-1)^2}x_{2k+1}
\end{align*}

By recurrence~\ref{eqn:rec-1}, we conclude that $\hatcal{G}(C_1)$ is the graph corresponding to $A^{n^2}B^{n^2-n}x_{2k}$.

\

\noindent\textbf{Step 3.} In this step we specify the sets $T_i$ for a choice of $p_1,p_2,p_3,p_4$ and verify equation~\ref{eqn:T_i}. In each of the diagrams in Figures~\ref{fig:cov-mon+-++-K} and \ref{fig:cov-mon+-++-R}, $p_1,p_2,p_3,p_4$ are the red points. Each figure shows the new contours $C_i$ and $C_{i+1}$ in green and in blue.

There is a bijection between perfect matchings of $\cal{G}(C_i)$ and perfect matchings of $\cal{G}(C)- S_i$. Let $M$ be any perfect matching of 
$\cal{G}(C_i)$. Essentially, the weight of the blocks in $T_i$ is exactly what we need to multiply $m(\cal{G}(C_i))w(M)$ by so that it corresponds to a term of $m(\cal{G}(C))w(\cal{G}(C)-S_i)$. 

We explain how Figure~\ref{fig:cov-mon+-++-K}(Left) allows us to determine that the weight of $T_1$ is $x_3x_3x_4x_4$. Let us start with the green contour $C_1$. A perfect matching of $\cal{G}(C_1)$ corresponds to a perfect matching of $G-\{p_1,p_2\}$ if we remove the red matchings and add in the green matchings. Algebraically, this corresponds to multiplying by the weight of these matchings. The covering monomial of $\cal{G}(C_1)$ must be multiplied by the weight of all blocks that are outside the green contour $C_1$ and within the largest contour $C$. Note that the weight of these blocks are divided out by many of the green matchings and only the two 4-blocks (green) along edge e and the single 3-block (cyan) near the special vertex remain.

In this particular case, the contour $C_1$ is not completely contained in $C$, so we must also divide by the weight of all blocks within $C_1$ and outside $C$. Again, note that these weight of these blocks divide out all but one of the red matchings. So overall, the weight of $T_1$ includes $x_3x_4x_4$ from the covering monomial of $C$, part of the weight $\frac{1}{x_3x_4x_5}$ of two green matchings (shaded), and the weight $x_4x_5$ of the single red matching within $C_1$ (green). So the weight of $T_1$ is $x_4x_4$. Similarly, we find the weight of $T_2$ is $x_3x_3$ since we simply need to multiply by the weight of blocks outside $C_2$ within $C$ and the only blocks that are not divided out by forced matchings are the 3-block near the special vertex (cyan) and the 3-block near $p_3$ (blue).

\

\noindent\textbf{Case 1:} Special vertex kept. See Figure~\ref{fig:cov-mon+-++-K}.
\begin{figure}[h!]
\includegraphics[scale=0.11]{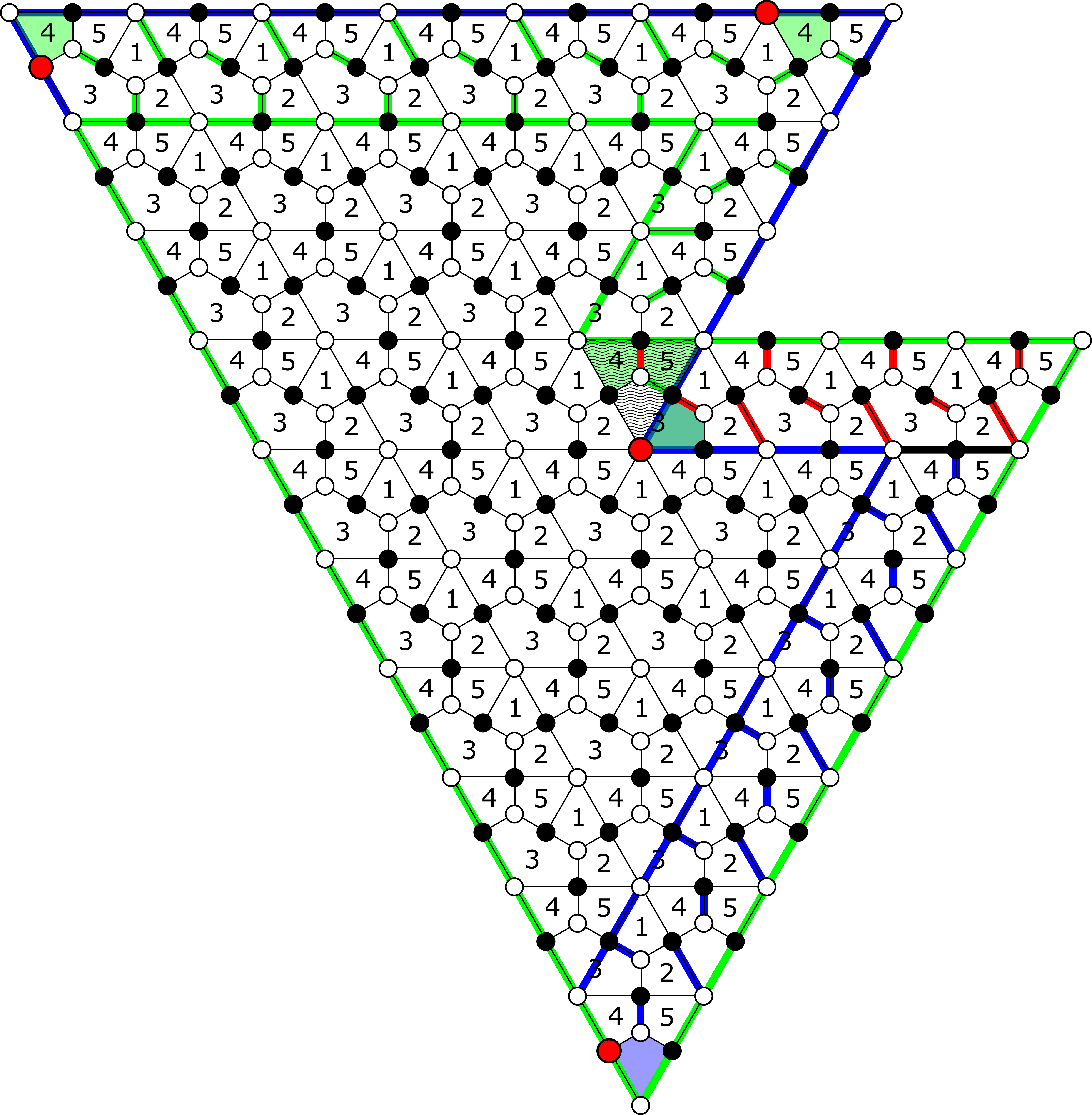}
\includegraphics[scale=0.11]{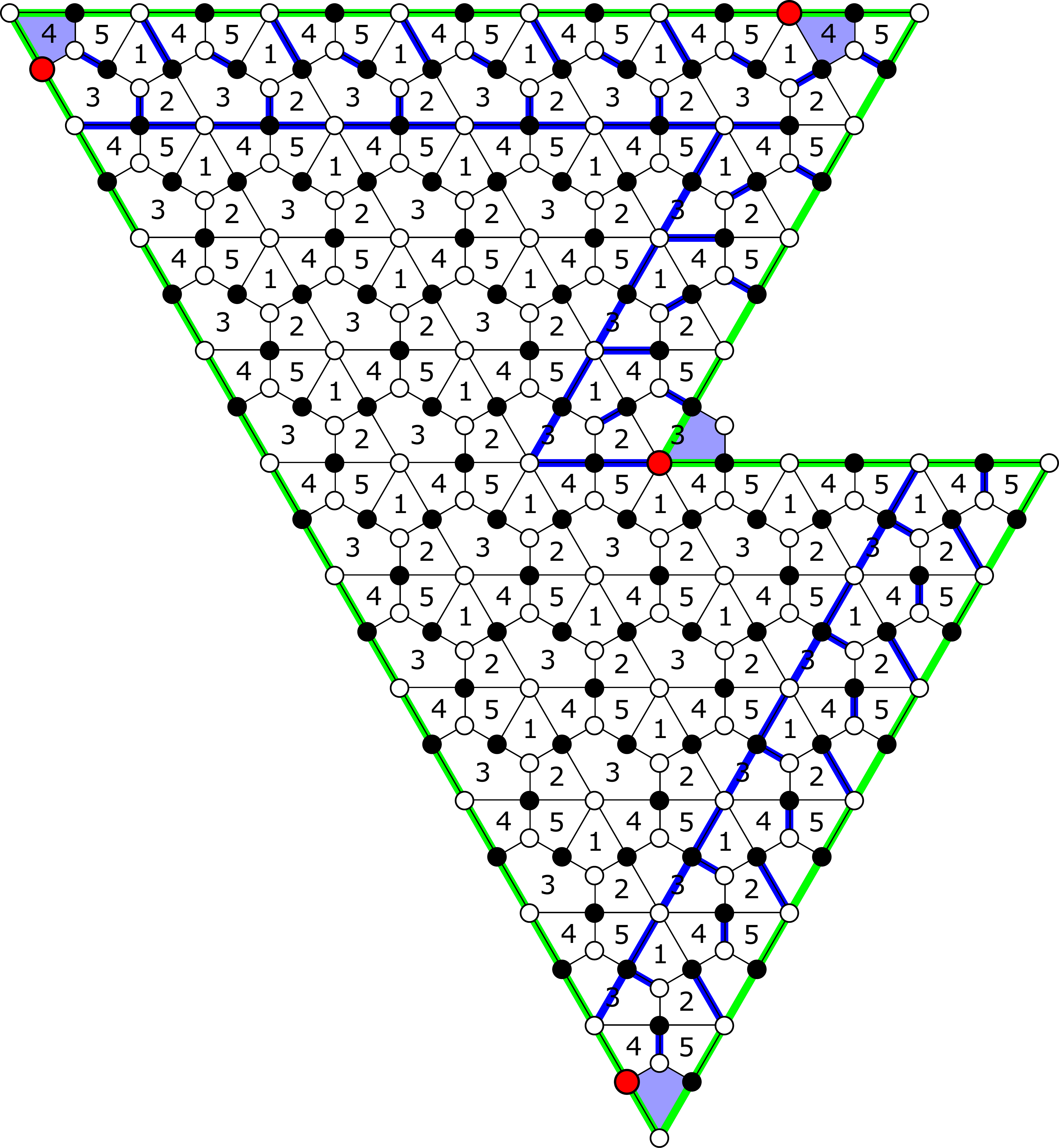}
\includegraphics[scale=0.11]{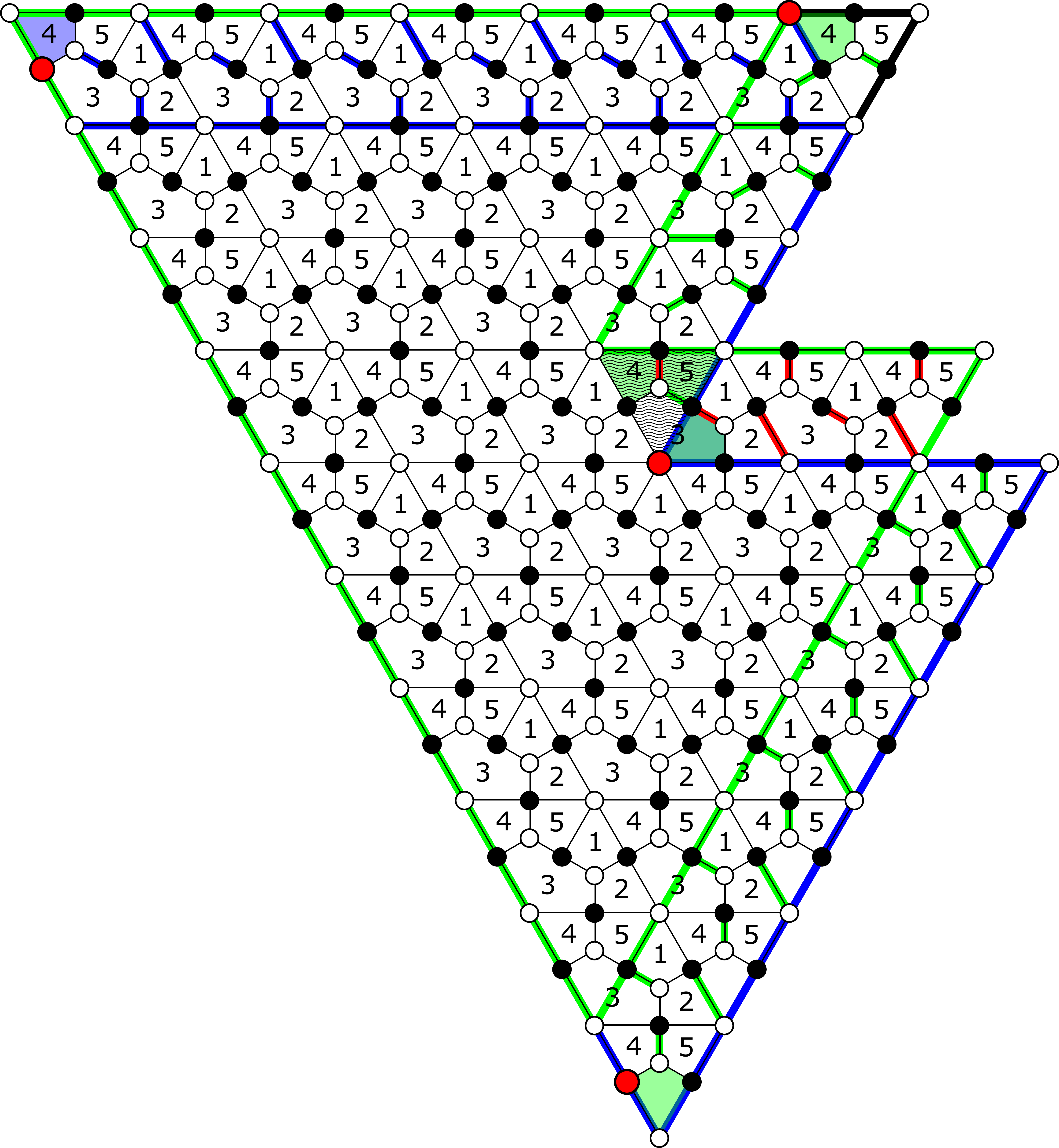}
\caption{Covering monomial for the case of $A^{n^2}B^{n(n-1)}$ with $n\geq1$, $k\geq3n-1$ and the special point kept. Left: $T(\{p_1,p_2\})$ and $T(\{p_3,p_4\})$. Middle: $T(\{p_1,p_2,p_3,p_4\})$ and $T(\emptyset)$. Right: $T(\{p_1,p_3\})$ and $T(\{p_2,p_4\})$.}
\label{fig:cov-mon+-++-K}
\end{figure}

For this choice of $p_1,p_2,p_3,p_4$, we have 
\begin{align*}
&\prod_{j\in T_1}x_j = x_3x_4x_4\frac{1}{x_3x_4x_5}x_4x_5 = x_4x_4,\qquad\prod_{j\in T_2}x_j = x_3x_3, \\
&\prod_{j\in T_3}x_j = 1,\qquad\prod_{j\in T_4}x_j = x_3x_3x_4x_4, \\
&\prod_{j\in T_5}x_j =  x_3x_3x_4\frac{1}{x_3x_4x_5}x_4x_5 = x_3x_4,\qquad\prod_{j\in T_6}x_j = x_3x_4.
\end{align*}

\

\noindent\textbf{Case 2:} Special vertex removed. See Figure~\ref{fig:cov-mon+-++-R}.
\begin{figure}[h!]
\includegraphics[scale=0.11]{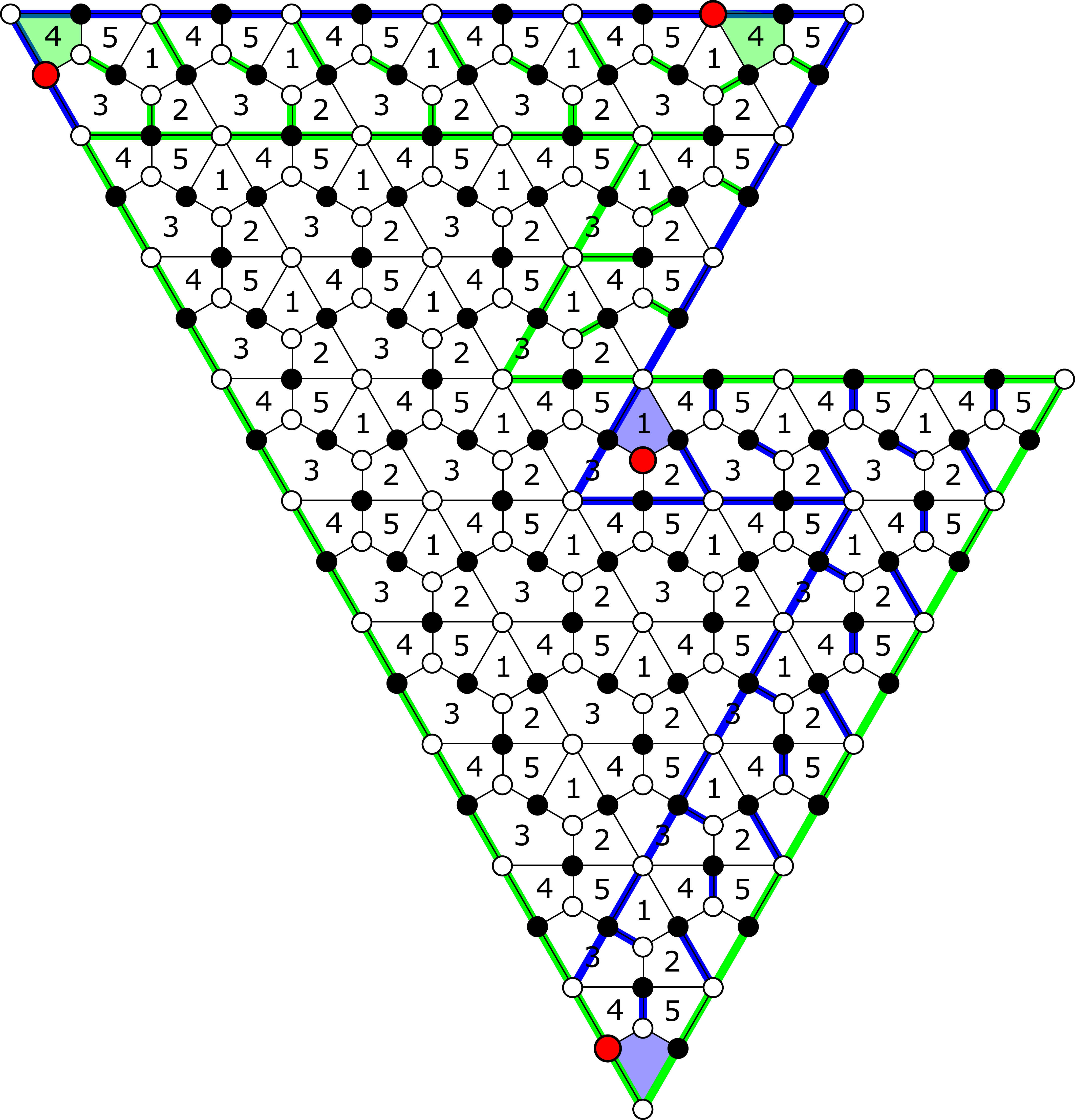}
\includegraphics[scale=0.11]{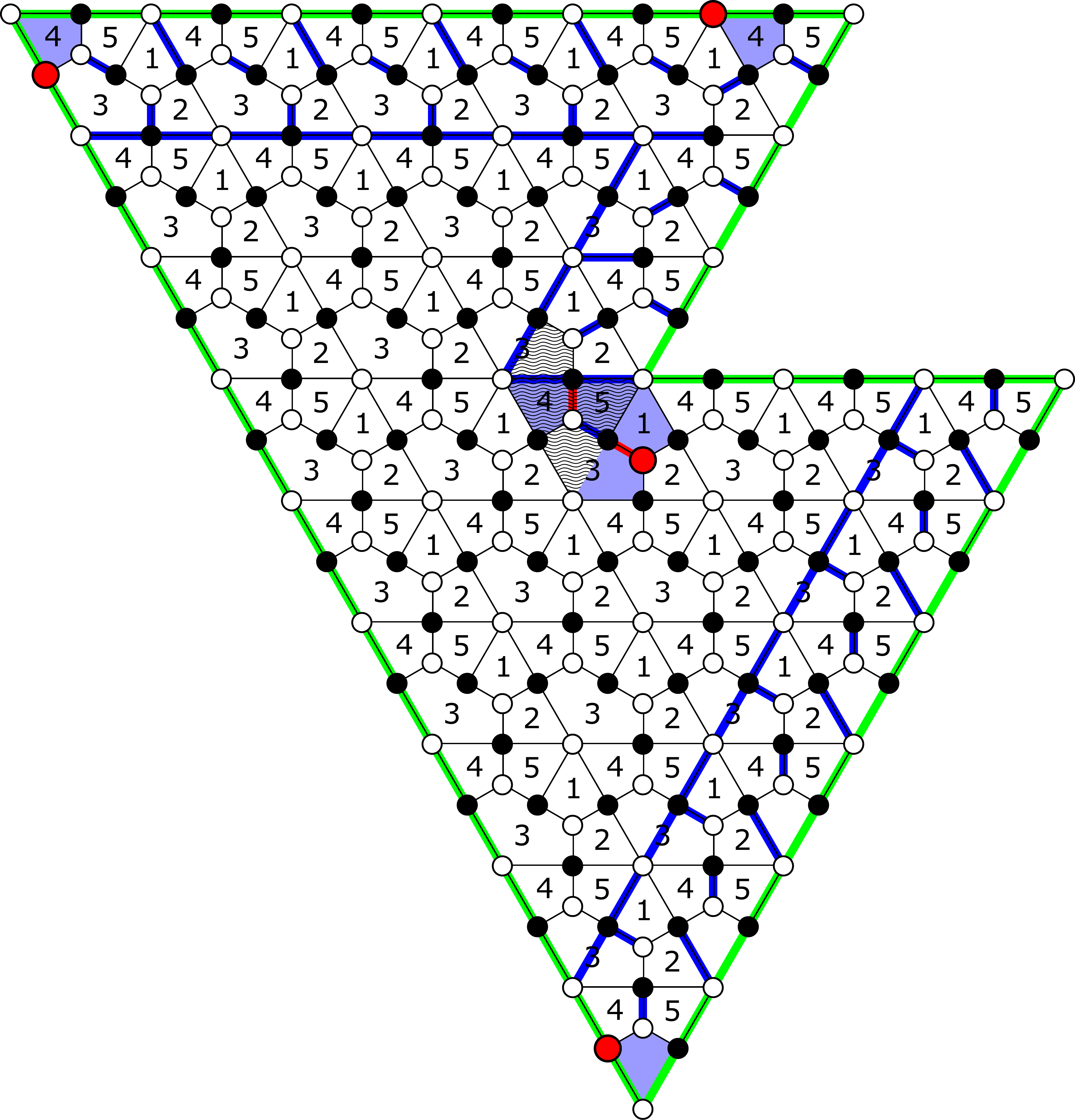}
\includegraphics[scale=0.11]{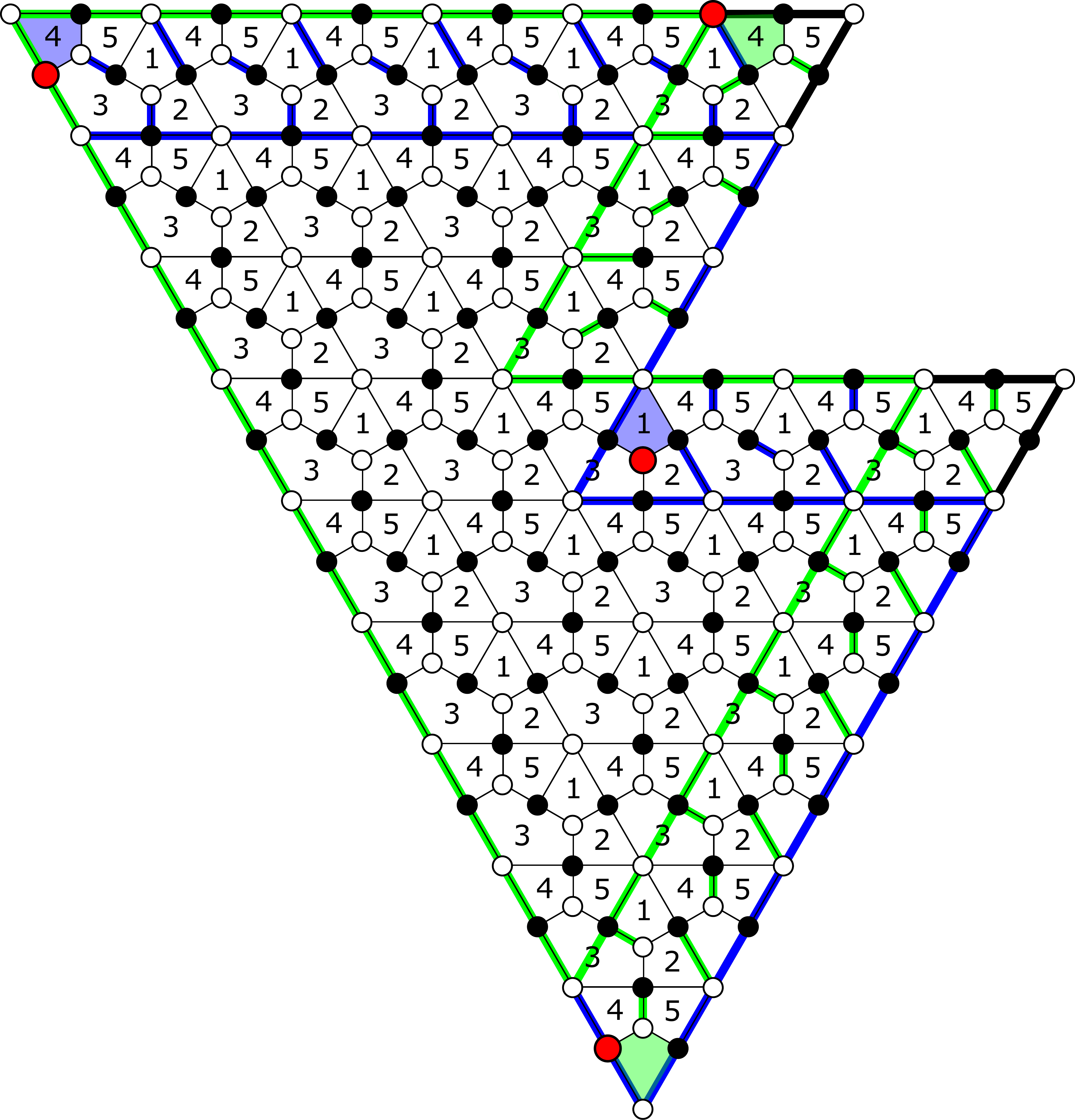}
\caption{Covering monomial for the case of $A^{n^2}B^{n(n-1)}$ with $n\geq1$, $k\geq3n-1$ and the special point removed. Left: $T(\{p_1,p_2\})$ and $T(\{p_3,p_4\})$. Middle: $T(\{p_1,p_2,p_3,p_4\})$ and $T(\emptyset)$. Right: $T(\{p_1,p_3\})$ and $T(\{p_2,p_4\})$.}
\label{fig:cov-mon+-++-R}
\end{figure}

For this choice of $p_1,p_2,p_3,p_4$, we have 
\begin{align*}
&\prod_{j\in T_1}x_j = x_4x_4,\qquad\prod_{j\in T_2}x_j = x_1x_3, \\
&\prod_{j\in T_3}x_j = 1,\qquad\prod_{j\in T_4}x_j = x_3x_3x_4\frac{x_3\cdot x_1x_3x_4x_5}{x_3x_3x_4x_5} = x_1x_3x_3x_4, \\
&\prod_{j\in T_5}x_j = x_3x_4,\qquad\prod_{j\in T_6}x_j = x_1x_4.
\end{align*}

We see that equation~\ref{eqn:T_i} holds in both cases.

\begin{remark}
As long as we fix the side and the color of a point $p_i$, the effect of removing $p_i$ is the same regardless of the shape of the contour, i.e. regardless of the signs of the other side lengths.  For instance, as shown in Figure~\ref{fig:same_effect}, the effects of removing $p_4$ in shapes $(+,-,+,+,-)$ and $(+,-,+,-,+)$ are the same.

\begin{figure}[h!]
\includegraphics[scale = 0.2]{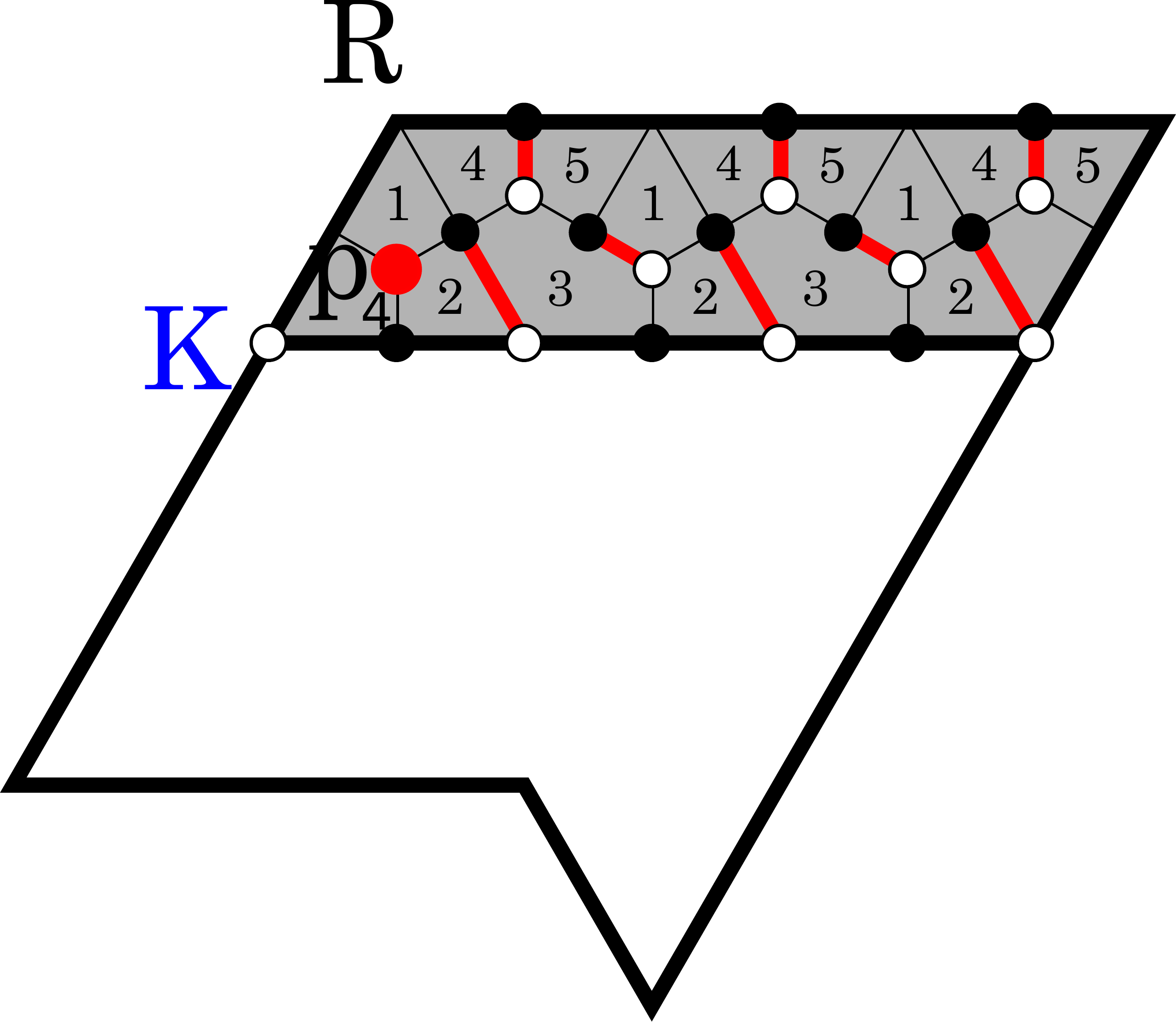}
\quad
\includegraphics[scale = 0.15]{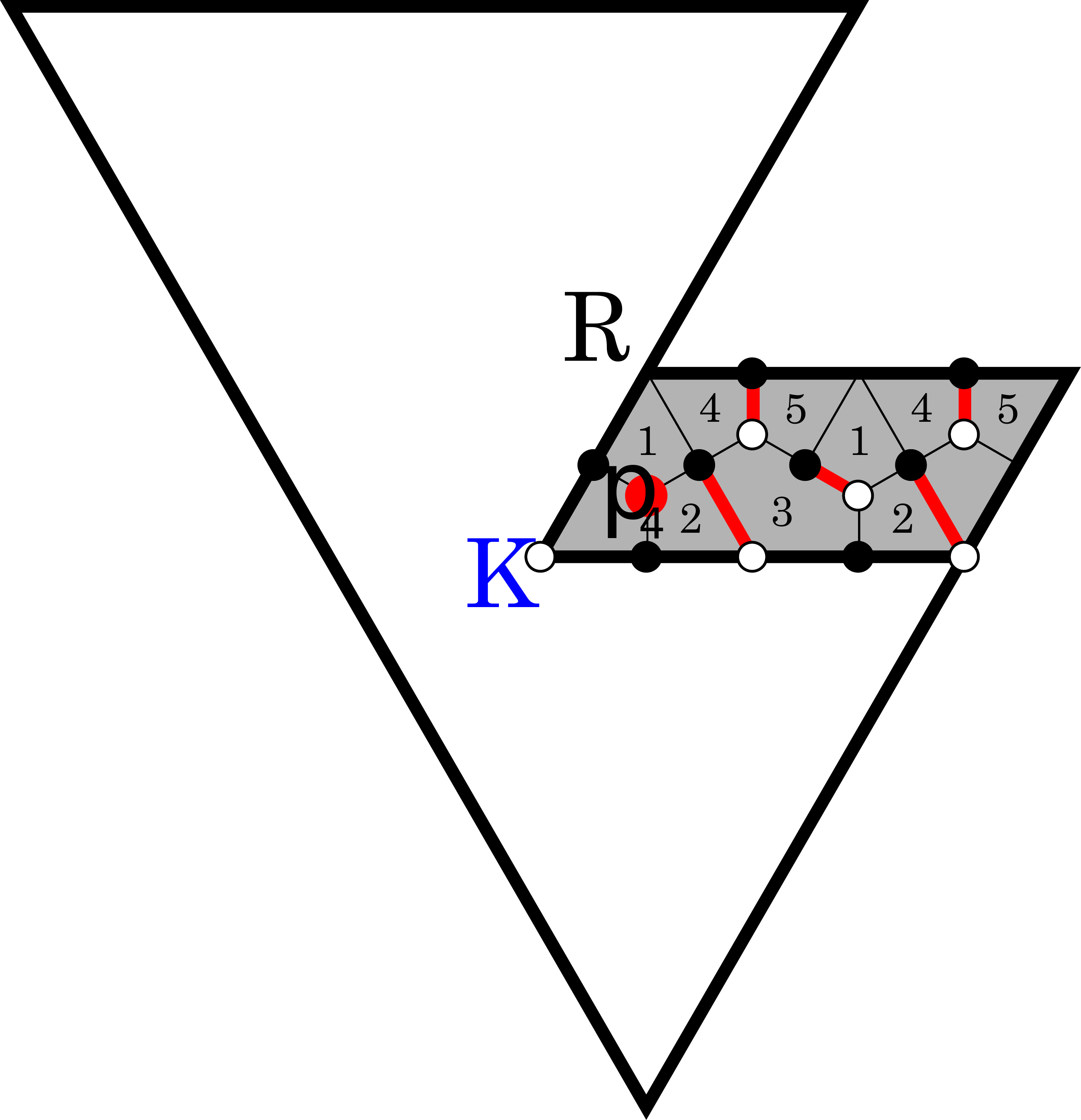}
\caption{The effects of removing $p_4$ in shape $(+,-,+,-,+)$ and $(+,-,+,+,-)$.}
\label{fig:same_effect}
\end{figure}

\end{remark}

\begin{remark}
The subgraphs in dP$_2$ quiver can look significantly different from those in dP$_3$ quiver. When side $c$ is long, there are many forced edges, which results in different shapes. See Figure~\ref{fig:long_c}.

\begin{figure}[h!]
\includegraphics[scale=0.12]{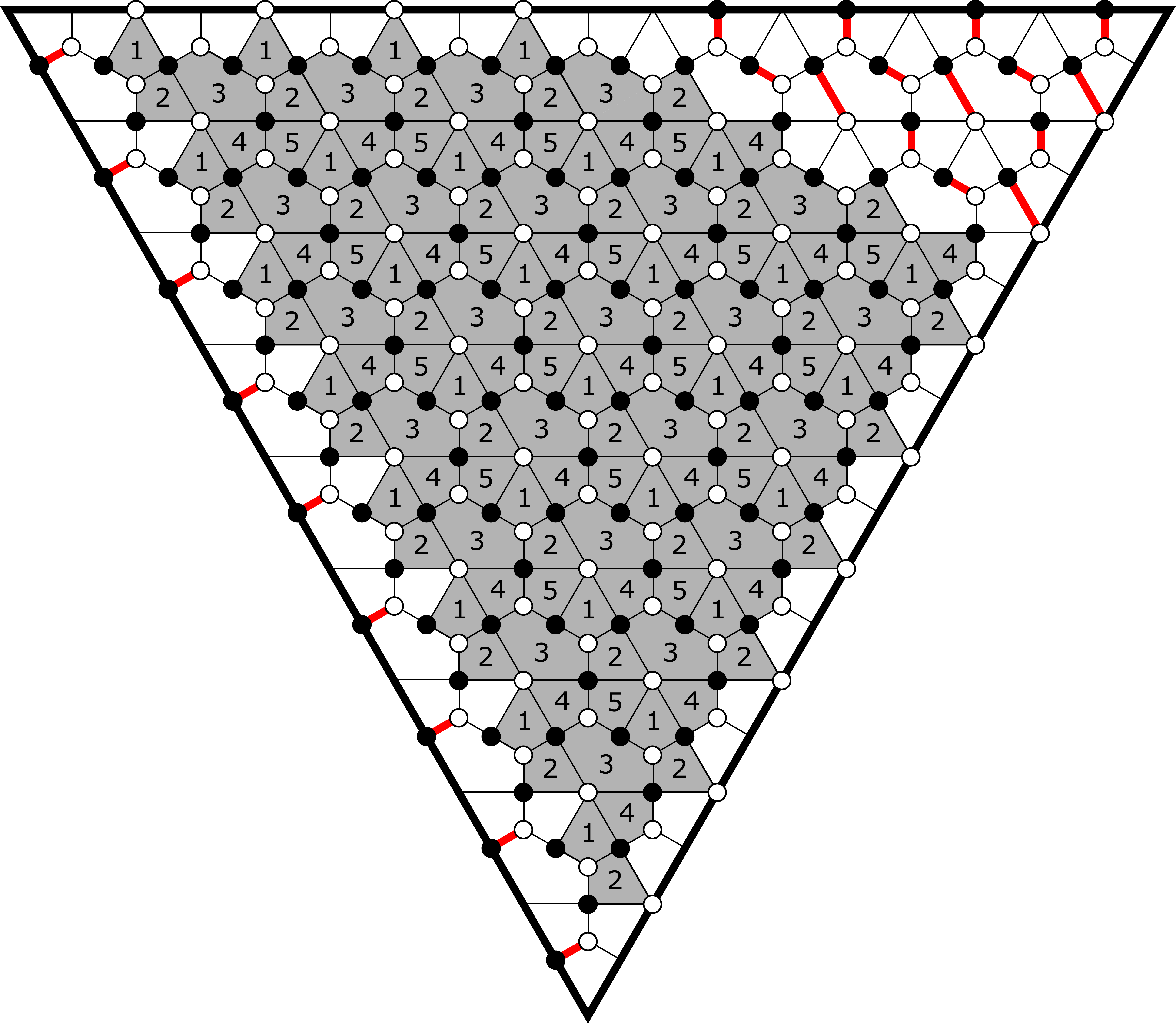}
\caption{Graph for $A^6B^4x_{17}$. Long edge $c$ results many in forced edges.}
\label{fig:long_c}
\end{figure}

\end{remark}

\section{Comparison with the Octahedron Recurrence}\label{sec:speyer}
David Speyer has given another combinatorial interpretation for the Laurent polynomials of the Somos-5 sequence in terms of the weight of some subgraphs of another brane tiling \cite{speyer2007perfect}. See Figure~\ref{fig:speyer} for the brane tiling and its corresponding quiver.

\begin{figure}[h!]
\includegraphics[scale=0.3]{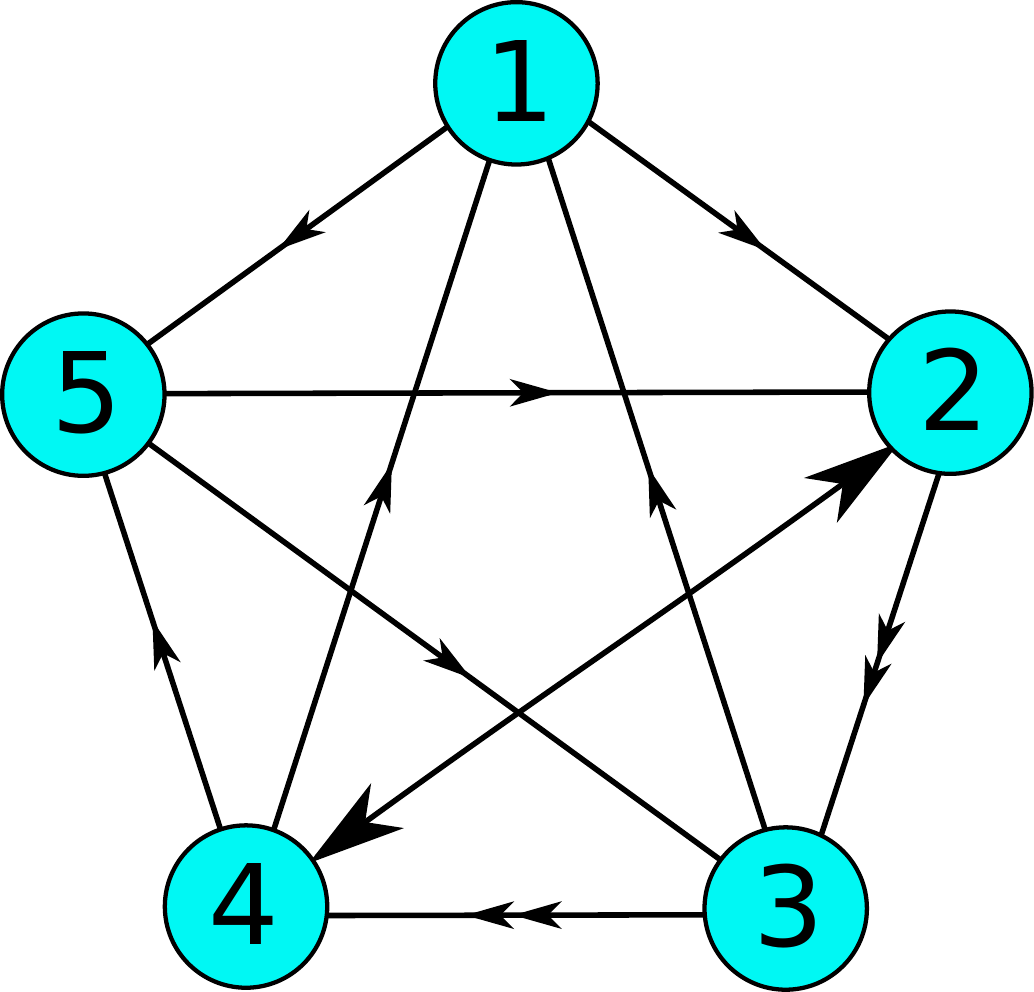}
\qquad
\includegraphics[scale=0.25]{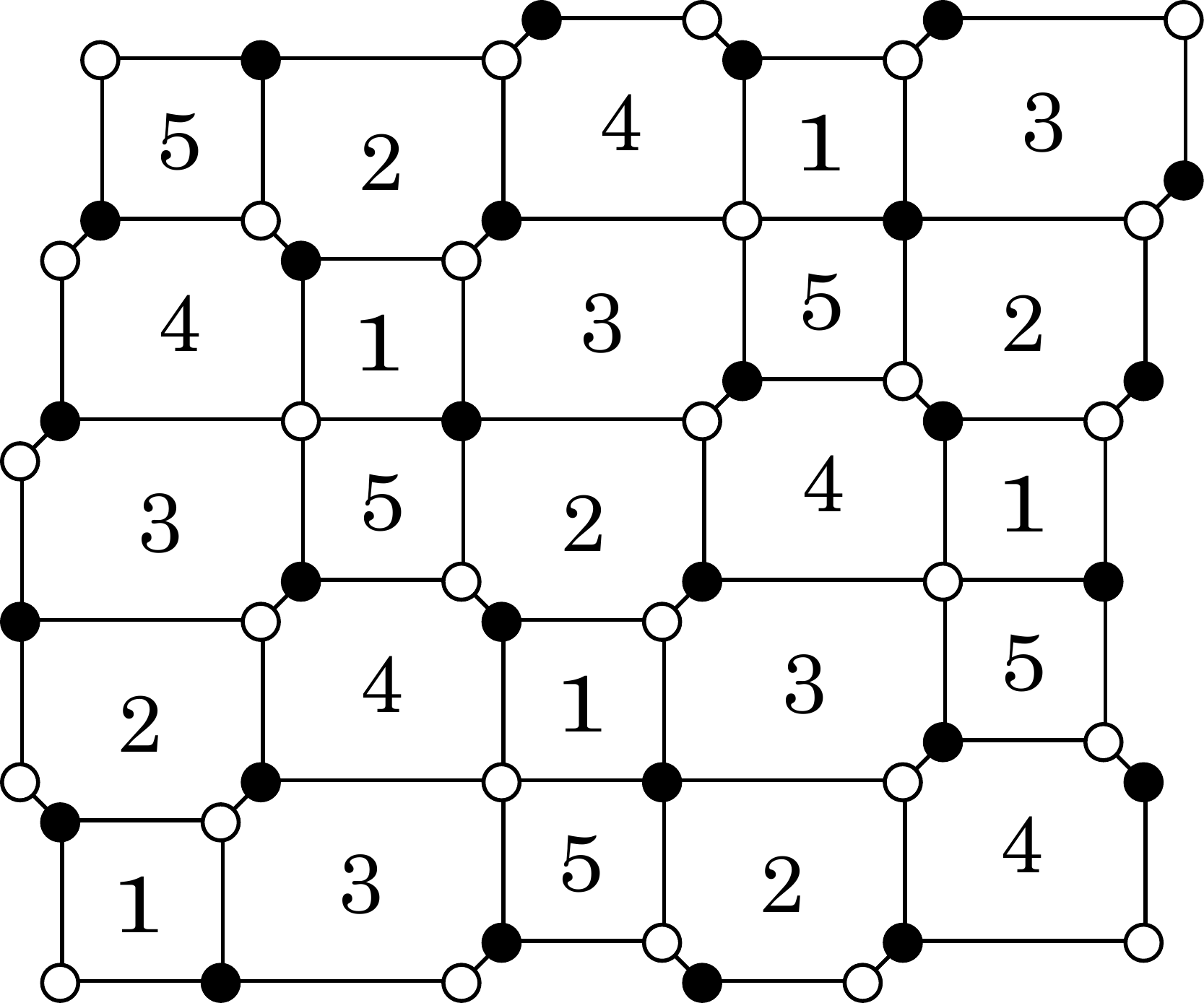}
\caption{The quiver and the brane tiling studied in \cite{speyer2007perfect}}
\label{fig:speyer}
\end{figure}

Notice that if we add a 2-cycle between vertex 2 and vertex 4 in our dP$_2$ quiver, we will obtain the quiver shown in Figure~\ref{fig:speyer}. However, it is hard to describe the transformation of these two brane tilings in a simple way. 

\begin{figure}[h!]
\includegraphics[scale=0.2]{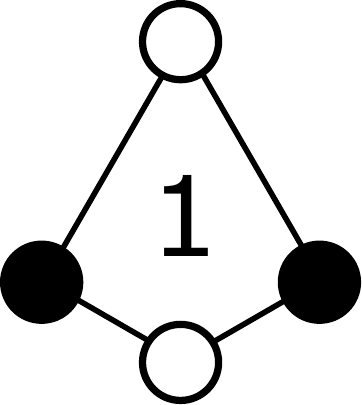}
\includegraphics[scale=0.2]{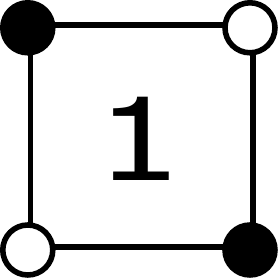}
\qquad
\includegraphics[scale=0.2]{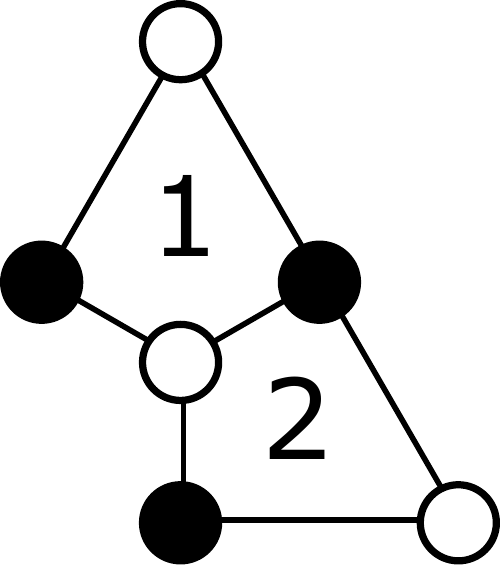}
\includegraphics[scale=0.2]{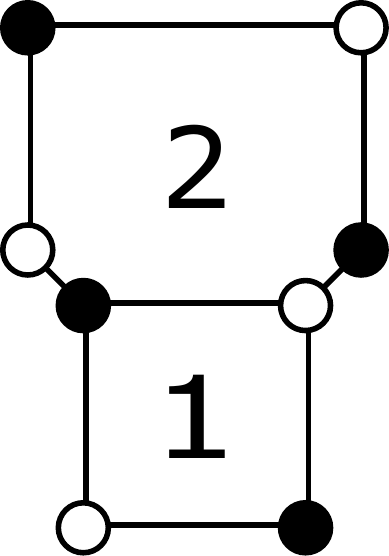}
\qquad
\includegraphics[scale=0.2]{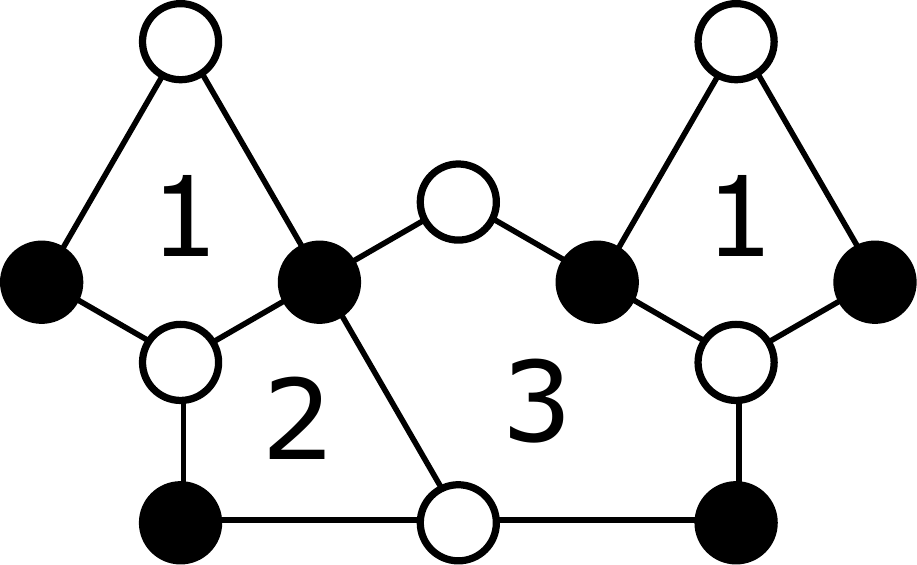}
\includegraphics[scale=0.2]{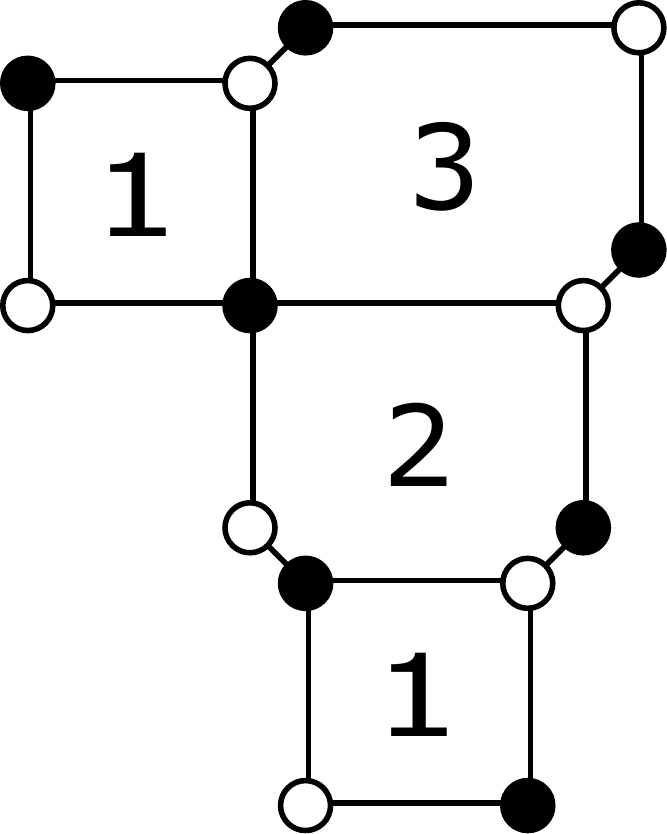}
\qquad
\includegraphics[scale=0.2]{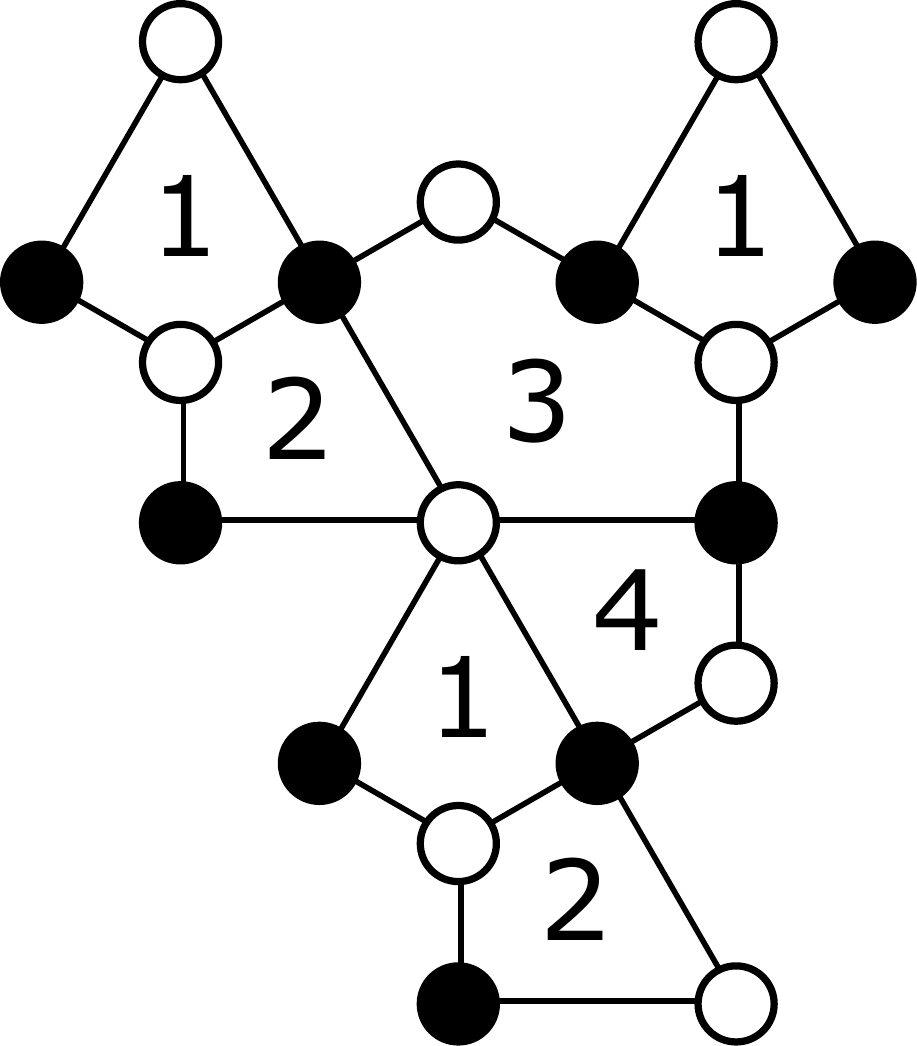}
\includegraphics[scale=0.2]{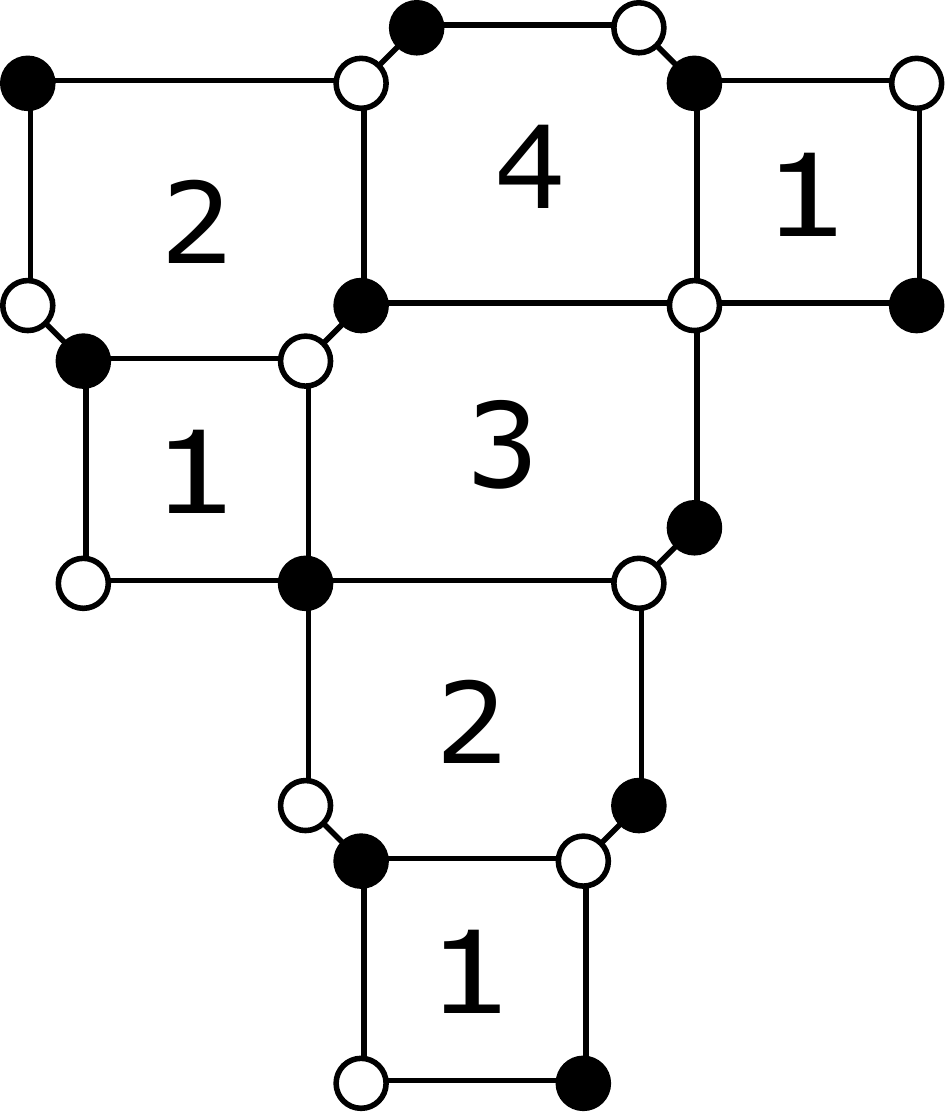}
\caption{subgraphs corresponding to terms $x_6,x_7,x_8,x_{9}$ in two different tilings}
\label{fig:speyer6-9}
\end{figure}

We provide a few terms of the Laurent polynomial of the Somos-5 sequence written as subgraphs of these two different brane tilings in Figure~\ref{fig:speyer6-9} and Figure~\ref{fig:speyer10-12}. As we can see, the blocks in each pair of subgraphs are similar but not exactly the same. Moreover, the subgraphs corresponding to $x_n$ in the dP$_2$ brane tiling are growing in two different directions (upper right and lower right) depending on the parity of $n$. But subgraphs in the tiling considered by Speyer have a growing pattern that seems to be unrelated to the parity of $n$. Therefore, we believe that these two problems regarding the two different tilings are sufficient different. There must exist some bijection between these subgraphs as we know how to generate them given $x_n$. Also, for each pair of subgraphs, there must exist some bijection between their perfect matchings. But as these two tilings are very different despite the similarity in the corresponding quivers, we leave such a bijection as an open question for future research.

\begin{figure}[h!]
\centering
\includegraphics[scale=0.2]{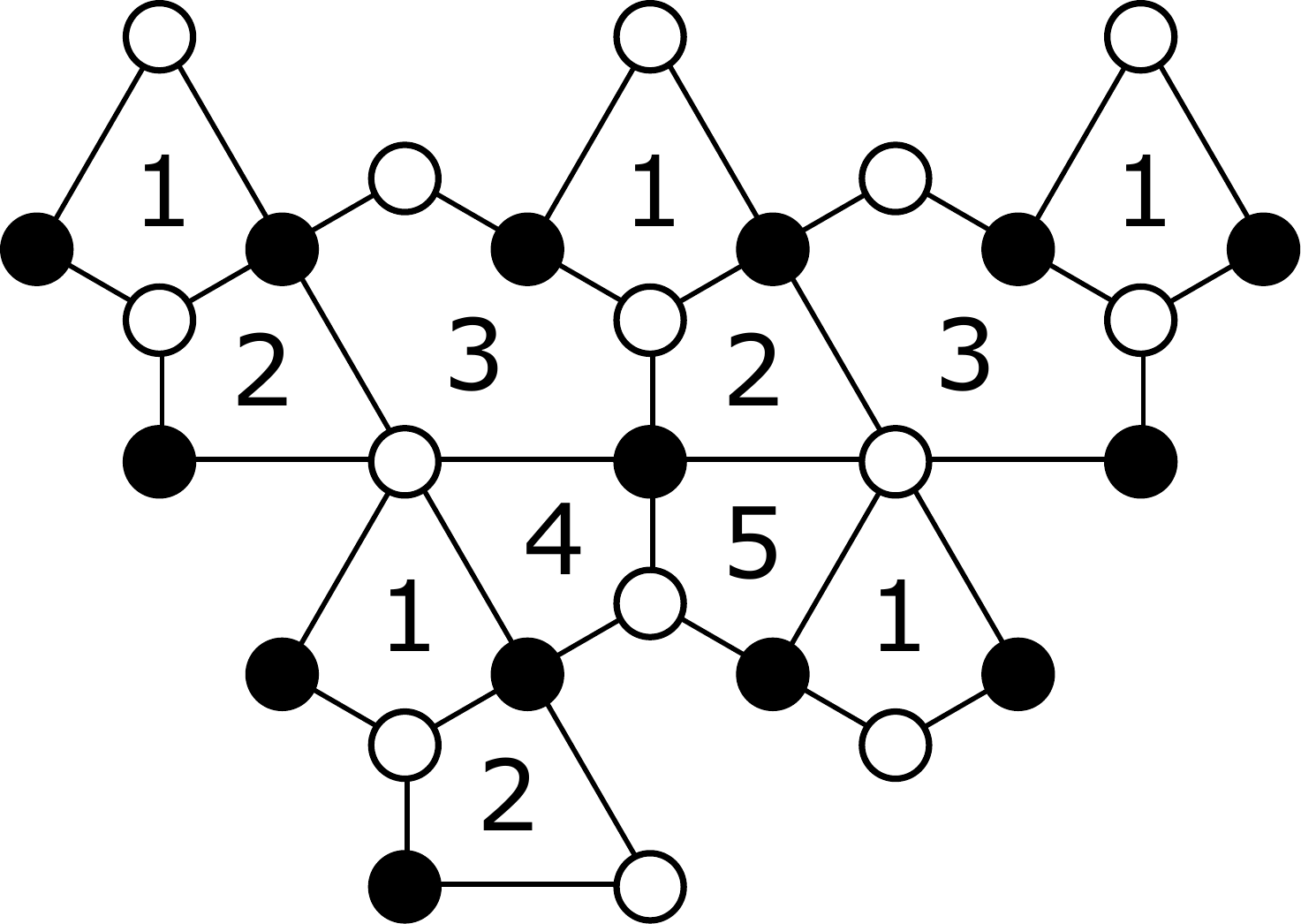}
\includegraphics[scale=0.2]{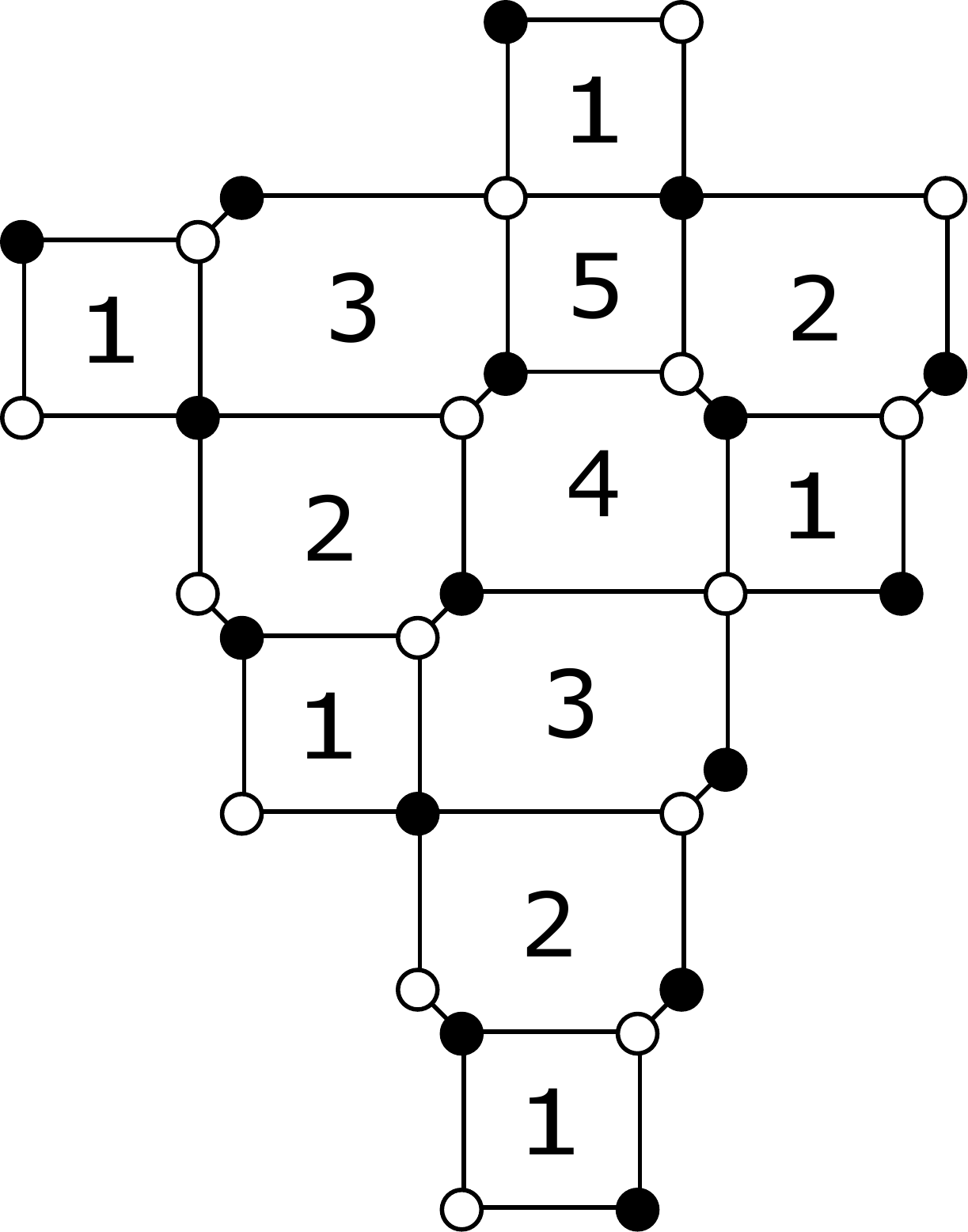}
\qquad
\includegraphics[scale=0.2]{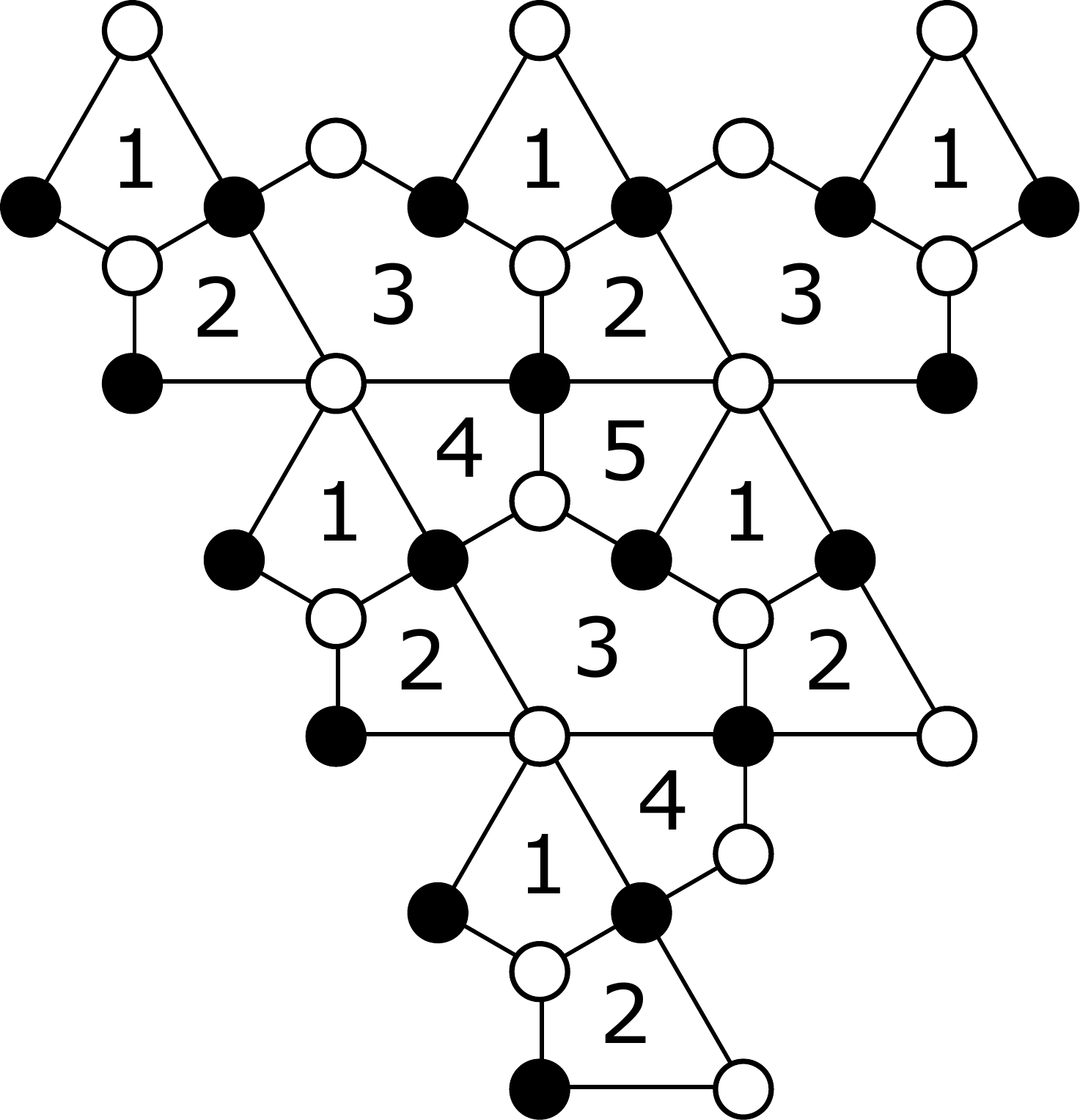}
\includegraphics[scale=0.2]{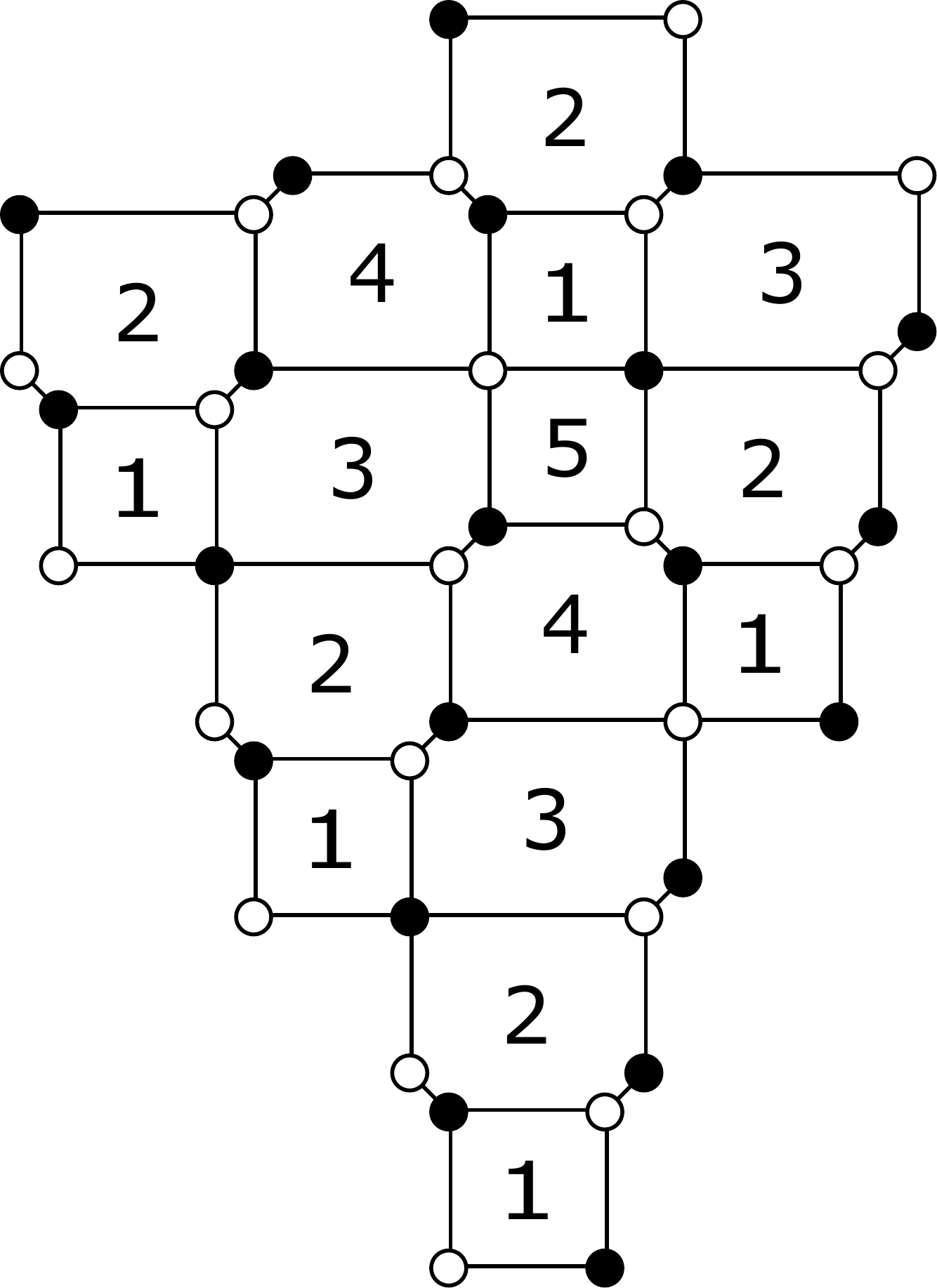}

\

\includegraphics[scale=0.2]{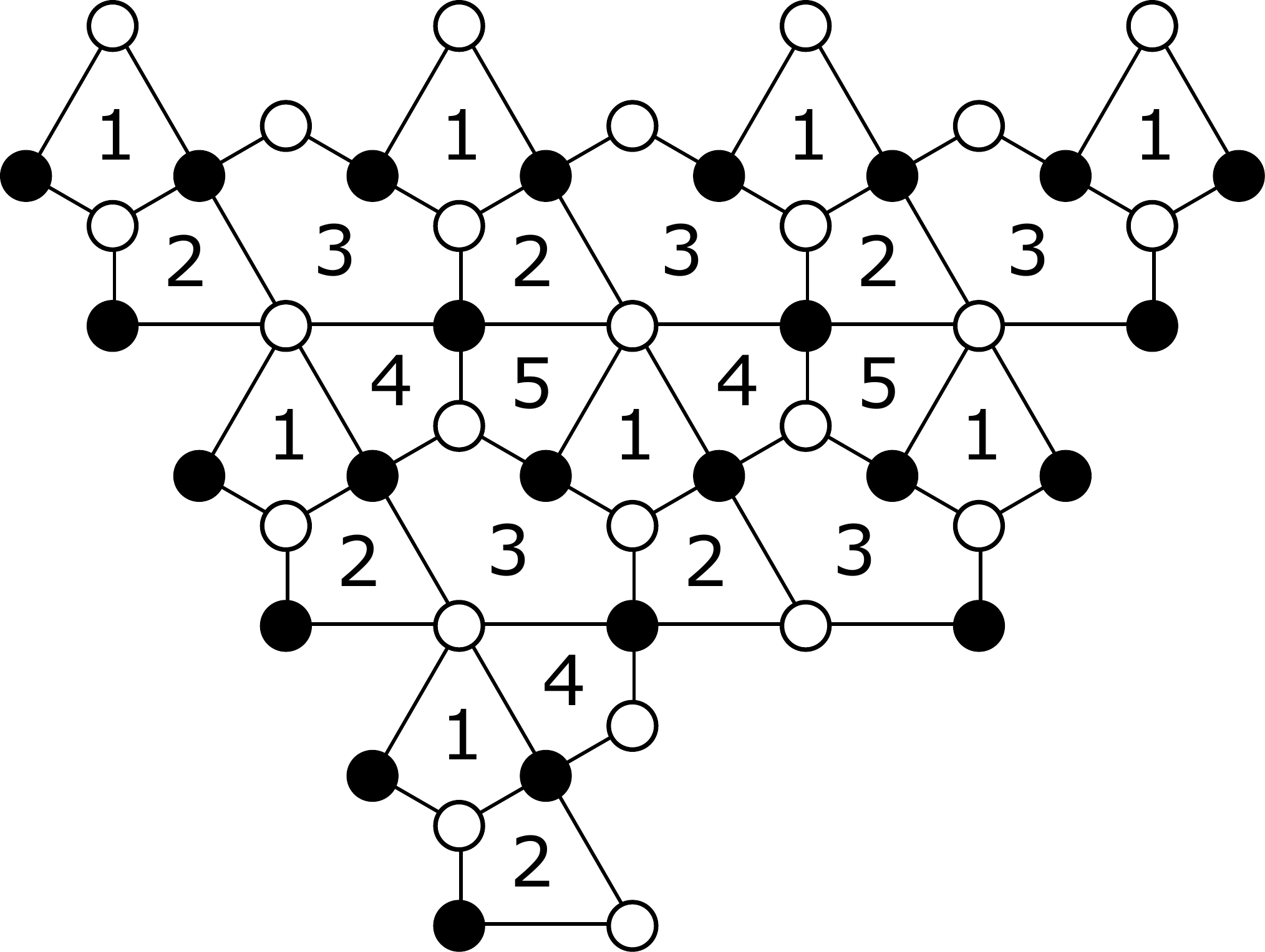}
\includegraphics[scale=0.2]{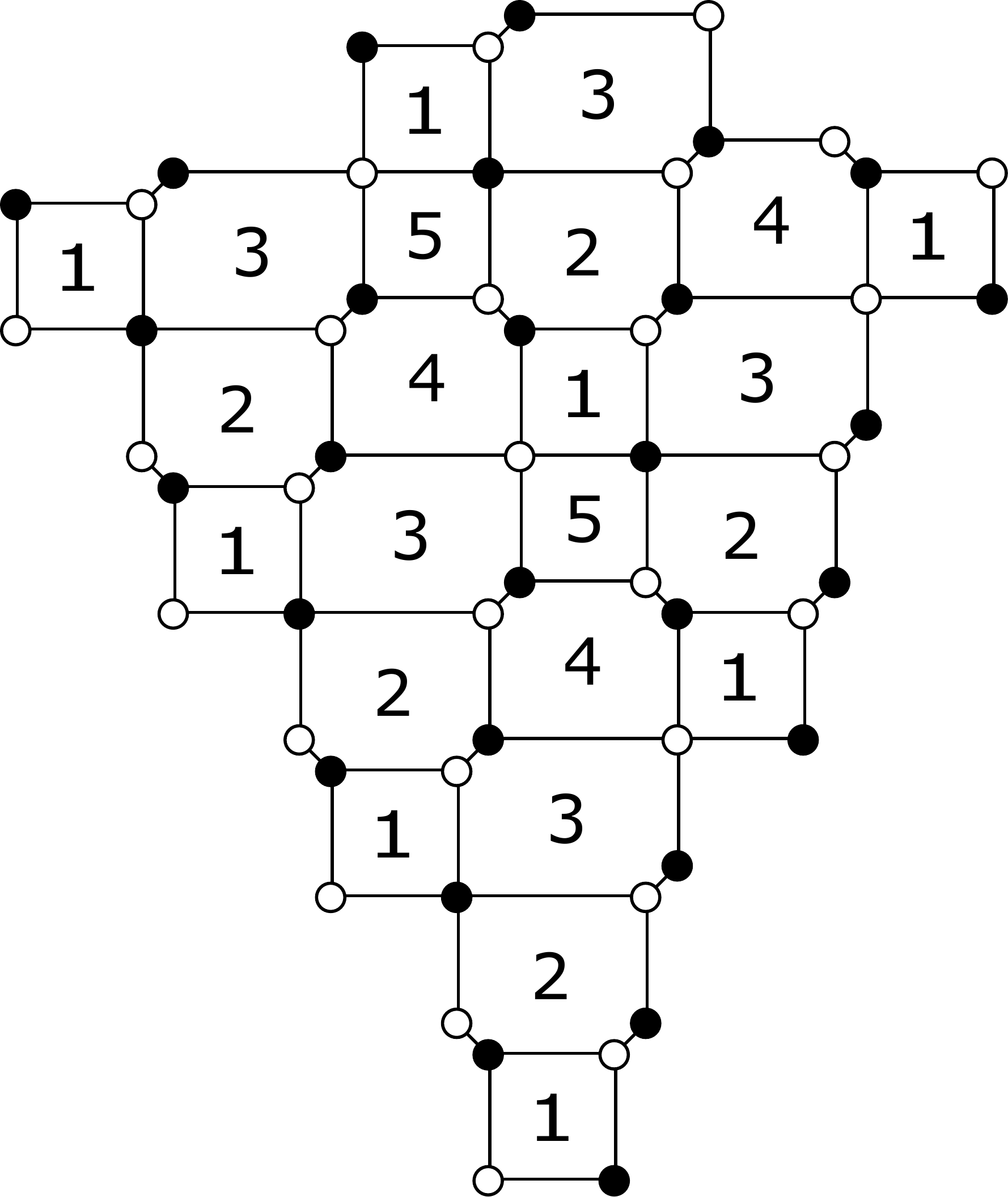}
\caption{subgraphs corresponding to terms $x_{10},x_{11},x_{12}$ in two different tilings}
\label{fig:speyer10-12}
\end{figure}

\section*{Acknowledgements}
This research was carried out as part of the 2016 REU program at the School of Mathematics at University of Minnesota, Twin Cities, and was supported by NSF RTG grant DMS-1148634 and by NSF grant DMS-1351590.  The authors would like to thank Sunita Chepuri and Elise DelMas for their comments and suggestions.  The authors are especially grateful to Gregg Musiker for his mentorship, support, and valuable advice.

\bibliographystyle{alpha}
\bibliography{main}
\setlength{\parindent}{0cm}

\newpage

\section{Appendix}
\label{sec:appendix}

Here are the data for other cases of Theorem~\ref{thm:contours}. We group these cases by the form of cluster variables. The way to use the appendix is shown in Section~\ref{sub:overview} and Section~\ref{sub:technique}.

\subsection{$A^{n^2}B^{n(n-1)}x_{2k}$, $n\geq1$, $k\geq2$}
\label{sec:caselisting1}
Recurrence we use: 
\begin{align*}
&(A^{n^2}B^{n(n-1)}x_{2k})(A^{(n-1)^2}B^{(n-1)(n-2)}x_{2k+2}) \\
=&(A^{(n-1)n}B^{(n-1)^2}x_{2k-1})(A^{(n-1)n}B^{(n-1)^2}x_{2k+3}) + (A^{(n-1)n}B^{(n-1)^2}x_{2k+1})^2
\end{align*}

Kuo's four points: $p_1,p_2,p_3,p_4$ are on edge $d,e,b,c$ respectively. 
\subsubsection{Case 1}
$k \geq 3n-1.$ Non-alternating Kuo. Shape $(+,-,+,+,-)$.

$$G := \hatcal{G}\left(k-1+n, -\left\lceil \frac{k+5n-5}{2} \right\rceil, 2n-2, \left\lfloor \frac{k-3n+3}{2} \right\rfloor, n-k-1\right)$$
\begin{align*}
-\{p_1\} (W): &K\rightarrow(0,-1,+1,-1,+1)-R\\
		      &R\rightarrow(0,0,+1,0,+1)-K\\
-\{p_2\} (B): &\rightarrow (-1,0,0,-1,+1)\\
-\{p_3\} (B): &\rightarrow (-1,+1,-1,0,0)\\
-\{p_4\} (W): &K\rightarrow(0,0,0,0,0)-R\\
			  &R\rightarrow(0,+1,0,+1,0)-K
\end{align*}

\begin{align*}
G =& A^{(n-1)n} B^{(n-1)^2} x_{2k+3} \\
G-\{p_1, p_2, p_3, p_4\} =& A^{(n-1)n} B^{(n-1)^2} x_{2k-1} \\
G-\{p_1, p_3\} =& A^{(n-1)n} B^{(n-1)^2} x_{2k+1} \\
G-\{p_2, p_4\} =& A^{(n-1)n} B^{(n-1)^2} x_{2k+1} \\
G-\{p_1, p_2\} =& A^{n^2} B^{n^2-n} x_{2k} \\
G-\{p_3, p_4\} =& A^{(n-1)^2} B^{(n-1)(n-2)} x_{2k+2}
\end{align*}

$G-K$ (Special vertex kept): let $p_1$ be the topmost (W) point on edge d, $p_2$ be the leftmost (B) point on edge e, $p_3$ be the bottommost (B) point on edge b, $p_4$ be the special vertex.
\begin{align*}
T(\emptyset) = 1, \quad T(\{p_1,p_2,p_3,p_4\}) = x_3x_3x_4x_4, \quad T(\{p_1,p_3\}) = x_3x_4, \\
T(\{p_2,p_4\}) = x_3x_4, \quad T(\{p_1,p_2\}) = x_4x_4, \quad T(\{p_3,p_4\}) = x_3x_3.
\end{align*}

$G-R$ (Special vertex removed): let $p_1$ be the topmost (W) point on edge d, $p_2$ be the leftmost (B) point on edge e, $p_3$ be the bottommost (B) point on edge b, $p_4$ be the (W) vertex bordering the 1-block below the special vertex.
\begin{align*}
T(\emptyset) = 1, \quad T(\{p_1,p_2,p_3,p_4\}) = x_1x_3x_4x_4, \quad T(\{p_1,p_3\}) = x_3x_4, \\
T(\{p_2,p_4\}) = x_1x_4, \quad T(\{p_1,p_2\}) = x_4x_4, \quad T(\{p_3,p_4\}) = x_1x_3.
\end{align*}

\subsubsection{Case 2}
$n+1 \leq k \leq 3n-2$. Unbalanced Kuo. Shape $(+,-,+,-,-)$.

When $(k+2)+(n-1)$ is odd, let $$G := \hatcal{G}\left(k-1+n, -\left\lceil \frac{k+5n-5}{2} \right\rceil, 2n-2+1, \left\lfloor \frac{k-3n+3}{2} \right\rfloor, n-k-1+1\right)-K$$

When $(k+2)+(n-1)$ is even, let $$G := \hatcal{G}\left(k-1+n, -\left\lceil \frac{k+5n-5}{2} \right\rceil-1, 2n-2+1, \left\lfloor \frac{k-3n+3}{2} \right\rfloor-1, n-k-1+1\right)-R$$

\begin{align*}
-\{p_1\} (B): &K\rightarrow(0,0,-1,0,-1)-R\\
		      &R\rightarrow(0,+1,-1,+1,-1)-K\\
-\{p_2\} (B): &\rightarrow (-1,0,0,-1,+1)\\
-\{p_3\} (B): &\rightarrow (-1,+1,-1,0,0)\\
-\{p_4\} (W): &K\rightarrow(0,0,0,0,0)-R\\
			  &R\rightarrow(0,+1,0,+1,0)-K
\end{align*}

\begin{align*}
G-\{p_1\} =& A^{(n-1)n} B^{(n-1)^2} x_{2k+3} \\
G-\{p_2, p_3, p_4\} =& A^{(n-1)n} B^{(n-1)^2} x_{2k-1} \\
G-\{p_3\} =& A^{(n-1)n} B^{(n-1)^2} x_{2k+1} \\
G-\{p_1, p_2, p_4\} =& A^{(n-1)n} B^{(n-1)^2} x_{2k+1} \\
G-\{p_2\} =& A^{n^2} B^{n^2-n} x_{2k} \\
G-\{p_1, p_3, p_4\} =& A^{(n-1)^2} B^{(n-1)(n-2)} x_{2k+2}
\end{align*}

$G-K$ (Special vertex kept): let $p_1$ be the bottommost (B) point on edge d, $p_2$ be the leftmost (B) point on edge e, $p_3$ be the bottommost (B) point on edge b, $p_4$ be the special vertex.
\begin{align*}
T(\{p_1\}) = x_3, \quad T(\{p_2,p_3,p_4\}) = x_3x_3x_4, \quad T(\{p_3\}) = x_3, \\
T(\{p_1,p_2,p_4\}) = x_3x_3x_4, \quad T(\{p_2\}) = x_4, \quad T(\{p_1,p_3,p_4\}) = x_3x_3x_3.
\end{align*}

$G-R$ (Special vertex removed): let $p_1$ be the bottommost (B) point on edge d, $p_2$ be the leftmost (B) point on edge e, $p_3$ be the bottommost (B) point on edge b, $p_4$ be the (W) vertex bordering the 1-block below the special vertex.
\begin{align*}
T(\{p_1\}) = x_4, \quad T(\{p_2,p_3,p_4\}) = x_1x_3x_4, \quad T(\{p_3\}) = x_3x_3, \\
T(\{p_1,p_2,p_4\}) = x_3x_3x_3x_4, \quad T(\{p_2\}) = x_4, \quad T(\{p_1,p_3,p_4\}) = x_1x_3x_3.
\end{align*}

\subsubsection{Case 3}
$k \leq n$. Balanced Kuo. Shape $(+,-,+,-,+).$

When $(k+2)+(n-1)$ is odd, let $$G := \hatcal{G}\left(k-1+n-1, -\left\lceil \frac{k+5n-5}{2} \right\rceil, 2n-2+1, \left\lfloor \frac{k-3n+3}{2} \right\rfloor-1, n-k-1+2\right)-K$$

When $(k+2)+(n-1)$ is even, let $$G := \hatcal{G}\left(k-1+n-1, -\left\lceil \frac{k+5n-5}{2} \right\rceil-1, 2n-2+1, \left\lfloor \frac{k-3n+3}{2} \right\rfloor-2, n-k-1+2\right)-R$$

\begin{align*}
-\{p_1\} (B): &K\rightarrow(0,0,-1,0,-1)-R\\
		      &R\rightarrow(0,+1,-1,+1,-1)-K\\
-\{p_2\} (W): &\rightarrow (+1,0,0,+1,-1)\\
-\{p_3\} (B): &\rightarrow (-1,+1,-1,0,0)\\
-\{p_4\} (W): &K\rightarrow(0,0,0,0,0)-R\\
			  &R\rightarrow(0,+1,0,+1,0)-K
\end{align*}

\begin{align*}
G-\{p_1,p_2\} =& A^{(n-1)n} B^{(n-1)^2} x_{2k+3} \\
G-\{p_3, p_4\} =& A^{(n-1)n} B^{(n-1)^2} x_{2k-1} \\
G-\{p_2, p_3\} =& A^{(n-1)n} B^{(n-1)^2} x_{2k+1} \\
G-\{p_1, p_4\} =& A^{(n-1)n} B^{(n-1)^2} x_{2k+1} \\
G=& A^{n^2} B^{n^2-n} x_{2k} \\
G-\{p_1,p_2,p_3, p_4\} =& A^{(n-1)^2} B^{(n-1)(n-2)} x_{2k+2}
\end{align*}

$G-K$ (Special vertex kept): let $p_1$ be the second from top (B) point on edge d, $p_2$ be the second from left (W) point on edge e, $p_3$ be the second from bottom (B) point on edge b, $p_4$ be the special vertex.
\begin{align*}
T(\{p_1,p_2\}) = x_1x_3, \quad T(\{p_3,p_4\}) = x_3x_5, \quad T(\{p_2,p_3\}) = x_3x_5, \\
T(\{p_1,p_4\}) = x_1x_3, \quad T(\emptyset) = 1, \quad T(\{p_1,p_2,p_3,p_4\}) = x_1x_3x_3x_5.
\end{align*}

$G-R$ (Special vertex removed): let $p_1$ be the second from bottom (B) point on edge d, $p_2$ be the second from right (W) point on edge e, $p_3$ be the second from bottom (B) point on edge b, $p_4$ be the (W) vertex bordering the 1-block below the special vertex.
\begin{align*}
T(\{p_1,p_2\}) = x_1x_3, \quad T(\{p_3,p_4\}) = x_1x_5, \quad T(\{p_2,p_3\}) = x_3x_5, \\
T(\{p_1,p_4\}) = x_1x_1, \quad T(\emptyset) = 1, \quad T(\{p_1,p_2,p_3,p_4\}) = x_1x_1x_3x_5.
\end{align*}

\subsection{$A^{n^2+n}B^{n^2}x_{2k-1}$, $n\geq1$, $k\geq3$}
Recurrence we use: 
\begin{align*}
&(A^{n^2+n}B^{n^2}x_{2k-1})(A^{(n-1)^2+(n-1)}B^{(n-1)^2}x_{2k+1}) \\
=&(A^{n^2}B^{n(n-1)}x_{2k-2})(A^{n^2}B^{n(n-1)}x_{2k+2}) + (A^{n^2}B^{n(n-1)}x_{2k})^2
\end{align*}

Kuo's four points: $p_1,p_2,p_3,p_4$ are on edge $d,e,b,c$ respectively.

The effect of removing $p_1,p_2,p_3,p_4$ and the sets used in the proof of covering monomial are the same as in Section~\ref{sec:caselisting1}. 

\subsubsection{Case 1}
$k \geq 3n+1$. Non-alternating Kuo. Shape $(+,-,+,+,-)$.

$$G := \hatcal{G}\left(k-1+n, -\left\lceil \frac{k+5n-3}{2} \right\rceil, 2n-1, \left\lfloor \frac{k-3n+1}{2} \right\rfloor, n-k\right).$$
\begin{align*}
G =& A^{n^2} B^{n^2-n} x_{2k+2} \\
G-\{p_1,p_2, p_3, p_4\} =& A^{n^2} B^{n^2-n} x_{2k-2} \\
G-\{p_1,p_3\} =& A^{n^2} B^{n^2-n} x_{2k} \\
G-\{p_2, p_4\} =& A^{n^2} B^{n^2-n} x_{2k} \\
G-\{p_1,p_2\} =& A^{n^2+n} B^{n^2} x_{2k-1} \\
G-\{p_3, p_4\} =& A^{(n-1)^2+(n-1)} B^{(n-1)^2} x_{2k+1}
\end{align*}

\subsubsection{Case 2}
$n+2 \leq k \leq 3n$. Unbalanced Kuo. Shape $(+,-,+,-,-)$.

When $(k+1)+n$ is odd, let $$G := \hatcal{G}\left(k-1+n, -\left\lceil \frac{k+5n-3}{2} \right\rceil, 2n-1+1, \left\lfloor \frac{k-3n+1}{2} \right\rfloor, n-k+1\right)-K$$

When $(k+1)+n$ is even, let $$G := \hatcal{G}\left(k-1+n, -\left\lceil \frac{k+5n-3}{2} \right\rceil-1, 2n-1+1, \left\lfloor \frac{k-3n+1}{2} \right\rfloor-1, n-k-1+1\right)-R$$

\begin{align*}
G-\{p_1\} =& A^{n^2} B^{n^2-n} x_{2k+2} \\
G-\{p_2, p_3, p_4\} =& A^{n^2} B^{n^2-n} x_{2k-2} \\
G-\{p_3\} =& A^{n^2} B^{n^2-n} x_{2k} \\
G-\{p_1, p_2, p_4\} =& A^{n^2} B^{n^2-n} x_{2k} \\
G-\{p_2\} =& A^{n^2+n} B^{n^2} x_{2k-1} \\
G-\{p_1, p_3, p_4\} =& A^{(n-1)^2+(n-1)} B^{(n-1)^2} x_{2k+1}
\end{align*}

\subsubsection{Case 3}
$3\leq k \leq n+1$. Balanced Kuo. Shape $(+,-,+,-,+)$.

When $(k+1)+n$ is odd, let $$G := \hatcal{G}\left(k-1+n-1, -\left\lceil \frac{k+5n-3}{2} \right\rceil, 2n-1+1, \left\lfloor \frac{k-3n+1}{2} \right\rfloor-1, n-k+2\right)-K$$

When $(k+1)+n$ is even, let $$G := \hatcal{G}\left(k-1+n-1, -\left\lceil \frac{k+5n-3}{2} \right\rceil-1, 2n-1+1, \left\lfloor \frac{k-3n+1}{2} \right\rfloor-2, n-k+2\right)-R$$

\begin{align*}
G-\{p_1,p_2\} =& A^{n^2} B^{n^2-n} x_{2k+2} \\
G-\{p_3, p_4\} =& A^{n^2} B^{n^2-n} x_{2k-2} \\
G-\{p_2, p_3\} =& A^{n^2} B^{n^2-n} x_{2k} \\
G-\{p_1, p_4\} =& A^{n^2} B^{n^2-n} x_{2k} \\
G=& A^{n^2+n} B^{n^2} x_{2k-1} \\
G-\{p_1,p_2,p_3, p_4\} =& A^{(n-1)^2+(n-1)} B^{(n-1)^2} x_{2k+1}
\end{align*}

\subsection{$A^{n^2}B^{n(n-1)}x_{2k}$, $n\leq-1$, $k\geq2$}
\label{sec:caselisting3}
Recurrence we use: 
\begin{align*}
&(A^{n^2}B^{n(n-1)}x_{2k})(A^{(n+1)^2}B^{n(n+1)}x_{2k+2}) \\
=&(A^{n(n+1)}B^{n^2}x_{2k-1})(A^{n(n+1)}B^{n^2}x_{2k+3}) + (A^{n(n+1)}B^{n^2}x_{2k+1})^2
\end{align*}

Kuo's four points: $p_1,p_2,p_3,p_4$ are on edge $d,a,b,c$ respectively.

\subsubsection{Case 1}
$k \geq 1-5n.$ Non-alternating Kuo. Shape $(+,-,-,+,-)$.

$$G := \hatcal{G}\left(k+n, -\left\lceil \frac{k+5n}{2} \right\rceil, 2n, \left\lfloor \frac{k-3n}{2} \right\rfloor, n-k\right)$$

\begin{align*}
-\{p_1\} (W): &K\rightarrow(0,-1,+1,-1,+1)-R\\
		      &R\rightarrow(0,0,+1,0,+1)-K\\
-\{p_2\} (W): &\rightarrow (-1,+1,0,0,+1)\\
-\{p_3\} (B): &\rightarrow (-1,+1,-1,0,0)\\
-\{p_4\} (B): &K\rightarrow(0,-1,0,-1,0)-R\\
			  &R\rightarrow(0,0,0,0,0)-K
\end{align*}

\begin{align*}
G =& A^{n^2+n} B^{n^2} x_{2k+3} \\
G-\{p_1, p_2, p_3, p_4\} =& A^{n^2+n} B^{n^2} x_{2k-1} \\
G-\{p_1, p_3\} =& A^{n^2+n} B^{n^2} x_{2k+1} \\
G-\{p_2, p_4\} =& A^{n^2+n} B^{n^2} x_{2k+1} \\
G-\{p_1, p_4\} =& A^{(n+1)^2} B^{(n+1)^2-(n+1)} x_{2k+2} \\
G-\{p_2, p_3\} =& A^{n^2} B^{n^2-n} x_{2k}
\end{align*}

$G-K$ (Special vertex kept): let $p_1$ be the topmost (W) point on edge d, $p_2$ be the topmost (W) point on edge a (not in a forced matching), $p_3$ be the bottommost (B) point on edge b (not in a forced matching), $p_4$ be the (B) point on the edge between the 4-block and 5-block above the special vertex.
\begin{align*}
T(\emptyset) = 1, \quad T(\{p_1,p_2,p_3,p_4\}) = x_4x_4x_5x_5, \quad T(\{p_1,p_3\}) = x_4x_5, \\
T(\{p_2,p_4\}) = x_4x_5, \quad T(\{p_1,p_4\}) = x_4x_4, \quad T(\{p_2,p_3\}) = x_5x_5.
\end{align*}

$G-R$ (Special vertex removed): let $p_1$ be the topmost (W) point on edge d, $p_2$ be the topmost (W) point on edge a (not in a forced matching), $p_3$ be the bottommost (B) point on edge b (in a forced matching), $p_4$ be the (B) point on the edge between the 2-block and 3-block above the special vertex.
\begin{align*}
T(\emptyset) = 1, \quad T(\{p_1,p_2,p_3,p_4\}) = x_2x_3x_4x_5, \quad T(\{p_1,p_3\}) = x_3x_4, \\
T(\{p_2,p_4\}) = x_2x_5, \quad T(\{p_1,p_4\}) = x_2x_4, \quad T(\{p_2,p_3\}) = x_3x_5.
\end{align*}

\subsubsection{Case 2}
$2-n \leq k \leq -5n.$ Unbalanced Kuo. Shape $(+,+,-,+,-)$.

$$G := \hatcal{G}\left(k+n-1, -\left\lceil \frac{k+5n}{2} \right\rceil+1, 2n-1, \left\lfloor \frac{k-3n}{2} \right\rfloor, n-k\right)$$

\begin{align*}
-\{p_1\} (W): &K\rightarrow(0,-1,+1,-1,+1)-R\\
		      &R\rightarrow(0,0,+1,0,+1)-K\\
-\{p_2\} (W): &\rightarrow (-1,+1,0,0,+1)\\
-\{p_3\} (W): &\rightarrow (+1,-1,+1,0,0)\\
-\{p_4\} (B): &K\rightarrow(0,-1,0,-1,0)-R\\
			  &R\rightarrow(0,0,0,0,0)-K
\end{align*}

\begin{align*}
G-\{p_3\} =& A^{n^2+n} B^{n^2} x_{2k+3} \\
G-\{p_1, p_2, p_4\} =& A^{n^2+n} B^{n^2} x_{2k-1} \\
G-\{p_1\} =& A^{n^2+n} B^{n^2} x_{2k+1} \\
G-\{p_2, p_3, p_4\} =& A^{n^2+n} B^{n^2} x_{2k+1} \\
G-\{p_2\} =& A^{n^2} B^{n^2-n} x_{2k} \\
G-\{p_1, p_3, p_4\} =& A^{(n+1)^2} B^{(n+1)^2-(n+1)} x_{2k+2}
\end{align*}

$G-K$ (Special vertex kept): let $p_1$ be the topmost (W) point on edge d, $p_2$ be the topmost (W) point on edge a, $p_3$ be the bottommost (W) point on edge b, $p_4$ be the (B) point on the edge between the 4-block and 5-block above the special vertex.
\begin{align*}
T(\{p_3\}) = 1, \quad T(\{p_1,p_2,p_4\}) = x_4x_4x_5x_5, \quad T(\{p_1\}) = x_4x_5, \\
T(\{p_2,p_3,p_4\}) = x_4x_5, \quad T(\{p_2\}) = x_4x_4, \quad T(\{p_1,p_3,p_4\}) = x_5x_5.
\end{align*}

$G-R$ (Special vertex removed): let $p_1$ be the topmost (W) point on edge d, $p_2$ be the topmost (W) point on edge a, $p_3$ be the bottommost (B) point on edge b, $p_4$ be the (B) point on the edge between the 2-block and 3-block above the special vertex.
\begin{align*}
T(\{p_3\}) = x_2, \quad T(\{p_1,p_2,p_4\}) = x_4x_4x_5, \quad T(\{p_1\}) = x_4, \\
T(\{p_2,p_3,p_4\}) = x_2x_4x_5, \quad T(\{p_2\}) = x_5, \quad T(\{p_1,p_3,p_4\}) = x_2x_4x_4.
\end{align*}

\subsubsection{Case 3}
$2\leq k\leq 1-n.$ Balanced Kuo. Shape $(-,+,-,+,-)$.

$$G := \hatcal{G}\left(k+n-2, -\left\lceil \frac{k+5n}{2} \right\rceil+2, 2n-1, \left\lfloor \frac{k-3n}{2} \right\rfloor, n-k+1\right)$$

\begin{align*}
-\{p_1\} (W): &K\rightarrow(0,-1,+1,-1,+1)-R\\
		      &R\rightarrow(0,0,+1,0,+1)-K\\
-\{p_2\} (B): &\rightarrow (+1,-1,0,0,-1)\\
-\{p_3\} (W): &\rightarrow (+1,-1,+1,0,0)\\
-\{p_4\} (B): &K\rightarrow(0,-1,0,-1,0)-R\\
			  &R\rightarrow(0,0,0,0,0)-K
\end{align*}

\begin{align*}
G-\{p_2, p_3\} =& A^{n^2+n} B^{n^2} x_{2k+3} \\
G-\{p_1, p_4\} =& A^{n^2+n} B^{n^2} x_{2k-1} \\
G-\{p_1, p_2\} =& A^{n^2+n} B^{n^2} x_{2k+1} \\
G-\{p_3, p_4\} =& A^{n^2+n} B^{n^2} x_{2k+1} \\
G-\{p_1, p_2, p_3, p_4\} =& A^{(n+1)^2} B^{(n+1)^2-(n+1)} x_{2k+2} \\
G =& A^{n^2} B^{n^2-n} x_{2k}
\end{align*}

$G-K$ (Special vertex kept): let $p_1$ be the topmost (W) point on edge d, $p_2$ be the bottommost (B) point on edge a, $p_3$ be the topmost (W) point on edge b (not in a forced matching), $p_4$ be the (B) point on the edge between the 4-block and 5-block above the special vertex.
\begin{align*}
T(\{p_2,p_3\}) = x_2x_2, \quad T(\{p_1,p_4\}) = x_4x_4, \quad T(\{p_1,p_2\}) = x_2x_4, \\
T(\{p_3,p_4\}) = x_2x_4, \quad T(\{p_1,p_2,p_3,p_4\}) = x_2x_2x_4x_4, \quad T(\emptyset) = 1.
\end{align*}

$G-R$ (Special vertex removed): let $p_1$ be the topmost (W) point on edge d, $p_2$ be the bottommost (B) point on edge a, $p_3$ be the topmost (W) point on edge b (not in a forced matching), $p_4$ be the (B) point on the edge between the 2-block and 3-block above the special vertex.
\begin{align*}
T(\{p_2,p_3\}) = x_2x_2, \quad T(\{p_1,p_4\}) = x_2x_4, \quad T(\{p_1,p_2\}) = x_2x_4, \\
T(\{p_3,p_4\}) = x_2x_2, \quad T(\{p_1,p_2,p_3,p_4\}) = x_2x_2x_2x_4, \quad T(\emptyset) = 1.
\end{align*}

\subsection{$A^{n^2+n}B^{n^2}x_{2k-1}$, $n\leq-1$, $k\geq3$}

Recurrence we use: 
\begin{align*}
&(A^{n^2+n}B^{n^2}x_{2k-1})(A^{(n+1)(n+2)}B^{(n+1)^2}x_{2k+1}) \\
=&(A^{(n+1)^2}B^{(n+1)n}x_{2k-2})(A^{(n+1)^2}B^{(n+1)n}x_{2k+2}) + (A^{(n+1)^2}B^{(n+1)n}x_{2k})^2
\end{align*}

Kuo's four points: $p_1,p_2,p_3,p_4$ are on edge $d,a,b,c$ respectively.

The effect of removing $p_1,p_2,p_3,p_4$ and the sets used in the proof of covering monomial are the same as in Section~\ref{sec:caselisting3}. 

\subsubsection{Case 1}
$k \geq -1-5n.$ Non-alternating Kuo. Shape $(+,-,-,+,-)$.

$$G := \hatcal{G}\left(k+n, -\left\lceil \frac{k+5n+2}{2} \right\rceil, 2(n+1)-1, \left\lfloor \frac{k-3n-2}{2} \right\rfloor, 1+n-k\right)$$

\begin{align*}
G =& A^{(n+1)^2} B^{(n+1)n} x_{2k+2} \\
G-\{p_1, p_2, p_3, p_4\} =& A^{(n+1)^2} B^{(n+1)n} x_{2k-2} \\
G-\{p_1, p_3\} =& A^{(n+1)^2} B^{(n+1)n} x_{2k} \\
G-\{p_2, p_4\} =& A^{(n+1)^2} B^{(n+1)n} x_{2k} \\
G-\{p_1, p_4\} =& A^{(n+1)^2+(n+1)} B^{(n+1)^2} x_{2k+1} \\
G-\{p_2, p_3\} =& A^{n^2+n} B^{n^2} x_{2k-1}
\end{align*}

\subsubsection{Case 2}
$2-n\leq k \leq -2-5n.$ Unbalanced Kuo. Shape $(+,+,-,+,-)$.

$$G := \hatcal{G}\left(k+n-1, -\left\lceil \frac{k+5n+2}{2} \right\rceil+1, 2(n+1)-1-1, \left\lfloor \frac{k-3n-2}{2} \right\rfloor, 1+n-k\right)$$

\begin{align*}
G-\{p_3\} =& A^{(n+1)^2} B^{(n+1)n} x_{2k+2} \\
G-\{p_1, p_2, p_4\} =& A^{(n+1)^2} B^{(n+1)n} x_{2k-2} \\
G-\{p_1\} =& A^{(n+1)^2} B^{(n+1)n} x_{2k} \\
G-\{p_2, p_3, p_4\} =& A^{(n+1)^2} B^{(n+1)n} x_{2k} \\
G-\{p_1, p_3, p_4\} =& A^{(n+1)^2+(n+1)} B^{(n+1)^2} x_{2k+1} \\
G-\{p_2\} =& A^{n^2+n} B^{n^2} x_{2k-1}
\end{align*}

\subsubsection{Case 3}
$2\leq k \leq 1-n.$ Balanced Kuo. Shape $(-,+,-,+,-)$.

$$G := \hatcal{G}\left(k+n, -\left\lceil \frac{k+5n+2}{2} \right\rceil, 2(n+1)-1, \left\lfloor \frac{k-3n-2}{2} \right\rfloor, 1+n-k\right)$$

\begin{align*}
G-\{p_2,p_3\} =& A^{(n+1)^2} B^{(n+1)n} x_{2k+2} \\
G-\{p_1, p_4\} =& A^{(n+1)^2} B^{(n+1)n} x_{2k-2} \\
G-\{p_1, p_3\} =& A^{(n+1)^2} B^{(n+1)n} x_{2k} \\
G-\{p_2, p_4\} =& A^{(n+1)^2} B^{(n+1)n} x_{2k} \\
G-\{p_1, p_2,p_3,p_4\} =& A^{(n+1)^2+(n+1)} B^{(n+1)^2} x_{2k+1} \\
G =& A^{n^2+n} B^{n^2} x_{2k-1}
\end{align*}

\subsection{$A^{n^2+n}B^{n^2}x_3$, $n\geq1$}

Recurrence we use: 
\begin{align*}
&(A^{n(n+1)}B^{n^2}x_3)(A^{n^2}B^{n(n-1)}x_8) \\
=&(A^{n(n+1)}B^{n^2}x_5)(A^{n^2}B^{n(n-1)}x_6) + (A^{n(n+1)}B^{n^2}x_7)(A^{n^2}B^{n(n-1)}x_4).
\end{align*}

Kuo's four points: $p_1,p_2,p_3,p_4$ are on edge $e,a,c,d$ respectively.

When $n =1$: can verify the contour match the graph using Balanced Kuo or just directly verify the matching polynomial.

Let $n \geq 2$. Unbalanced Kuo. Shape $(+,-,+,-,+)$.

When $n+3$ is odd, let $$G := \hatcal{G}\left(3-2+n, -\left\lceil \frac{3-2+5n}{2} \right\rceil, 2n, \left\lfloor \frac{3-3n-2}{2} \right\rfloor, n-3+2\right)-R$$

When $n+3$ is even, let $$G := \hatcal{G}\left(3-2+n, -\left\lceil \frac{3-2+5n}{2} \right\rceil-1, 2n, \left\lfloor \frac{3-3n-2}{2} \right\rfloor-1, n-3+2\right)-K$$

\begin{align*}
-\{p_1\} (W): &\rightarrow(+1,0,0,+1,-1)\\
-\{p_2\} (W): &\rightarrow (-1,+1,0,0,+1)\\
-\{p_3\} (W): &K\rightarrow (0,0,0,0,0)-R\\
			&R\rightarrow (0,+1,0,+1,0)-K\\
-\{p_4\} (B): &K\rightarrow(0,0,-1,0,-1)-R\\
			  &R\rightarrow(0,+1,-1,+1,-1)-K
\end{align*}

\begin{align*}
G-\{p_1\} =& A^{n^2+n} B^{n^2} x_7 \\
G-\{p_2, p_3, p_4\} =& A^{n^2} B^{n(n-1)} x_4 \\
G-\{p_3\} =& A^{n^2+n} B^{n^2} x_{5} \\
G-\{p_1, p_2, p_4\} =& A^{n^2} B^{n(n-1)} x_6 \\
G-\{p_2\} =& A^{n^2+n} B^{n^2} x_{3} \\
G-\{p_1, p_3, p_4\} =& A^{n^2} B^{n(n-1)} x_{8}
\end{align*}

$G-K$ (Special vertex kept): let $p_1$ be the leftmost (W) point on edge e (the bottommost point of edge d), $p_2$ be the bottommost (W) point on edge a (not in a forced matching), $p_3$ be the special vertex, $p_4$ be the (B) point on the edge between the 1-block and 4-block below the special vertex.
\begin{align*}
T(\{p_1\}) = x_1, \quad T(\{p_2,p_3,p_4\}) = x_1x_3x_5, \quad T(\{p_3\}) = x_3, \\
T(\{p_1,p_2,p_4\}) = x_1x_1x_5, \quad T(\{p_2\}) = x_5, \quad T(\{p_1,p_3,p_4\}) = x_1x_1x_3.
\end{align*}

$G-R$ (Special vertex removed): let $p_1$ be the leftmost (W) point on edge e (the bottommost point of edge d), $p_2$ be the bottommost (W) point on edge a (not in a forced matching), $p_3$ be the (W) point below the special vertex, $p_4$ be the (B) point below $p_3$.
\begin{align*}
T(\{p_1\}) = x_1, \quad T(\{p_2,p_3,p_4\}) = x_1x_3x_5, \quad T(\{p_3\}) = x_1, \\
T(\{p_1,p_2,p_4\}) = x_1x_3x_5, \quad T(\{p_2\}) = x_5, \quad T(\{p_1,p_3,p_4\}) = x_1x_1x_3.
\end{align*}

\subsection{$A^{n^2+n}B^{n^2}x_3$, $n\leq-1$}

Recurrence we use: 
\begin{align*}
&(A^{n(n+1)}B^{n^2}x_3)(A^{(n+1)^2}B^{n(n+1)}x_8) \\
=&(A^{n(n+1)}B^{n^2}x_5)(A^{(n+1)^2}B^{n(n+1)}x_6) + (A^{n(n+1)}B^{n^2}x_7)(A^{(n+1)^2}B^{n(n+1)}x_4)
\end{align*}

Kuo's four points: $p_1,p_2,p_3,p_4$ are on edge $a,e,c,b$ respectively.

When $n=-1$, can check directly to see contour for $Bx_3$ is correct.

Let $n \leq -2$. Unbalanced Kuo. Shape $(-,+,-,+,-)$.

When $n+3$ is odd, let $$G := \hatcal{G}\left(3-2+n, -\left\lceil \frac{3-2+5n}{2} \right\rceil+1, 2n, \left\lfloor \frac{3-3n-2}{2} \right\rfloor+1, n-3+2\right)-K$$

When $n+3$ is even, let $$G := \hatcal{G}\left(3-2+n, -\left\lceil \frac{3-2+5n}{2} \right\rceil-1, 2n, \left\lfloor \frac{3-3n-2}{2} \right\rfloor-1, n-3+2\right)-R$$

\begin{align*}
-\{p_1\} (B): &\rightarrow(+1,-1,0,0,-1)\\
-\{p_2\} (B): &\rightarrow (-1,+1,0,0,+1)\\
-\{p_3\} (B): &K\rightarrow (0,-1,0,-1,0)-R\\
			&R\rightarrow (0,0,0,0,0)-K\\
-\{p_4\} (W): &\rightarrow(+1,-1,+1,0,0)
\end{align*}

\begin{align*}
G-\{p_1\} =& A^{n^2+n} B^{n^2} x_7 \\
G-\{p_2, p_3, p_4\} =& A^{n^2} B^{n(n-1)} x_4 \\
G-\{p_3\} =& A^{n^2+n} B^{n^2} x_{5} \\
G-\{p_1, p_2, p_4\} =& A^{n^2} B^{n(n-1)} x_6 \\
G-\{p_2\} =& A^{n^2+n} B^{n^2} x_{3} \\
G-\{p_1, p_3, p_4\} =& A^{n^2} B^{n(n-1)} x_{8}
\end{align*}

$G-K$ (Special vertex kept): let $p_1$ be the topmost (B) point on edge a (in a forced matching), $p_2$ be the rightmost (B) point on edge e (in a forced matching), $p_3$ be the (B) point with 3 neighbors on the 3-block above the special vertex, $p_4$ be the bottommost (W) point on edge b.
\begin{align*}
T(\{p_1\}) = x_1, \quad T(\{p_2,p_3,p_4\}) = x_2x_3x_5, \quad T(\{p_3\}) = x_5, \\
T(\{p_1,p_2,p_4\}) = x_1x_2x_3, \quad T(\{p_2\}) = x_3, \quad T(\{p_1,p_3,p_4\}) = x_1x_2x_5.
\end{align*}

$G-R$ (Special vertex removed): let $p_1$ be the topmost (B) point on edge a (in a forced matching), $p_2$ be the rightmost (B) point on edge e (in a forced matching), $p_3$ be the (B) point on the edge between the 2-block and 3-block above the special vertex, $p_4$ be the bottommost (W) point on edge b.
\begin{align*}
T(\{p_1\}) = x_1, \quad T(\{p_2,p_3,p_4\}) = x_2x_2x_5, \quad T(\{p_3\}) = x_2, \\
T(\{p_1,p_2,p_4\}) = x_1x_2x_5, \quad T(\{p_2\}) = x_5, \quad T(\{p_1,p_3,p_4\}) = x_1x_2x_2.
\end{align*}

\end{document}